\documentclass[10pt]{amsart}
\usepackage{pdfsync}
\usepackage{amssymb}
\usepackage{amscd}

\def\today{\number\day\space\ifcase\month\or   January\or February\or
   March\or April\or May\or June\or   July\or August\or September\or
   October\or November\or December\fi\   \number\year}

\newcounter{TmpEnumi}

\theoremstyle{definition}
\newtheorem{thm}{Theorem}[section]
\newtheorem{lem}[thm]{Lemma}
\newtheorem{prp}[thm]{Proposition}
\newtheorem{dfn}[thm]{Definition}
\newtheorem{cor}[thm]{Corollary}

\newtheorem{rmk}[thm]{Remark}
\newtheorem{ntn}[thm]{Notation}
\newtheorem{exa}[thm]{Example}

\newtheorem{qst}[thm]{Question}

\newcommand{\beq}{\begin{equation}}
\newcommand{\eeq}{\end{equation}}
\newcommand{\beqr}{\begin{eqnarray*}}
\newcommand{\eeqr}{\end{eqnarray*}}
\newcommand{\bal}{\begin{align*}}
\newcommand{\eal}{\end{align*}}
\newcommand{\bei}{\begin{itemize}}
\newcommand{\eei}{\end{itemize}}
\newcommand{\limi}[1]{\lim_{{#1} \to \infty}}

\newcommand{\af}{\alpha}
\newcommand{\bt}{\beta}
\newcommand{\gm}{\gamma}
\newcommand{\dt}{\delta}
\newcommand{\ep}{\varepsilon}
\newcommand{\zt}{\zeta}
\newcommand{\et}{\eta}

\newcommand{\io}{\iota}

\newcommand{\ld}{\lambda}
\newcommand{\sm}{\sigma}
\newcommand{\kp}{\kappa}
\newcommand{\ph}{\varphi}
\newcommand{\ps}{\psi}
\newcommand{\rh}{\rho}
\newcommand{\om}{\omega}
\newcommand{\ta}{\tau}

\newcommand{\Gm}{\Gamma}
\newcommand{\Dt}{\Delta}

\newcommand{\Th}{\Theta}
\newcommand{\Ld}{\Lambda}

\newcommand{\Ph}{\Phi}
\newcommand{\Ps}{\Psi}
\newcommand{\Om}{\Omega}

\newcommand{\Q}{{\mathbb{Q}}}
\newcommand{\Z}{{\mathbb{Z}}}
\newcommand{\R}{{\mathbb{R}}}
\newcommand{\C}{{\mathbb{C}}}
\newcommand{\N}{{\mathbb{Z}}_{> 0}}
\newcommand{\Nz}{{\mathbb{Z}}_{\geq 0}}

\newcommand{\id}{{\mathrm{id}}}

\newcommand{\sint}{{\mathrm{int}}}

\newcommand{\Prim}{{\mathrm{Prim}}}

\newcommand{\supp}{{\mathrm{supp}}}

\newcommand{\spn}{{\mathrm{span}}}
\newcommand{\card}{{\mathrm{card}}}
\newcommand{\Aut}{{\mathrm{Aut}}}
\newcommand{\Ad}{{\mathrm{Ad}}}

\newcommand{\dirlim}{\varinjlim}

\newcommand{\OT}{{\mathcal{O}_2}}

\newcommand{\Open}{{\mathbb{O}}}
\newcommand{\M}{M}

\newcommand{\andeqn}{\,\,\,\,\,\, {\mbox{and}} \,\,\,\,\,\,}

\newcommand{\Wolog}{Without loss of generality}
\newcommand{\Tfae}{The following are equivalent}
\newcommand{\tfae}{the following are equivalent}
\newcommand{\ifo}{if and only if}

\newcommand{\ca}{C*-algebra}
\newcommand{\uca}{unital C*-algebra}

\newcommand{\hm}{homomorphism}

\newcommand{\fd}{finite dimensional}
\newcommand{\tst}{tracial state}

\newcommand{\pj}{projection}

\newcommand{\nzp}{nonzero projection}
\newcommand{\mvnt}{Murray-von Neumann equivalent}

\newcommand{\ct}{continuous}
\newcommand{\cfn}{continuous function}
\newcommand{\nbhd}{neighborhood}

\newcommand{\chs}{compact Hausdorff space}
\newcommand{\hme}{homeomorphism}
\newcommand{\mh}{minimal homeomorphism}

\newcommand{\cp}{crossed product}

\newcommand{\dr}{pseudocovering}

\renewcommand{\S}{\subset}

\newcommand{\SM}{\setminus}
\newcommand{\I}{\infty}
\newcommand{\E}{\varnothing}

\title[Minimal dynamics on prime C*-algebras]{Minimal
  dynamical systems on prime C*-algebras}

\author{Eberhard Kirchberg}

\address{Institut f\"{u}r Mathematik,
 Humboldt-Universit\"{a}t zu Berlin,
 Unter den Linden~6,
 DE-10099, Germany.}

\author{N.~Christopher Phillips}

\date{6 December 2023}

\address{Department of Mathematics, University  of Oregon,
       Eugene OR 97403-1222, USA.}

\email[]{ncp@darkwing.uoregon.edu}

\thanks{This material is based upon work of the second author
supported by the US National Science Foundation under Grants
DMS-0701076 and DMS-1101742
and by a visiting professorship at the Research Institute
for Mathematical Sciences at Kyoto University,
and on work of both authors supported
by the Danish National Research Foundation
through the Centre for Symmetry and Deformation.}

\begin{document}

\begin{abstract}
We give a number of examples of exotic actions of locally compact groups
on separable nuclear C*-algebras.
In particular, we give examples of the following:
\begin{itemize}
\item
Minimal effective actions of $\Z$ and $F_n$
on unital nonsimple prime AF algebras.
\item
For any second countable noncompact locally compact group,
a minimal effective action
on a separable nuclear nonsimple prime C*-algebra.
\item
For any amenable second countable noncompact locally compact group,
a minimal effective action
on a separable nuclear nonsimple prime C*-algebra
(unital when the group is ${\mathbb{Z}}$ or ${\mathbb{R}}$)
such that the crossed product
is $K \otimes {\mathcal{O}}_2$
(${\mathcal{O}}_2$ when the group is ${\mathbb{Z}}$).
\item
For any second countable
locally compact abelian group~$G$ which is not discrete,
an action on $K \otimes {\mathcal{O}}_2$
such that the crossed product is a nonsimple prime \ca.
\end{itemize}
In most of these situations,
we can specify the primitive ideal space of the C*-algebra
(of the crossed product in the last item)
within a class of spaces.
\end{abstract}

\maketitle

\indent
The purpose of this paper is to give examples of several kinds
of exotic actions of locally compact groups on \ca{s}.
In particular, we give examples of minimal actions
(there are no nontrivial invariant ideals;
see Definition~\ref{D:MinActCSt})
of various groups, including $\R$ and~$\Z$,
on \ca{s} which are prime but not simple.
In some cases, Takai duality then gives actions of abelian
groups on $K \otimes {\mathcal{O}}_2$
or ${\mathcal{O}}_2$ such that the crossed products
are prime but not simple.
The original motivation was questions related to~\cite{It}.
In particular,
our minimal actions of~$\R$ on prime nonsimple \ca{s}
give examples for Corollary~3.6 of~\cite{It}.

In more detail,
we do the following.
For every locally compact noncompact Hausdorff space~$P$,
we define a non Hausdorff compactification $\Xi (P)$,
in such a way that any separable \ca{}
whose primitive ideal space is $\Xi (P)$
must be prime.
(See Definition~\ref{D:X}.)
Using this construction, we give the following kinds of examples.
(In some cases, the actual examples are more general than
described here,
or have additional properties.)
\begin{enumerate}
\item\label{Summ_2X25_ZAF}
For any homeomorphism $h$ of a totally disconnected
nonempty second countable locally compact Hausdorff space~$P$
which has a not necessarily finite invariant Borel measure
with full support
but such that $h$ has no nonempty compact invariant subsets,
we give a minimal automorphism $\af$
of a unital nonsimple prime AF algebra~$A$
such that $\Prim (A) \cong \Xi (P)$
in such a way that the \hme{} induced by~$\af$
comes from~$h$.
(This is a special case of Corollary~\ref{C-FnAct508}.)
In particular (Corollary~\ref{C-ZAF428}) we may take $P = \Z$
with the translation action.
\item\label{Summ_2X25_FnAF}
For $n \geq 2$, we give a minimal action of $F_n$
on a unital nonsimple prime AF algebra
whose restriction to every abelian subgroup is not minimal.
(See Example~\ref{E-FnMin508}.)
\item\label{Summ_2X25_AllGp}
For any second countable noncompact locally compact group~$G$
and every second countable locally compact Hausdorff $G$-space $P$
with no nonempty compact $G$-invariant subsets,
we give a minimal action $\gm$
on a separable nuclear nonsimple prime C*-algebra~$A$
such that $\Prim (A) \cong \Xi (P)$
with the action induced by~$\gm$ being the same as the
one coming from the given action on~$P$.
(See Theorem~\ref{T:ExistMinAct}.)
In particular, we can take $P = G$ with translation.
\item\label{Summ_2X25_CPO2}
If $G$ in~(\ref{Summ_2X25_AllGp}) is amenable,
we can arrange that $C^* (G, C, \gm) \cong K \otimes \OT$.
(See Theorem~\ref{T-CPO2-430}.)
\item\label{Summ_2X25_UnitR}
If $G = \R$ in~(\ref{Summ_2X25_AllGp}),
we can take $C$ to be unital.
(See Proposition~\ref{P-RUnital}.)
\item\label{Summ_2X25_Star}
For any second countable
locally compact abelian group~$G$ which is not discrete
and any second countable
nonempty locally compact Hausdorff ${\widehat{G}}$-space $P$
with no nonempty compact ${\widehat{G}}$-invariant subsets,
we give an action $\af \colon G \to \Aut (K \otimes \OT)$
such that $C^* (G, \, K \otimes \OT, \, \af)$
is a nonsimple prime \ca{} whose primitive ideal space
is homeomorphic to $\Xi (P)$.
(See Corollary~\ref{C:ExoticCptAb}.)
\end{enumerate}

It is not possible to produce examples as in~(\ref{Summ_2X25_AllGp})
with compact groups (Corollary~\ref{C:CptGMinA})
or as in~(\ref{Summ_2X25_Star}) with discrete abelian groups
(Remark~\ref{R_2X27No}).
There are stronger restrictions if the algebra has the ideal property
(Corollary~\ref{C:ConnNotRR0} and Corollary~\ref{C:PrimeNotRR0}).

Some of our examples have similar properties to some examples
and results in~\cite{Brt}.
It is shown there that crossed products
by product type actions of connected locally compact abelian groups
are never simple, but they can easily be prime.
This can also happen for other locally compact abelian groups,
including some which are totally disconnected.
The dual actions on the crossed products are minimal.
Our AF~examples differ:
we get only actions of $\Z$ and $F_n$
(but \cite{Brt} gives nothing for $F_n$),
our algebras are unital
(not possible with the constructions in~\cite{Brt}),
and we construct examples with predetermined
primitive ideal spaces within a suitable class.
See the introduction to Section~\ref{Sec:AF}
for a more careful comparison of the results.

The basic method for most of the examples is to specify
a suitable primitive ideal space~$X$
(or the form $\Xi (P)$ as above)
for the \ca{} to be constructed,
with an action of the group $G$ on~$X$,
use the results of~\cite{HK} to obtain a separable nuclear \ca~$A$
with primitive ideal space~$X$, and, with additional work,
show that $A$ can be chosen so that there is an action
of $G$ on~$A$ which induces the given action on~$X$.
Thus, the results are an application of the range result of~\cite{HK}
for primitive ideal spaces of separable nuclear \ca.

This paper is organized as follows.
Section~\ref{Sec:1} describes the
non Hausdorff compactification $\Xi (P)$,
gives its basic properties,
and discusses group actions on it.
The easy parts of the proof that $\Xi (P)$
is the primitive ideal space of a separable nuclear \ca{} are also here.
The most complicated step of this proof is in Section~\ref{Sec:1b}.

In Section~\ref{Sec:AF},
we prove the results (\ref{Summ_2X25_ZAF}) and~(\ref{Summ_2X25_FnAF})
above involving AF algebras,
and related results.

The short Section~\ref{Sec:2a} gives the nonexistence result
for minimal actions of compact groups on nonsimple prime \ca{s}.

Section~\ref{Sec:2b}
contains the construction of \ca{s} with actions of groups
which induce given actions
on the primitive ideal space when it has the form $\Xi (P)$.
This construction is used to produce
the examples in (\ref{Summ_2X25_AllGp}) and~(\ref{Summ_2X25_UnitR}),
as well as some related examples.
In Section~\ref{Sec:3},
we show how to make the crossed products simple,
giving~(\ref{Summ_2X25_CPO2}) and related results.
Examples as in~(\ref{Summ_2X25_Star})
are then obtained using Takai duality.

Our convention is that
compact and locally compact spaces are not required to be Hausdorff.
(Thus, spaces we call compact or locally compact
are the spaces sometimes called
quasi-compact and locally quasi-compact.)
However, compact and locally compact groups are assumed to be Hausdorff.

We are grateful to Benjam\'{\i}n Itz\'{a}-Ortiz
for raising the questions
addressed here,
to Lew Ward for pointing out the Menger-N\"{o}beling Theorem,
to Ken Goodearl for the reference to Theorem~10.9 of~\cite{Gd},
and to George Willis (via Ken Ross)
for the example at the end of Section~\ref{Sec:1}.

\section{A non-Hausdorff compactification}\label{Sec:1}

\indent
We want to construct \ca{s} by constructing their
primitive ideal spaces.
In this section,
we recall some relevant definitions and results from~\cite{HK}.
We then define a non-Hausdorff compactification
of a locally compact space,
and prove basic facts about it,
such as functoriality.
We verify the easier parts of the hypotheses of the
reconstruction theorem of~\cite{HK}.
The harder parts are done in Section~\ref{Sec:1b}.

\begin{ntn}\label{N:PrimA}
For a \ca~$A$,
we let $\Prim (A)$ denote its primitive ideal space.
For an isomorphism $\ph \colon A \to B$ of \ca{s},
we let $\Prim (\ph) \colon \Prim (A) \to \Prim (B)$
denote the induced map $\Prim (\ph) (J) = \ph (J) \S B$
for primitive ideals $J \S A$.
\end{ntn}

Our notation is chosen so that an action of a group $G$ on~$A$
corresponds to an action of $G$ on $\Prim (A)$,
without replacing $g$ by $g^{-1}$.
However, our definition of $\Prim (\ph)$ goes the ``wrong way''.
(If $\ph$ is surjective,
then the inverse image of a primitive ideal is primitive,
but the direct image might not be.
For injective maps, the inverse image need not be primitive,
and the direct image need not be an ideal.
All counterexamples can be taken to have dimension at most~$4$.)

\begin{dfn}[Definition~A.1 of~\cite{HK}]\label{D:PtComplete}
Let $X$ be a ${\mathrm{T}}_0$-space.
Then $X$ is said to be {\emph{prime}}
(in algebraic geometry: irreducible)
if $X$ is not the union of two proper closed subsets.
The space $X$ is said to be {\emph{point-complete}}
if for every nonempty closed $T \S X$ which is prime in its relative
topology,
there is $x \in T$ such that $T = {\overline{ \{ x \} }}$.
\end{dfn}

Other terms are used in the literature.
In Definition~4.9 of~\cite{HfKm},
the term irreducible is used,
a closed subset is called monogenic if
it is the closure of a point,
and (also see the beginning of Section~3 of~\cite{BE})
a point-complete ${\mathrm{T}}_0$-space is called
a spectral space.
In Definition~2.6 of~\cite{HL},
a point-complete ${\mathrm{T}}_0$-space is called sober.

\begin{dfn}[Definition~1.3 of~\cite{HK}]\label{D:Pseudo}
Let $X$ and $Z$ be ${\mathrm{T}}_0$-spaces,
and let $\pi \colon X \to Z$ be \ct.
\begin{enumerate}
\item\label{D:Pseudo:Graph}
We define the {\emph{pseudo-graph}} $R_{\pi}$ of $\pi$
to be
\[
R_{\pi} = \big\{ (x, y) \in X \times X
  \colon \pi (y) \in {\overline{ \{ \pi (x) \} }} \big\}
\]
\item\label{D:Pseudo:Inv}
A subset $S \S X$ is {\emph{$R_{\pi}$-invariant}}
if $y \in S$ and $(x, y) \in R_{\pi}$ imply $x \in S$.
\item\label{D:Pseudo:Open}
The map $\pi \colon X \to Z$ is {\emph{pseudo-open}} if:
\begin{enumerate}
\item\label{D:Pseudo:Open:1}
The map $R_{\pi} \to X$, given by $(x, y) \mapsto x$, is open.
\item\label{D:Pseudo:Open:2}
For every $R_{\pi}$-invariant open set $U \S X$,
the image $\pi (U)$ is an open subset of $\pi (X)$.
\end{enumerate}
\item\label{D:Pseudo:Epi}
The map $\pi \colon X \to Z$ is {\emph{pseudo-epimorphic}} if
for every closed subset $F \S Z$,
the subset $\pi (X) \cap F$ is dense in~$F$.
\end{enumerate}
\end{dfn}

\begin{thm}[\cite{HK}; \cite{Kb}]\label{T:ExistA}
Let $X$ be a ${\mathrm{T}}_0$-space.
Suppose that $X$ is point-complete (Definition~\ref{D:PtComplete})
and that there exists a locally compact Polish space~$Z$
and a \ct{} map $Z \to X$ which is
pseudo-open (Definition \ref{D:Pseudo}(\ref{D:Pseudo:Open}))
and pseudo-epimorphic (Definition \ref{D:Pseudo}(\ref{D:Pseudo:Epi})).
Then there exists a separable nuclear \ca~$A$
such that $K \otimes \OT \otimes A \cong A$,
and a \hme{} $g \colon X \to \Prim (A)$.
Moreover, we have the following uniqueness statement.
Whenever $B$ is a separable nuclear \ca{}
such that $K \otimes \OT \otimes B \cong B$ and
$h \colon X \to \Prim (B)$ is a \hme,
then there exists an isomorphism $\ph \colon A \to B$
such that, with $\Prim (\ph)$ being as in Notation~\ref{N:PrimA},
we have $\Prim (\ph) \circ g = h$.
\end{thm}

\begin{proof}
Existence of~$A$ is Corollary~1.5 of~\cite{HK}.
Uniqueness is in the discussion after Corollary~1.5 of~\cite{HK}.
\end{proof}

\begin{dfn}\label{D:X}
Let $P$ be a locally compact Hausdorff space.
We define a non-Hausdorff compactification $\Xi (P)$ as follows.
As a set,
$\Xi (P) = P \cup \{ \I \}$,
the same as for the one point (Alexandroff) compactification~$P^+$.
The topology of $\Xi (P)$ consists of the following sets:
\begin{itemize}
\item
$\Xi (P)$.
\item
$\E$.
\item
All complements in $\Xi (P)$ of compact subsets of~$P$.
\end{itemize}
\end{dfn}

If $P$ is compact,
we have just added an extra point whose closure is the whole space.
If $P$ is in addition second countable,
the resulting space is easily realized as the primitive ideal
space of a separable nuclear \ca.
Namely,
for a unital essential extension
\[
0 \longrightarrow K
  \longrightarrow E
  \longrightarrow C (P)
  \longrightarrow 0,
\]
we have $\Prim (E) \cong \Xi (P)$.
However, this algebra,
and in fact the space $\Xi (P)$ for compact~$P$,
are not useful for the purposes of this paper.

\begin{lem}\label{L:Top}
The sets in Definition~\ref{D:X}
do in fact form a topology on $\Xi (P)$,
and $\Xi (P)$ is a compact ${\mathrm{T}}_0$-space
such that every nonempty open subset of $\Xi (P)$ contains~$\infty$.
\end{lem}

\begin{proof}
It is obvious that the collection of complements of these sets
is closed under finite unions and arbitrary intersections.
It is easy to check that this topology makes
$\Xi (P)$ a ${\mathrm{T}}_0$-space.

For compactness,
let ${\mathcal{U}}$ be an open cover of $\Xi (P)$.
Choose $U_0 \in {\mathcal{U}}$ such that $\infty \in U_0$.
If $U_0 = \Xi (P)$, we have found a finite subcover of~${\mathcal{U}}$.
Otherwise, $\Xi (P) \SM U_0$ is a subset of~$P$
which is compact in the original topology on~$P$.
The sets $U \SM \{ \I \}$, for $U \in {\mathcal{U}}$,
are open subsets of $P$ in its original topology
and which cover $\Xi (P) \SM U_0$.
Since $\Xi (P) \SM U_0$ is compact,
there are $n$ and $U_1, U_2, \ldots, U_n \in {\mathcal{U}}$
such that
$U_1 \SM \{ \I \}, \, U_2 \SM \{ \I \}, \, \ldots, \, U_n \SM \{ \I \}$
cover $\Xi (P) \SM U_0$.
Then $U_0, U_1, U_2, \ldots, U_n$ cover $\Xi (P)$.

It is immediate from Definition~\ref{D:X}
that every nonempty open subset of $\Xi (P)$ contains~$\infty$.
\end{proof}

\begin{cor}\label{L-Prime426}
Let $P$ be a locally compact Hausdorff space,
and let $A$ be a \ca{} such that $\Prim (A) \cong \Xi (P)$.
Then $A$ is prime.
\end{cor}

\begin{proof}
It is immediate from the last part of Lemma~\ref{L:Top}
that $\Xi (P)$ does not contain two disjoint nonempty open subsets.
\end{proof}

\begin{lem}\label{L:PtC}
Let the notation be as in Definition~\ref{D:X}.
The space $\Xi (P)$ is point-complete
in the sense of Definition~\ref{D:PtComplete}.
\end{lem}

\begin{proof}
The nonempty closed subsets of $\Xi (P)$ consist of:
\begin{itemize}
\item
$\Xi (P)$.
\item
All nonempty compact subsets of $P \S \Xi (P)$.
\end{itemize}
Of these,
$\Xi (P) = {\overline{ \{ \I \} }}$,
one element subsets of $P$ are their own closures,
and compact subsets of $P$ with more than one element
are not prime.
\end{proof}

We will prove that $\Xi (P)$ is the primitive ideal space
of a separable nuclear \ca{}
using Theorem~\ref{T:ExistA}.

The space $\Xi (P)$ is also {\emph{coherent}},
in the sense that the intersection of
two compact $G_{\dt}$-subsets of $\Xi (P)$ is again compact.
(This is fairly straightforward to check, but we omit the proof.)
It is known~\cite{KrC} (but not yet published)
that every point-complete locally compact second countable
coherent ${\mathrm{T}}_0$-space is the primitive ideal space
of a separable nuclear \ca.
However, the proof is much more complicated
than what we give in this paper.

We will need machinery to verify the other hypotheses of
Theorem~\ref{T:ExistA}.
The following definition is intended only for use
in this paper.

\begin{dfn}\label{D:DR}
Let $P$ be
a noncompact second countable locally compact Hausdorff space.
A {\emph{\dr}} for $P$ is a $5$-tuple $(Y, Z, \rh, h, q)$
consisting of:
\begin{enumerate}
\item\label{D:DR:Y}
A second countable locally compact Hausdorff space~$Y$.
\item\label{D:DR:Z}
A compact metric space $Z$ with metric~$\rh$.
\item\label{D:DR:h}
An injective \ct{} function $h \colon Y \to Z$.
\item\label{D:DR:q}
A proper open \ct{} surjective function $q \colon Y \to P$.
\setcounter{TmpEnumi}{\value{enumi}}
\end{enumerate}
In addition, we require that:
\begin{enumerate}
\setcounter{enumi}{\value{TmpEnumi}}
\item\label{D:DR:EpD}
For every $\ep > 0$ there is a compact set $K \S P$
such that for all $x \in P \SM K$,
the set $h (q^{-1} (x))$ is $\ep$-dense in~$Z$.
\end{enumerate}
\end{dfn}

In Definition~\ref{D:DR}, since $q$ is proper,
we may replace Condition~(\ref{D:DR:EpD})
by the requirement that
for every $\ep > 0$ there be a compact set $L \S Y$
such that for all $y \in Y \SM L$,
the set $h (q^{-1} ( q (y)))$ is $\ep$-dense in~$Z$.

\begin{lem}\label{L:CompDense}
Let $P$ be
a noncompact second countable locally compact Hausdorff space,
and let $(Y, Z, \rh, h, q)$
be a \dr{} for~$P$.
Then $Z \SM h (Y)$ is dense in~$Z$.
\end{lem}

\begin{proof}
For every compact set $L \S Y$,
the set $Z \SM h (L)$ is open.
We will prove that it is also dense.
Since $Y$ is a countable union of compact sets,
the Baire Category Theorem will imply that $Z \SM h (Y)$ is dense.

We in fact prove that $Z \SM h (L)$ is $\ep$-dense for every $\ep > 0$.
Choose a compact set $K \S P$
as in Definition~\ref{D:DR}(\ref{D:DR:EpD})
for the given value of~$\ep$.
Set $E = q^{-1} (q (L)) \cup q^{-1} (K)$,
which is a compact subset of $Y$ because $q$ is proper.
Since $P$ is not compact, neither is~$Y$, so $Y \SM E \neq \E$.
Choose $y \in Y \SM E$.
Then $h (q^{-1} ( q (y)))$ is $\ep$-dense in~$Z$.
Also,
$q^{-1} ( q (y)) \cap E = \E$ and $h$ is injective,
so $h (q^{-1} ( q (y))) \S Z \SM h (L)$.
Thus $Z \SM h (L)$ is $\ep$-dense.
\end{proof}

\begin{lem}\label{L:VerPseudo}
Let $X$ and $Z$ be ${\mathrm{T}}_0$-spaces,
and let $\pi \colon X \to Z$ be \ct, surjective, and open.
Then $\pi$ is pseudo-open and pseudo-epimorphic.
\end{lem}

\begin{proof}
A surjective map trivially satisfies
Definition \ref{D:Pseudo}(\ref{D:Pseudo:Epi}).
An open map trivially satisfies
Condition~(\ref{D:Pseudo:Open:2})
of Definition \ref{D:Pseudo}(\ref{D:Pseudo:Open}).
It remains only to prove
Condition~(\ref{D:Pseudo:Open:1})
of Definition \ref{D:Pseudo}(\ref{D:Pseudo:Open}).
By considering unions, it suffices to consider
the images  of open subsets of $R_{\pi}$
of the form $R_{\pi} \cap (U \times V)$ for $U, V \S X$ open.

We claim that if $U, V \S X$ are open,
then the image $S$ of $R_{\pi} \cap (U \times V)$
is equal to $U \cap \pi^{-1} (\pi (V))$.
This will prove the lemma,
because the hypotheses on $\pi$ imply that $U \cap \pi^{-1} (\pi (V))$
is open in~$X$.
First, if $x \in U \cap \pi^{-1} (\pi (V))$,
then $\pi (x) \in \pi (V)$.
Choosing $y \in V$ such that $\pi (y) = \pi (x)$
gives $(x, y) \in R_{\pi} \cap (U \times V)$.
So $S \S U \cap \pi^{-1} (\pi (V))$.
For the reverse inclusion, let $(x, y) \in R_{\pi} \cap (U \times V)$.
It is obvious that $x \in U$.
Suppose now that $x \not\in \pi^{-1} (\pi (V))$.
Then $\pi (x) \not\in \pi (V)$.
Since $\pi (V)$ is open
(because $\pi$ is an open map),
it follows that ${\overline{ \{ \pi (x) \} }} \cap \pi (V) = \E$.
Thus, there is no $y \in V$ such that
$\pi (y) \in {\overline{ \{ \pi (x) \} }}$.
So there is no $y \in V$ such that $(x, y) \in R_{\pi}$,
which contradicts $x \in S$.
We conclude that $x \in \pi^{-1} (\pi (V))$,
completing the proof that $S \S U \cap \pi^{-1} (\pi (V))$.
\end{proof}

\begin{prp}\label{P:AppOfDR}
Let $P$ be
a noncompact second countable locally compact Hausdorff space,
and let $(Y, Z, \rh, h, q)$ be a \dr{} for $P$.
Define $\pi \colon Z \to \Xi (P)$ as follows:
\begin{enumerate}
\item\label{P:AppOfDR:h}
If $z = h (y)$ for some $y \in Y$
(necessarily unique), then $\pi (z) = q (y)$.
\item\label{P:AppOfDR:Noth}
If $z \in Z \SM h (Y)$, then $\pi (z) = \I$.
\end{enumerate}
Then $\pi$ is surjective, \ct, and open.
\end{prp}

\begin{proof}
Every point of $P$ is in $\pi (Z)$ because $q$ is surjective.
Also $\I \in \pi (Z)$ because
Lemma~\ref{L:CompDense} implies that
$Z \SM h (Y) \neq \E$.
So $\pi$ is surjective.

We claim that $\pi$ is \ct.
Let $F \S \Xi (P)$ be closed.
If $\I \in F$ then $F = \Xi (P)$ so $\pi^{-1} (F) = Z$ is closed.
Otherwise, $F$ is a compact subset of~$P$.
Since $q$ is proper
(Definition~\ref{D:DR}(\ref{D:DR:q})),
the set $q^{-1} (F)$ is a compact subset of~$Y$.
Therefore
$\pi^{-1} (F) = h (q^{-1} (F))$ is a compact subset of~$Z$,
and hence closed.
This proves the claim.

We complete the proof by showing that $\pi$ is open.
Clearly $\pi (\E) = \E$ is open.
So let $U \S Z$ be a nonempty open subset.
Then $h^{-1} (U) \S Y$ is open.
Since $q$ is open (Definition~\ref{D:DR}(\ref{D:DR:q})),
the set $\pi (U) \SM \{ \I \} = q (h^{-1} (U))$ is open in~$P$
in its usual topology.

Using Lemma~\ref{L:CompDense},
choose $z \in U \cap ( Z \SM h (Y))$,
and then choose $\ep > 0$ such that
$B_{\ep} (z) = \{ y \in Z \colon \rh (y, z) < \ep \}$ is contained
in~$U$.
Choose (Definition~\ref{D:DR}(\ref{D:DR:EpD}))
a compact subset $K \S P$ such that
for all $x \in P \SM K$,
the set $h (q^{-1} (x))$ is $\ep$-dense in~$Z$.
In particular,
if $x \in (P \SM K) \cup \{ \I \}$
then $x \in \pi (B_{\ep} (z)) \S \pi (U)$.
Thus, $\Xi (P) \SM \pi (U) \S K$.
Since we already know that $\Xi (P) \SM \pi (U)$ is closed
in the usual topology on~$P$,
it follows that  $\Xi (P) \SM \pi (U)$ is compact.
Therefore $\pi (U)$ is open in $\Xi (P)$.
This completes the proof.
\end{proof}

We now know that if $P$ is a noncompact second countable
locally compact Hausdorff space which has a pseudocovering,
then $\Xi (P)$ satisfies the hypotheses on $X$
in Theorem~\ref{T:ExistA},
so that $\Xi (P)$ is the primitive ideal space
of a separable nuclear \ca.
In fact, the existence of \dr{s} is automatic,
but we postpone the proof to the next section.
We finish this section by giving some basic functoriality results.

\begin{prp}\label{P:Funct}
Let $P$ and $Q$ be locally compact Hausdorff spaces,
and let $h \colon P \to Q$ be \ct{} and proper.
Then the map $\Xi (h) \colon \Xi (P) \to \Xi (Q)$,
given by $\Xi (h) |_P = h$ and $\Xi (h) (\I) = \I$,
is \ct.
Moreover, with this definition, $\Xi$ becomes a functor
from locally compact Hausdorff spaces and proper maps
to compact ${\mathrm{T}}_0$-spaces and \ct{} maps.
\end{prp}

\begin{proof}
The only point needing proof is continuity of $\Xi (h)$.
But this follows easily from the fact that if $K \S Q$ is compact,
then so is $h^{-1} (K)$.
\end{proof}

\begin{prp}\label{P:IndActXi}
Let $P$ be a locally compact Hausdorff space,
and let $G$ be a topological group which acts continuously on~$P$.
Then the induced action on $\Xi (P)$
(following Proposition~\ref{P:Funct}) is \ct.
\end{prp}

We don't need the action to be proper,
but of course for each~$g$,
the map $x \mapsto g x$ is a \hme{} and therefore a proper map.

\begin{proof}[Proof of Proposition~\ref{P:IndActXi}]
Let $U \S \Xi (P)$ be open,
and let $g_0 \in G$ and $x_0 \in \Xi (P)$ satisfy $g_0 x_0 \in U$.
We have to find open sets $V \S \Xi (P)$ with $x_0 \in V$
and $W \S G$ with $g_0 \in W$ such that
whenever $g \in W$ and $x \in V$,
then $g x \in U$.
There is nothing to do if $U = \Xi (P)$ (or $U = \E$).
Therefore we may assume that there is a subset $K \S P$
which is compact in the original topology on~$P$
such that $U = \Xi (P) \SM K$.
Since $P$ is locally compact and $x_0 \not\in g_0^{-1} K$,
there is a compact subset $L \S P$ (for the original topology on~$P$)
such that $x_0 \not\in L$ and $g_0^{-1} K \S \sint (L)$.
(The interior is taken in~$P$.)

A standard argument, using compactness of~$K$,
shows that
$W = \{ g \in G \colon g^{-1} K \S \sint (L) \}$
is a \nbhd{} of $g_0$ in~$G$.
So $W \times (\Xi (P) \SM L)$ is a neighborhood of $(g_0, x_0)$
in $G \times \Xi (P)$.
One checks easily that $g \in W$ and $x \in V = \Xi (P) \SM L$
imply $g x \not\in K$.
\end{proof}

We identify $G$-spaces~$P$
such that the induced action on $\Xi (P)$ is minimal.

\begin{lem}\label{L:MinXiP}
Let $G$ be a locally compact group
acting on a locally compact Hausdorff space~$P$.
Let $\Xi (P)$ be as in Definition~\ref{D:X}.
\Tfae:
\begin{enumerate}
\item\label{L:MinXiP:Min}
The induced action of $G$ on $\Xi (P)$ is minimal.
\item\label{L:MinXiP:CptInv}
There are no nonempty compact $G$-invariant subsets of~$P$.
\item\label{L:MinXiP:CptOrb}
For each $x \in P$,
the orbit closure ${\overline{G x}}$ is not compact.
\end{enumerate}
\end{lem}

\begin{proof}
We prove that (\ref{L:MinXiP:CptInv}) implies~(\ref{L:MinXiP:Min}).
Assume that the induced action of $G$ on $\Xi (P)$ is not minimal,
so that there is
a nontrivial closed $G$-invariant subset $K \S \Xi (P)$.
Then $\I \not\in K$ since $K \neq \Xi (P)$.
Therefore $K$ is a nonempty compact subset of $P$
in its original topology,
necessarily $G$-invariant.

To see that (\ref{L:MinXiP:Min}) implies~(\ref{L:MinXiP:CptOrb}),
suppose there is $x \in P$ such that ${\overline{G x}}$ is compact.
Then ${\overline{G x}}$ is a nonempty closed subset of $\Xi (P)$
which contains $x$ but not~$\I$,
contradicting minimality.

For (\ref{L:MinXiP:CptOrb}) implies~(\ref{L:MinXiP:CptInv}),
assume that $K \S P$
is a nonempty compact $G$-invariant subset of~$P$.
Choose $x \in K$.
Then ${\overline{G x}}$ is compact.
\end{proof}

We give some examples for Lemma~\ref{L:MinXiP}.
First, we recall the following definition,
which we use to rule out one form of triviality.

\begin{dfn}\label{D-Effective}
Let $G$ be a group and let $S$ be a set.
(In particular, $S$ could be a topological space or a \ca.)
An action $(g, s) \mapsto g s$ is {\emph{effective}}
if for any $g \in G \SM \{ 1 \}$,
there is $s \in S$ such that $g s \neq s$.
\end{dfn}

That is,
no nontrivial element of $G$ acts as the identity.
This definition is standard for actions on topological spaces.
If $X$ is a \chs,
then an action of $G$ on~$X$ is effective
\ifo{} the corresponding action on $C (X)$ is effective.

\begin{exa}\label{E-111109dd}
The following are all examples of actions
of a noncompact second countable locally compact group~$G$
on a second countable locally compact Hausdorff space~$P$
such that the induced action of $G$ on $\Xi (P)$
is minimal and effective.
\begin{enumerate}
\item\label{E-111109dd-1}
Let $G$ as above be otherwise arbitrary,
and take $P = G$ with left translation.
The hypothesis of Lemma~\ref{L:MinXiP}(\ref{L:MinXiP:CptOrb})
is immediate.
\item\label{E-111109dd-2}
Let $G$ as above be otherwise arbitrary,
let $T$ be any second countable locally compact Hausdorff space,
and take $P = G \times T$ with left translation
in the first coordinate
and the trivial action in the second coordinate.
The hypothesis of Lemma~\ref{L:MinXiP}(\ref{L:MinXiP:CptOrb})
is again immediate.
\item\label{E-111109dd-3}
Let $G$ as above be otherwise arbitrary,
and let $H \subset G$ be a closed subgroup of~$G$
such that $G / H$ is not compact and
$\bigcap_{g \in G} g H g^{-1} = \{ 1 \}$.
Take $P = G / H$ with left translation.
Since the action is transitive,
the hypothesis of Lemma~\ref{L:MinXiP}(\ref{L:MinXiP:CptOrb})
is satisfied.
\item\label{E-111109dd-4}
Let $X$ be the Cantor set,
and set $P = \Z \times X$.
In Section~3 of~\cite{Dn},
there is a construction of many minimal \hme{s} of~$P$.
Each of them gives an action of~$\Z$ on~$P$
which obviously satisfies the condition of
Lemma~\ref{L:MinXiP}(\ref{L:MinXiP:CptInv}).
\end{enumerate}
\end{exa}

In Example~\ref{E-111109dd}(\ref{E-111109dd-3}),
noncompactness of $G / H$ is automatic if $G$ is discrete
(this is not hard),
but can fail otherwise.
The following example is due to George Willis (via Ken Ross).
Let $\ph \in \Aut (S^1 \times S^1)$ be the group automorphism
given by $\ph (\zt_1, \zt_2) = (\zt_1 \zt_2, \, \zt_2)$
for $\zt_1, \zt_2 \in S^1$.
This generates an action of~$\Z$ on $S^1 \times S^1$.
Let $G = \Z \ltimes (S^1 \times S^1)$ be the corresponding
semidirect product.
Take $H = \{ (n, 1, 1) \colon n \in \Z \}$.
Then $\bigcap_{g \in G} g H g^{-1} = \{ 1 \}$,
but $G / H$ is compact.

\section{The existence of pseudocoverings}\label{Sec:1b}

\indent
In this section, we prove that every second countable locally compact
noncompact Hausdorff space~$P$ has a \dr{} in the sense of
Definition~\ref{D:DR}.
This is the only remaining step
of the proof that $\Xi (P)$ is the primitive ideal space
of a separable nuclear \ca.

We begin by constructing a \dr{} for a special space.
We use the following ``affine transversality'' lemma.
It is surely well known,
but we do not know a reference.

\begin{lem}\label{L:AffTrans}
Let $m, n, d, k \in \Nz$ satisfy $m + n + 1 \leq d$
and $0 \leq k \leq n$.
Let $K \S \R^d$
be a finite union of images of subsets of $\R^m$ under affine maps.
Let $h \colon \R^n \to \R^d$
be an affine map
such that $h |_{[0, 1]^k \times \{ 0 \} }$ is injective.
Let $\xi_1, \xi_2, \ldots, \xi_n$
be the standard basis vectors of~$\R^n$.
For
$\et = (\et_{k + 1}, \et_{k + 2}, \ldots, \et_{n}) \in \R^{d (n - k)}$,
let $h_{\et} \colon \R^n \to \R^d$
be the affine map determined by
\[
h_{\et} (0) = h (0),
\,\,\,
h_{\et} (\xi_1) = h (\xi_1),
\,\,\,
h_{\et} (\xi_2) = h (\xi_2),
\,\,\,
\ldots,
\,\,\,
h_{\et} (\xi_k) = h (\xi_k),
\]
\[
h_{\et} (\xi_{k + 1}) = \et_{k + 1},
\,\,\,
h_{\et} (\xi_{k + 2}) = \et_{k + 2},
\,\,\,
\ldots,
\,\,\,
h_{\et} (\xi_{n}) = \et_{n}.
\]
Then the set
\[
W_K = \big\{ \et \in \R^{d (n - k)} \colon
  {\mbox{$h_{\et}$ is injective and
   $h_{\et} \big( [0, 1]^n \SM \big[ [0, 1]^k \times \{ 0 \} \big] \big)
        \cap K = \E$}} \big\}
\]
contains a dense open set in $\R^{d (n - k)}$.
\end{lem}

\begin{proof}
The statement is vacuously true when $k = n$, so assume that $k < n$.

It suffices to prove the result when $K$ is
a finite union $K = K_1 \cup K_2 \cup \cdots \cup K_s$
of images of $\R^m$ under affine maps.
By intersecting the corresponding sets
$W_{K_1}, W_{K_2}, \ldots, W_{K_s}$,
it then suffices to prove this when $K$ is just one such image,
that is, there is an affine map $g \colon \R^m \to \R^d$
such that $K = g (\R^m)$.
Replacing $g$ by $\xi \mapsto g (\xi) - g (0)$
and $h$ by $\xi \mapsto h (\xi) - g (0)$,
we may assume that $0 \in K$.
Thus, $K$ is a subspace of $\R^d$ of dimension $m_0 \leq m$.
So
\[
V = \spn \big( K, \, h (0), \,
  h (\xi_1), \, h (\xi_2), \, \ldots, \, h (\xi_k) \big)
\]
is a subspace of $\R^d$
of dimension $r \leq m_0 + k + 1 \leq m + k + 1$.
Let $\mu_1, \mu_2, \ldots, \mu_r \in \R^d$ form a basis for~$V$.

Let $U_0$ be the set of all $(n - k)$-tuples
$(\nu_{k + 1}, \nu_{k + 2}, \ldots, \nu_{n}) \in \R^{d (n - k)}$
such that the $r + n - k$ elements
\[
\mu_1, \, \mu_2, \, \ldots, \, \mu_r, \,
\nu_{k + 1}, \, \nu_{k + 2}, \, \ldots, \, \nu_{n} \in \R^d
\]
are linearly independent.
Since $r + n - k \leq d$,
the set $U_0$ is dense and open in $\R^{d (n - k)}$.
Let $U \S \R^{d (n - k)}$ be the
dense open set consisting of all $(n - k)$-tuples
$(\et_{k + 1}, \et_{k + 2}, \ldots, \et_{n}) \in \R^{d (n - k)}$
such that
\[
\big( \et_{k + 1} - h (0), \, \et_{k + 2} - h (0),
 \, \ldots, \, \et_{n} - h (0) \big) \in U_0.
\]

To finish the proof, we show that if
$\et = (\et_{k + 1}, \et_{k + 2}, \ldots, \et_{n}) \in U$,
then $h_{\et}$ satisfies the requirements in the conclusion.
Observe that $h_{\et}$ is given by the formula,
for $\af_1, \af_2, \ldots, \af_n \in \R$,
\[
h_{\et} (\af_1, \af_2, \ldots, \af_n)
 = h (0) + \sum_{j = 1}^k \af_j [ h (\xi_j) - h (0) ]
         + \sum_{j = k + 1}^n \af_j [ \et_j - h (0) ].
\]

We prove that $h_{\et}$ is injective.
First, for $j = 1, 2, \ldots, k$, the elements
\[
h_{\et} (\xi_j) - h_{\et} (0) = h (\xi_j) - h (0)
\]
are linearly independent
because $h |_{[0, 1]^k \times \{ 0 \} }$ is injective.
Also, by construction,
for $j = k + 1, \, k + 2, \, \ldots, \, n$, the elements
\[
h_{\et} (\xi_{j}) - h_{\et} (0) = \et_{j} - h (0)
\]
have linearly independent images mod~$V$,
and in particular have linearly independent images mod
\[
\spn \big( \big\{ h_{\et} (\xi_j) - h_{\et} (0) \colon
   j = 1, 2, \ldots, k \big\} \big).
\]
So the elements
$h_{\et} (\xi_j) - h_{\et} (0)$,
for $j = 1, 2, \ldots, n$,
are linearly independent,
whence $h_{\et}$ is injective.

To prove that
$h_{\et} \big( [0, 1]^n \SM ([0, 1]^k \times \{ 0 \}) \big)
        \cap K = \E$,
it suffices to show that if $\af_1, \af_2, \ldots, \af_n \in \R$
and $h_{\et} (\af_1, \af_2, \ldots, \af_n) \in V$,
then $\af_{k + 1} = \af_{k + 2} = \cdots = \af_n = 0$.
But this is immediate from the fact,
observed in the previous paragraph,
that,
for $j = k + 1, \, k + 2, \, \ldots, \, n$, the elements
$h_{\et} (\xi_j) - h_{\et} (0) = \et_j - h (0)$
have linearly independent images mod~$V$.
\end{proof}

\begin{lem}\label{L:SpecialDR}
Let $n \in \Nz$.
Then there exists a \dr{} $(Y, Z, \rh, h, q)$
for $[0, 1]^n \times [0, \I)$,
with $Z = [0, 1]^{2 n + 3}$.
\end{lem}

\begin{proof}
Set $N = 2 n + 3$.
For convenience,
we use $Z = [-1, 1]^{N}$ instead of $Z = [0, 1]^{N}$.

Whether a $5$-tuple $(Y, Z, \rh, h, q)$ is a \dr{}
does not depend on the choice of the metric $\rh$ on~$Z$
(as long as it defines the right topology),
so for convenience we take $\rh$ to be given by the
supremum norm $\| \cdot \|_{\I}$ on $\R^N$.
We also use $\| \cdot \|_{\I}$ on $\R^n$, $\R^{n + 1}$,
and their subsets.

Set $S = \{ -1, 1 \}^N$.
We construct an infinite tree $Y^{(0)}$
with branching of degree $\card (S)$
at each vertex.
(The set $Y$ in the statement of the lemma
will be a subset of $[0, 1]^m \times Y^{(0)}$.)
The construction is direct and elementary,
but we need names for the pieces.
For $m \in \N$ and $x \in S^m$,
set $T_{m, x} = [m - 1, \, m] \times \{ x \}$,
and let
\[
T_m = [m - 1, \, m] \times S^m = \coprod_{x \in S^m} T_{m, x}.
\]
Define $T = \coprod_{m \in \N} T_m$.
(This is the disjoint union of the edges.)
We identify appropriate endpoints
using the equivalence relation $\sim$
on~$T$ defined as follows.
The equivalence classes with more than one element consist of
$\{ 0 \} \times S \S T_1$
and, for $m \in \N$ and $x \in S^m$,
the set
\[
\{ (m, x) \} \cup \{ (m, \, (x, s)) \colon s \in S \}
 \S T_{m, x}
    \cup \bigcup_{y \in S} T_{m + 1, \, (x, y)}.
\]
Then set $Y^{(0)} = T / {\sim}$.
It is easily seen to be
a second countable locally compact Hausdorff space.
For $(\ld, x) \in T_{m, x} \S T$,
we write $[\ld, x]$ for its equivalence class in~$Y^{(0)}$.
The vertices can thus be identified with
$\coprod_{m = 0}^{\I} S^m$,
with $x \in S^m$, for $m \geq 1$,
corresponding to $[m, x]$,
and with the single element of~$S^0$
corresponding to the equivalence class
$\{ 0 \} \times S \S T_1$.
When notationally convenient,
we also write $T_{m, x}$ for its image in $Y^{(0)}$.
(The identification map is a \hme{} from $T_{m, x}$ to its image.)
For $m \in \N$ further let $Y^{(0)}_m$ be
the image of $\bigcup_{k = 1}^m T_m$ in $Y^{(0)}$.
Set $Y^{(0)}_0 = \{ [0, s] \}$ for any particular $s \in S$.
(This point does not depend on~$s$.)
For $m \in \Nz$,
the set $Y_m^{(0)}$ is a finite tree and a compact subset of $Y^{(0)}$,
and $Y^{(0)}$ is the increasing union of the $Y^{(0)}_m$.

Define a function $q_0 \colon Y^{(0)} \to [0, \I)$,
giving the ``distance'' of a point of $Y^{(0)}$ from~$Y^{(0)}_0$,
by
$q_0 ([\ld, x]) = \ld$ for $\ld \in [m - 1, \, m]$ and $x \in S^m$
with $m \in \N$.
The definition of $\sim$ implies that $q_0$ is well defined,
and then one easily checks that
$q_0$ is surjective, \ct, open, and proper.
Therefore so is
\[
\id_{[0, 1]^n} \times q_0 \colon [0, 1]^n \times Y^{(0)}
  \to [0, 1]^n \times [0, \I).
\]

We now construct a \cfn{} $f \colon Y^{(0)} \to \R^N$
such that, as the distance $\ld$ from $Y_0^{(0)}$ increases,
the points $f (y)$ for $y \in Y^{(0)}$ such that $q_0 (y) = \ld$
are $\ep_{\ld}$-dense in $[-1, 1]^N$ with $\ep_{\ld} \to 0$
as $\ld \to \I$.
(This function will not be injective.)
The construction is carried out by defining
maps $f_m$ from the image in $Y^{(0)}$ of~$T_m$
to~$\R^N$,
by induction on~$m$.

Define $f_0$ by sending the single point in $Y_0^{(0)}$
to $0 \in \R^N$.

Given $f_m$,
for $\ld \in [0, 1]$, $x \in S^m$, and $s \in S$,
define
\[
f_{m + 1} ([m + \ld, \, (x, s) ]) = f_m ([m, x]) + 2^{- m - 1} \ld s.
\]
Then the $f_m$ are well defined and fit together to give
a well defined \cfn{} $f \colon Y^{(0)} \to \R^N$.
We claim that $f$ has the following properties:
\begin{enumerate}
\item\label{L:SpecialDR:1}
For $m \in \Nz$, we have
$\| f (y) \|_{\I} \leq 1 - 2^{-m}$ whenever $q_0 (y) \leq m$.
\item\label{L:SpecialDR:2}
For $m \in \N$,
the set
\[
\big\{ f (y) \colon {\mbox{$y \in Y^{(0)}$ with $q_0 (y) = m$}} \big\}
\]
consists exactly of all $z \in [-1, 1]^N$ such that
every coordinate $z_k$ of $z$
has the form $z_k = l_k / 2^m$ with $l_k \in \Z$ odd.
\item\label{L:SpecialDR:3}
Whenever $m \in \Nz$ and $\ld \geq m$,
then
\[
\big\{ f (y) \colon {\mbox{$y \in Y^{(0)}$ with $q_0 (y) = \ld$}} \big\}
\]
is $2^{- m + 1}$-dense in $[-1, 1]^N$.
\end{enumerate}

We prove~(\ref{L:SpecialDR:1}).
For $m = 0$ there is only one such~$y$,
and $f (y) = 0$,
so the statement holds.
Suppose now the statement holds for some particular value of~$m$.
We prove it for $m + 1$.
Since $1 - 2^{- (m + 1)} > 1 - 2^{-m}$,
we need only consider $y$ such that $q_0 (y) = m + \ld$
with $\ld \in [0, 1]$.
Then for $x \in S^m$ and $s \in S$,
using $\| s \|_{\I} = 1$, we have
\begin{align*}
\| f ([m + \ld, \, (x, s) ]) \|_{\I}
& \leq \| f ([m, x]) \|_{\I} + 2^{- m - 1} \ld  \| s \|_{\I} \\
& \leq 1 - 2^{-m} + 2^{- m - 1}
  = 1 - 2^{- m - 1}.
\end{align*}
So~(\ref{L:SpecialDR:1}) follows by induction.

We next prove~(\ref{L:SpecialDR:2}).
For $m \in \N$,
let
\[
Q_m = \big\{ f (y) \colon
       {\mbox{$y \in Y^{(0)}$ with $q_0 (y) = m$}} \big\}
\]
and let
\[
R_m = \big\{ z \in [-1, 1]^N \colon
       {\mbox{$z_k \in 2^{-m} (2 \Z + 1)$ for $k = 1, 2, \ldots, N$}}
              \big\}.
\]
We are supposed to prove that $Q_m = R_m$.
Using the definition of $f_1$,
we have
\[
Q_1 = \big\{ \tfrac{1}{2} s \colon s \in S \big\} = R_1.
\]
The proof now follows by induction,
since both
\[
Q_{m + 1} = Q_m + 2^{- m - 1} S
\andeqn
R_{m + 1} = R_m + 2^{- m - 1} S.
\]

Now we prove~(\ref{L:SpecialDR:3}).
It suffices to prove that for $m \in \Nz$ and $\ld \in [0, 1]$,
the set
\[
F_{m, \ld} = \big\{ f (y) \colon
   {\mbox{$y \in Y^{(0)}$ with $q_0 (y) = m + \ld$}} \big\}
\]
is $2^{- m + 1}$-dense in $[-1, 1]^N$.
Let $z \in [-1, 1]^N$.
It follows from
the identification of $F_{m, 1}$ in~(\ref{L:SpecialDR:2})
that there is $w \in F_{m, 1}$
with $\| w - z \|_{\I} \leq 2^{- (m + 1)}$.
Write
\[
w = f ([m + 1, \, (x, s)])
\]
 with $x \in S^m$ and $s \in S$.
For $\ld \in [0, 1]$,
the element
$w_{\ld} = f ([m + \ld, \, (x, s)])$ is in $F_{m, \ld}$ and satisfies
\[
\| w_{\ld} - w \|_{\I}
  = \| 2^{- m - 1} (1 - \ld) s \|_{\I}
  = 2^{- m - 1} (1 - \ld)
  \leq 2^{- m - 1}.
\]
Therefore
\[
\| w_{\ld} - z \|_{\I}
  \leq \| w_{\ld} - w \|_{\I} + \| w - z \|_{\I}
  \leq 2^{- m - 1} + 2^{- m - 1}
  < 2^{- m + 1}.
\]
This proves~(\ref{L:SpecialDR:3}).

We are now going to construct
a map $h^{(0)} \colon [0, 1]^n \times Y^{(0)} \to \R^N$
which satisfies the conditions on $h$ in Definition~\ref{D:DR},
except that,
first, we do not insist that the range of $h^{(0)}$
be contained in $[-1, 1]^N$,
and, second,
we require the extra condition
$\big\| h^{(0)} (0, y) - f ( y ) \big\|_{\I}
 < 2^{- m - 2}$
when $m \in \Nz$ and $q_0 (y) \geq m$.
We will fix the missing condition afterwards.
The construction is carried out by defining
maps $h_m^{(0)} \colon [0, 1]^n \times Y_m^{(0)} \to \R^N$
by induction on~$m$.

Recall that for $x \in S^m$,
we identify $T_{m, x} = [m - 1, \, m] \times \{ x \}$
with its image in~$Y^{(0)}$.
Our function $h^{(0)}$ will be piecewise affine,
in the sense that for $m \in \Nz$ and $x \in S^m$,
its restriction to
$[0, 1]^n \times T_{m, x}$ becomes affine under the obvious
identification of this set with $[0, 1]^{n + 1}$.

Choose an injective affine map $g_0 \colon \R^n \to \R^N$
such that $g_0 (0) = 0$.
Take $h_0^{(0)} = g_0 |_{[0, 1]^n}$.

Suppose now that $h^{(0)}_m$ has been constructed,
is \ct{} and injective,
is affine on $[0, 1]^n \times T_{k, x} \cong [0, 1]^{n + 1}$
whenever $k \leq m$ and $x \in S^k$,
and satisfies
\[
\big\| h^{(0)}_m (0, y) - f (y) \big\|_{\I} < 2^{- k - 2}
\]
for $y \in Y^{(0)}_m$ with $q_0 (y) \geq k$.
Label the elements of $S^{m + 1}$ as $x_1, x_2, \ldots, x_M$
with $M = 2^{(m + 1) N}$.

Write $x_1 = (w_1, s_1)$ with $w_1 \in S^m$ and $s_1 \in S$.
Let $\dt_1, \dt_2, \ldots, \dt_{n + 1}$ be the standard
basis vectors in $\R^{n + 1}$.
Let $g_{m + 1, \, 1} \colon \R^{n + 1} \to \R^N$ be the
unique affine map such that
$g_{m + 1, \, 1} (\xi, 0) = h^{(0)}_m (\xi, [m, w_1])$
for $\xi \in [0, 1]^n$,
and such that $g_{m + 1, \, 1} (\dt_{n + 1}) = f ([m + 1, \, x_1])$.
Apply Lemma~\ref{L:AffTrans} with $n + 1$ in place of both $m$ and~$n$,
with $N$ in place of~$d$,
with $n$ in place of~$k$,
and with $h^{(0)}_m \big( Y^{(0)}_m \big)$ in place of~$K$.
The lemma implies the existence of an injective affine function
$h_{m + 1, \, 1}^{(0)} \colon \R^{n + 1} \to \R^N$
such that
$h_{m + 1, \, 1}^{(0)} |_{\R^n \times \{ 0 \} }
  = g_{m + 1, \, 1} |_{\R^n \times \{ 0 \} }$,
such that
\[
\big\| h_{m + 1, \, 1}^{(0)} (\dt_{n + 1}) - f ([m + 1, \, x_1])
       \big\|_{\I}
    < 2^{- m - 3},
\]
and such that
\[
h_{m + 1, \, 1}^{(0)}
            \big( [0, 1]^{n + 1}
                 \SM \big[ [0, 1]^n \times \{ 0 \} \big] \big)
        \cap h^{(0)}_m \big( [0, 1]^n \times Y^{(0)}_m \big) = \E.
\]
Since
\[
\big\| h_{m + 1, \, 1}^{(0)} (0) - f ([m, x_1]) \big\|_{\I}
  = \big\| h^{(0)}_m ([m, w_1]) - f ([m, w_1]) \big\|_{\I}
  < 2^{- m - 2},
\]
convexity gives
\[
\big\| h_{m + 1, \, 1}^{(0)} (\ld \dt_{n + 1})
                  - f ([m + \ld, \, x_1]) \big\|_{\I}
  < 2^{- m - 2}
\]
for $\ld \in [0, 1]$.

We now proceed by induction on~$l$.
Suppose that for $k = 1, 2, \ldots, l$
we are given injective affine functions
$h_{m + 1, \, k}^{(0)} \colon \R^{n + 1} \to \R^N$
such that
\begin{equation}\label{Eq:1112-aa}
h_{m + 1, \, k}^{(0)} |_{\R^n \times \{ 0 \} }
  = g_{m + 1, \, k} |_{\R^n \times \{ 0 \} },
\end{equation}
such that
\begin{equation}\label{Eq:1112-bb}
\big\| h_{m + 1, \, k}^{(0)} (\dt_{n + 1}) - f ([m + 1, \, x_k])
       \big\|_{\I}
    < 2^{- m - 3},
\end{equation}
and
\begin{equation}\label{Eq:1112-cc}
\big\| h_{m + 1, \, k}^{(0)} ( \ld \dt_{n + 1}) - f ([m + \ld, \, x_k])
       \big\|_{\I}
    < 2^{- m - 2},
\end{equation}
and such that
\begin{align}\label{Eq:1112-dd}
&
h_{m + 1, \, k}^{(0)}
            \big( [0, 1]^{n + 1} \SM ([0, 1]^n \times \{ 0 \}) \big)
\\
& \hspace*{6em} {\mbox{}}
        \cap \left( h^{(0)}_m \big( [0, 1]^n \times Y^{(0)}_m \big)
               \cup \bigcup_{j = 1}^{k - 1}
                     h_{m + 1, \, j}^{(0)} \big( [0, 1]^{n + 1} \big)
                           \right)
     = \E.
  \notag
\end{align}
Let $g_{m + 1, \, l} \colon \R^{n + 1} \to \R^N$ be the
unique affine
map such that
$g_{m + 1, \, l} (\xi, 0) = h^{(0)}_m (\xi, [m, w_l])$
for $\xi \in [0, 1]^n$,
and such that $g_{m + 1, \, l} (\dt_{n + 1}) = f ([m + 1, \, x_l])$.
Using the same method as for $h_{m + 1, \, 1}^{(0)}$,
we find an injective affine function
$h_{m + 1, \, l}^{(0)} \colon \R^{n + 1} \to \R^N$
such that
(\ref{Eq:1112-aa}),
(\ref{Eq:1112-bb}),
(\ref{Eq:1112-cc}),
and~(\ref{Eq:1112-dd})
hold with $l + 1$ in place of~$k$.

With
$h_{m + 1, \, 1}^{(0)}, \, h_{m + 1, \, 2}^{(0)}, \, \ldots,
  \, h_{m + 1, \, M}^{(0)}$
in hand,
we proceed to define
\[
h_{m + 1}^{(0)} \colon [0, 1]^n \times Y^{(0)}_{m + 1} \to \R^N
\]
by $h_{m + 1}^{(0)} |_{[0, 1]^n \times Y^{(0)}_{m} } = h^{(0)}_m$
and
\[
h_{m + 1}^{(0)} (\xi, \, [m + \ld, \, x_k])
   = h_{m + 1, \, k}^{(0)} (\xi + \ld \dt_{n + 1})
\]
for $\xi \in [0, 1]^n$, $\ld \in [0, 1]$, and $k = 1, 2, \ldots, M$.
The choices made
($h_{m + 1, \, k}^{(0)}$ is injective and affine,
(\ref{Eq:1112-aa}),
(\ref{Eq:1112-bb}),
(\ref{Eq:1112-cc}),
and~(\ref{Eq:1112-dd}))
ensure that $h_{m + 1}^{(0)}$ is injective,
piecewise affine,
and satisfies
\begin{equation}\label{Eq:111108a}
\big\| h_{m + 1}^{(0)} (0, \, [l + \ld, \, x]) - f ([l + \ld, \,  x])
       \big\|_{\I}
    < 2^{- l - 2}
\end{equation}
for $l = 0, 1, \ldots, m$, $\ld \in [0, 1]$, and $x \in S^{l + 1}$.
There is then an injective,
piecewise affine function
$h^{(0)} \colon [0, 1]^n \times Y^{(0)} \to \R^N$
which agrees with $h_{m}^{(0)}$ on $[0, 1]^n \times Y_m^{(0)}$
for all $m \in \N$,
and also still satisfies the estimate~(\ref{Eq:111108a}),
now for all $l \in \Nz$.
Moreover, $h^{(0)}$ is affine on each set
$[0, 1]^n \times T_{m, x} \cong [0, 1]^n \times [m - 1, \, m]$
for $x \in S^m$.

We now modify our function $h^{(0)}$ by restricting it
to a suitable subset $Y \subset [0, 1]^n \times Y^{(0)}$
so that the range of the restriction is contained in $[-1, \, 1]^N$.

For $y \in Y^{(0)}$,
let $h_y \colon [0, 1]^n \to \R^N$
be given by $h_y (\xi) = h (\xi, y)$.
Then each $h_y$ is affine.
For $x \in S^m$ and $\ld \in [m - 1, \, m]$,
the maps $h_{[\ld, x]}$ are all restrictions of the same
affine map on $[0, 1]^n \times [m - 1, \, m]$
to subsets of the form $[0, 1]^n \times \{ \mu \}$.
Therefore they all have the same Lipschitz constant.
Since $Y^{(0)}$ is connected,
it follows that the maps $h_y$ have the same Lipschitz constant,
say~$L$, for all $y \in Y^{(0)}$.

Choose a \cfn{} $\ep \colon [0, \I) \to (0, 1)$
such that for $m \in \Nz$ and $\ld \in [0, 1]$,
we have $L \ep (m + \ld) < 2^{- m - 2}$.
Set
\[
Y = \big\{ (\xi, y) \in [0, 1]^n \times Y^{(0)} \colon
       \| \xi \|_{\I} \leq \ep (q_0 (y)) \big\}.
\]
Define $b \colon Y \to [0, 1]^n \times Y^{(0)}$
by $b (\xi, y) = \big( \ep (q_0 (y))^{-1} \xi, \, y \big)$
for $\xi \in [0, 1]^n$ and $y \in Y^{(0)}$.
Then $b$ is a \hme,
so that the function
\[
q = \big( \id_{[0, 1]^n} \times q_0 \big) \circ b \colon
      Y \to [0, 1]^n \times [0, \I)
\]
is surjective, \ct, open, and proper.

Further set $h = h^{(0)} |_Y$.
Then $h$ is injective and \ct.
By construction, we have
\begin{equation}\label{Eq:111108bb}
\| h (0, \, [m + \ld, \, x]) - f ([m + \ld, \,  x]) \|_{\I}
    < 2^{- m - 2}
\end{equation}
for $m \in \Nz$, $\ld \in [0, 1]$, and $x \in S^{m + 1}$.
Combining this with the property~(\ref{L:SpecialDR:1})
of~$f$,
we get
\[
\| h (0, \, [m + \ld, \, x]) \|_{\I} < 1 - 2^{- m - 2}
\]
for $m \in \Nz$, $\ld \in [0, 1]$, and $x \in S^{m + 1}$.
For $\xi \in [0, 1]^n$ such that $(\xi, \, [m + \ld, \, x]) \in Y$,
we then have
\begin{align*}
\| h (\xi, \, [m + \ld, \, x]) \|_{\I}
& \leq \| h (0, \, [m + \ld, \, x]) \|_{\I}
    + \| h (0, \, [m + \ld, \, x]) - h (\xi, \, [m + \ld, \, x]) \|_{\I}
       \\
& < 1 - 2^{- m - 2} + L \| \xi \|_{\I}
  \leq 1 - 2^{- m - 2} + L \ep (m + \ld)
  < 1.
\end{align*}
Therefore $h (Y) \S [-1, 1]^N$.

Combining the inequality~(\ref{Eq:111108bb})
with the property~(\ref{L:SpecialDR:3})
of~$f$,
we find that
$q_0 (y) \geq m$ implies
\[
\big\{ h (0, y) \colon
   {\mbox{$y \in Y^{(0)}$ with $q_0 (y) = \ld$}} \big\}
\]
is $\big( 2^{- m + 1} + 2^{- m - 2} \big)$-dense in $[-1, 1]^N$.
So certainly
\[
\big\{ h (\xi, y) \colon
   {\mbox{$(\xi, y) \in Y$ with $q_0 (y) = \ld$}} \big\}
\]
is $\big( 2^{- m + 1} + 2^{- m - 2} \big)$-dense in $[-1, 1]^N$.
Condition~(\ref{D:DR:EpD}) of Definition~\ref{D:DR} now follows
easily.
\end{proof}

\begin{lem}\label{L:ExtDR}
Let $P_0$ and $P$ be
noncompact second countable locally compact Hausdorff spaces.
Let $(Y_0, Z, \rh, h_0, q_0)$ be a \dr{} for $P_0$.
Let $S$ be a closed subset of $P_0$ such that
the inclusion $j \colon S \to P_0$ is proper,
and let $g \colon S \to P$ be proper, open, \ct, and surjective.
Set $Y = q_0^{-1} (S)$,
let $i \colon Y \to Y_0$ be the inclusion,
set $h = h_0 \circ i$,
and set $q = g \circ \big( q_0 |_Y \big)$.
Then $(Y, Z, \rh, h, q)$ is a \dr{} for~$P$.
\end{lem}

The following diagram shows the relationships among the spaces
in this lemma.
\[
\begin{CD}
Y @>{i}>> Y_0 @>{h_0}>> Z   \\
@V{q_0 |_Y}VV  @VV{q_0}V &        \\
S @>{j}>> P_0  \\
@V{g}VV  & {} \\
P
\end{CD}
\]

\begin{proof}[Proof of Lemma~\ref{L:ExtDR}]
It is clear that
$Y$ is a second countable locally compact Hausdorff space,
that $h$ is injective,
and that $Z \SM h (Y)$ is dense in~$Z$.

We claim that $q_0 |_Y \colon Y \to S$
is proper, open, \ct, and surjective.
Continuity and surjectivity are immediate.
For properness,
let $K \S S$ be compact.
Then $\big( q_0 |_Y \big)^{-1} (K) = q_0^{-1} (K)$
is compact because $q$ is proper.
Finally, let $U \S Y$ be open.
Choose an open subset $U_0 \S Y_0$ such that $U = Y \cap U_0$.
We have $q_0 (U_0) \cap S \S q_0 (U)$,
because if $y \in U_0$ and $q_0 (y) \in S$,
then $y \in Y$ by the definition of~$Y$.
So $q_0 (U) = q_0 (U_0) \cap S$,
which is open in $S$ because $q_0$ is open.
This proves the claim,
and it follows by considering compositions
that $q$ is proper, open, \ct, and surjective.

It remains to verify Condition~(\ref{D:DR:EpD}) of
Definition~\ref{D:DR}.
Let $\ep > 0$.
Choose a compact set $K_0 \S P_0$
such that for all $x \in P_0 \SM K_0$,
the set $h_0 (q_0^{-1} (x))$ is $\ep$-dense in~$Z$.
Then $K_0 \cap S$ is compact because the inclusion of $S$ in $Y_0$
is proper,
so $K = g (K_0 \cap S)$ is a compact subset of~$P$.
Let $x \in P \SM K$.
Choose $x_0 \in S$ such that $g (x_0) = x$.
Then $x_0 \not\in K_0$.
Therefore $h_0 (q_0^{-1} (x_0))$ is $\ep$-dense in~$Z$.
Since $h_0 (q_0^{-1} (x_0)) \S h (q^{-1} (x))$,
it follows that $h (q^{-1} (x))$ is $\ep$-dense in~$Z$.
\end{proof}

\begin{lem}\label{L:DRFor1dProd}
Let $X$ be a compact metric space with finite covering dimension
$\dim (X)$.
Then there exists a \dr{} for the space $P = X \times [0, \I)$.
\end{lem}

\begin{proof}
By the Menger-N\"{o}beling Theorem
(Theorem~V~2 of~\cite{HW};
see page~42 of~\cite{HW} for the definition of $I_m$),
the space $X$
is homeomorphic to a subset of $[0, 1]^{2 \dim (X) + 1}$.
Therefore it suffices to prove the result for
a closed subset $X \S [0, 1]^n$ for $n \in \N$.

Apply Lemma~\ref{L:SpecialDR}
to find a \dr{} $(Y_0, Z, \rh, h_0, q_0)$
for $P_0 = [0, 1]^n \times [0, \I)$.
Let
\[
S = X \times [0, \I) \S [0, 1]^n \times [0, \I).
\]
The inclusion of $X \times [0, \I)$ in $[0, 1]^n \times [0, \I)$
is clearly proper.
Apply Lemma~\ref{L:ExtDR} with $P_0$, $(Y_0, Z, \rh, h_0, q_0)$,
and $S$ as given, and with $g = \id_P$.
\end{proof}

\begin{lem}\label{L:DRForProd}
Let $X$ be any compact metric space.
Then there exists a \dr{} for the space $P = X \times [0, \I)$.
\end{lem}

\begin{proof}
Corollary~3.8 of~\cite{KR} provides a
compact metric space $C$ with $\dim (C) \leq 1$
and an open \ct{} surjective map $p \colon C \to X$.
Apply Lemma~\ref{L:DRFor1dProd}
to find a \dr{} $(Y_0, Z, \rh, h_0, q_0)$
for $P_0 = C \times [0, \I)$.
The map $g = p \times \id_{[0, \I)}$ is proper since $C$ is compact.
Apply Lemma~\ref{L:ExtDR} with $P_0$, $(Y_0, Z, \rh, h_0, q_0)$,
and $g$ as given,
and with $S = P_0$.
\end{proof}

\begin{prp}\label{P:ExistDR}
Let $P$ be
a noncompact second countable locally compact Hausdorff space.
Then there exists a \dr{} for~$P$.
\end{prp}

\begin{proof}
Let $Q = [0, 1]^{\N}$ be the Hilbert cube.
Let $P^+ = P \cup \{ \I \}$
be the one point (Alexandroff) compactification of~$P$.
Then $P^+$ is a compact metric space.
Therefore there is an embedding $i_0 \colon P^+ \to Q$.
Since $P^+$ is metrizable,
there is a \cfn{} $f \colon P^+ \to [0, \I]$
such that $f (\I) = \I$ and $f (x) < \I$ for all $x \in P$.
Define $i \colon P \to Q \times [0, \I)$
by $i (x) = \big( i_0 (x), \, f (x) \big)$ for $x \in P$.
Then $i$ is the restriction of the function
$j \colon P^+ \to Q \times [0, \I]$ given by the same formula.
Since $j$ is injective and $P^+$ is compact,
$j$ is a homeomorphism onto its image.
Therefore also $i$ is a homeomorphism onto its image.

We claim that $i$ is proper.
Let $K \S Q \times [0, \I)$ be compact.
Then there is $N$ such that $K \S Q \times [0, N]$.
So $i^{-1} (K) \S f^{-1} ([0, N])$,
which is a compact subset of $P^+$ not containing~$\I$,
so is therefore a compact subset of~$P$.
Since $i^{-1} (K)$ is closed in~$P$,
it is also compact.
The claim is proved.

Apply Lemma~\ref{L:DRForProd}
to find a \dr{} $(Y_0, Z, \rh, h_0, q_0)$
for $P_0 = Q \times [0, \I)$.
Now apply Lemma~\ref{L:ExtDR} with $P_0$ and $(Y_0, Z, \rh, h_0, q_0)$
as given, with $S = i (P)$, and with $g = i^{-1}$.
\end{proof}

\begin{cor}\label{C-111109aa}
Let $P$ be
a noncompact second countable locally compact Hausdorff space.
Then the space $\Xi (P)$ of Definition~\ref{D:X}
satisfies the hypotheses on~$X$ in Theorem~\ref{T:ExistA}.
\end{cor}

\begin{proof}
Combine Lemma~\ref{L:Top}, Lemma~\ref{L:PtC},
Proposition~\ref{P:ExistDR}, Proposition~\ref{P:AppOfDR},
and Lemma~\ref{L:VerPseudo}.
\end{proof}

\begin{cor}\label{C-111109bb}
Let $P$ be
a noncompact second countable locally compact Hausdorff space.
Then there exists a separable nuclear \ca~$A$
such that $\Prim (A) \cong \Xi (P)$
and $K \otimes \OT \otimes A \cong A$.
\end{cor}

\begin{proof}
Combine Theorem~\ref{T:ExistA} and Corollary~\ref{C-111109aa}.
\end{proof}

\section{Minimal actions
   on unital nonsimple prime AF~algebras}\label{Sec:AF}

\indent
Corollary~\ref{C-111109bb}
shows that if $P$ is
a noncompact second countable locally compact Hausdorff space,
then there is a separable nuclear \ca~$A$
with primitive ideal space $\Xi (P)$.
In Section~\ref{Sec:2b},
we will make this result equivariant for suitable actions
of second countable locally compact noncompact groups.
The existence of minimal actions of such groups
on nonsimple prime \ca{s} will follow.

The algebras we get this way
are all nonunital and $\OT$-stable.
For the specific cases $G = \Z$ and $G = F_n$,
the free group on $n$ generators,
in this section we use the classification theory of AF~algebras
to construct unital nonsimple prime AF~algebras
with minimal actions of~$G$.
Besides its intrinsic interest,
this result and its ingredients have two consequences
for the more general constructions of Section~\ref{Sec:2b},
not given by the methods used there.
For $G = \Z$ and $G = F_n$, it is possible to make some of the algebras
constructed there unital,
and when $G$ is discrete,
some of the algebras constructed there have real rank zero.

The paper~\cite{Brt} contains related but different examples.
In Theorem~3.1 of \cite{Brt},
it is shown that if $H$ is a nontrivial connected
locally compact abelian group
(or, more generally, the product of such a group
and any other locally compact abelian group),
and $\bt \colon H \to \Aut (B)$ is any product type action of~$H$
on a UHF~algebra~$B$,
then $C^* (H, B, \bt)$ is not simple.
There are, however, many cases in which $C^* (H, B, \bt)$ is prime.
See Section~4 of~\cite{Brt} for a detailed analysis of one such
example,
the gauge action of~$S^1$
on the CAR algebra, which is the $2^{\infty}$~UHF algebra.
Proposition~3.4 of~\cite{Brt} shows that this can also happen
for compact totally disconnected groups.
(This phenomenon is not possible for discrete abelian groups.
See the discussion after Corollary~\ref{C:ExoticCptAb}.)

In any such case,
the dual action $\af = {\widehat{\bt}}$ of $G = {\widehat{H}}$
on $A = C^* (H, B, \bt)$
is a minimal action on a nonsimple prime \ca,
and moreover $C^* (G, A, \af)$ is a simple AF~algebra by Takai duality.
When $G$ is discrete,
that is, when $H$ is compact, $A$ is even an AF~algebra.

What we do differs in three ways.

First, our choice of groups is different.
We don't get actions of locally compact abelian groups other than~$\Z$,
but we do get actions of the nonabelian groups~$F_n$,
including one (see Example~\ref{E-FnMin508} below)
which is no longer minimal when restricted to abelian subgroups.

Second, our choices of~$A$ are unital.
The construction using~\cite{Brt} can never give a unital algebra,
because $A$ is a crossed product by a group which is not discrete.
In fact, it is not possible for the crossed product of any
action $\af$ of~$\Z$ on a unital \ca{} to be~AF.
Indeed, if $A$ is not stably finite,
then neither is the crossed product.
If $A$ is stably finite,
the Pimsner-Voiculescu exact sequence~\cite{PV}
implies that $K_1$ of the crossed product is nonzero.
(The kernel of $\id_A - (\af_{-1})_*$ is nonzero
because it contains $[1]$,
and there is a surjective map from $K_1 (C^* (\Z, A, \af))$
to this kernel.)

Third,
we can prescribe the primitive ideal space within a certain class.
This is how Example~\ref{E-FnMin508}
is shown to have the properties that it does.
It is far from clear how to use the construction based on~\cite{Brt}
to obtain a prescribed primitive ideal space.

Even though it is not visible in our construction or proofs,
the correct choice of the $K_0$-group used in this section
was worked out
for the specific example of $\Z$ acting on~$\Z$ by translation
using the recipe implicit in the proofs in~\cite{BE}.

\begin{dfn}\label{D:MinActCSt}
Let $A$ be a \ca, let $G$ be a group,
and let $\af \colon G \to \Aut (A)$ be an action of $G$ on~$A$.
Then $\af$ is {\emph{minimal}} if the only $\af$-invariant ideals
of $A$ are $0$ and~$A$.
We say that an automorphism is {\emph{minimal}} if it
generates a minimal action of~$\Z$.
\end{dfn}

Thus, an action of~$G$ on $C_0 (X)$ is minimal \ifo{} the
corresponding action of $G$ on $X$ is minimal.
More generally, an action of $G$ on $A$ is minimal \ifo{} the
corresponding action of $G$ of $\Prim (A)$ is minimal.
(We use here the usual sense of minimality,
namely that there exist no nontrivial closed invariant subsets,
even when the space is not Hausdorff.)

An action of $G$ on $A$ is minimal
\ifo{} $A$ is $G$-simple in the sense used,
for example, in~3.3 of~\cite{OPd1}.

We give a fairly general construction.
The basic example for the following is $P = \Z$,
$H = \Z$ acting by translation,
$\mu$ is counting measure,
and $\Dt = \Q$.

\begin{dfn}\label{D:G}
Fix the following objects:
\begin{enumerate}
\item\label{D-G-Space}
A totally disconnected
noncompact second countable locally compact Hausdorff space~$P$.
\item\label{D-G-Homeo}
A \ct{} action of a discrete group~$H$ on~$P$
such that there are no nonempty compact $H$-invariant subsets of~$P$.
\item\label{D-G-Meas}
An $H$-invariant Borel measure $\mu$ on~$P$
(not required to be finite)
such that $\mu (U) > 0$ for every nonempty open set $U \S P$.
\item\label{D-G-Exh}
A sequence $L_1 \S L_2 \S \cdots \S P$
of nonempty compact open subsets
such that $\bigcup_{n = 1}^{\I} L_n = P$.
\item\label{D-G-Group}
A countable subgroup $\Dt \S \R$ such that $1 \in \Dt$
and such that if $r \in \Dt$,
then $\mu (L) r \in \Dt$ for every compact open subset $L \S P$
and $\mu (L_n)^{-1} r \in \Dt$ for all $n \in \N$.
\end{enumerate}
Give $\Dt$ the discrete topology.
Let $\Gm_0$ be the group
consisting of all \ct{} functions $\et \colon P \to \Dt$
with compact support such that $\int_P \et \, d \mu = 0$.
Let $\Gm$ be the group
consisting of all functions $\et \colon \Z \to \Dt$
for which there is $r \in \Dt$ such that,
regarding $r$ as a constant function on $\Z$,
we have $\et - r \in \Gm_0$.
Let $\Gm_{+} \S \Gm$ be the set of all $\et \in \Gm$
such that $\et (x) \geq 0$ for all $x \in P$.
We take the positive elements of $\Gm$ to be $\Gm_{+}$.
We identify $\Dt$ with the subgroup of~$\Gm$
consisting of the constant functions.
\end{dfn}

When constructing a \ca{}
with an action of a group~$G$,
we usually take $H = G$ in Definition~\ref{D:G},
but this is not necessary.

The following results address the existence of the data
required for Definition~\ref{D:G}.

\begin{lem}\label{L-Ctbl505}
Let $P$ be a totally disconnected
second countable locally compact Hausdorff space.
Then the collection of compact open sets in~$P$
is countable.
\end{lem}

\begin{proof}
Let ${\mathcal{U}}$ be a countable base for the topology of~$P$.
Let $L \S X$ be compact and open.
Being open, $L$ is a union of elements of~${\mathcal{U}}$.
Since $L$ is compact,
$L$ is in fact a finite union of elements of~${\mathcal{U}}$.
There are only countably many
finite unions of elements of~${\mathcal{U}}$.
\end{proof}

\begin{cor}\label{C-Exh505}
Let $P$ be a totally disconnected
second countable locally compact Hausdorff space.
Then there exists a sequence $L_1 \S L_2 \S \cdots \S P$
of nonempty compact open subsets
such that $\bigcup_{n = 1}^{\I} L_n = P$.
\end{cor}

\begin{proof}
By Lemma~\ref{L-Ctbl505},
there is an enumeration
$U_1, U_2, \ldots$ of the compact open sets in~$P$.
For $n \in \N$, define $L_n = \bigcup_{k = 1}^n U_k$.
Since $P$ is totally disconnected and locally compact,
every point of~$P$ is contained in $\bigcup_{n = 1}^{\I} L_n$.
\end{proof}

We summarize for convenient reference some
basic facts which follow from Definition~\ref{D:G}.

\begin{lem}\label{L:BasicOnG}
Let the notation be as in Definition~\ref{D:G}.
\begin{enumerate}
\item\label{L:BasicOnG:0}
Let $\et \in \Gm$.
Then the range of $\et$ is finite.
\item\label{L:BasicOnG:1}
Let $\et \in \Gm$.
Then there is a unique $r \in \Dt$,
which we call $\om (\et)$,
such that $\et - r$ has compact support.
\item\label{L:BasicOnG:2}
Let $\et \in \Gm$.
Suppose $r \in \Dt$ and $\et - r$ has compact support.
Then $\et - r \in \Gm_0$.
\item\label{L:BasicOnG:3}
Let $\et \colon P \to \Dt$ be a \ct{} function and let $r \in \Dt$.
Suppose that $\et - r$ has compact support.
Then \tfae:
\begin{enumerate}
\item\label{L:BasicOnG:3a}
$\et \in \Gm$.
\item\label{L:BasicOnG:3b}
For some nonempty compact open $F \S P$ containing $\supp (\et - r)$,
we have
\[
\frac{1}{\mu (F)} \int_F \et \, d \mu = r.
\]
\item\label{L:BasicOnG:3c}
For any nonempty compact $F \S P$ containing $\supp (\et - r)$
such that $\mu (F) \neq 0$,
we have
\[
\frac{1}{\mu (F)} \int_F \et \, d \mu = r.
\]
\end{enumerate}
\item\label{L:BasicOnG:4}
Let $\et \in \Gm_{+} \SM \{ 0 \}$.
Then $\big\{ x \in P \colon \et (x) = 0 \big\}$
is a compact open set,
and there is $t \in (0, \I)$
such that $\et (x) \geq t$ for all $x \not\in F_0$.
\end{enumerate}
\end{lem}

\begin{proof}
Parts (\ref{L:BasicOnG:0}), (\ref{L:BasicOnG:1}), (\ref{L:BasicOnG:2}),
and~(\ref{L:BasicOnG:3}) are immediate.
For~(\ref{L:BasicOnG:4}),
choose $r \in \Dt$ such that $\et - r \in \Gm_0$.
We then have $\et \neq 0$, $\et (x) \geq 0$ for all $x \in P$,
and $\int_{P} (\et - r ) \, d \mu = 0$.
Since $\mu (U) > 0$ for all nonempty open subsets $U \S P$,
it follows that $r \neq 0$.
The result now follows from (\ref{L:BasicOnG:0}), continuity of~$\et$,
and the fact that $\et - r$ has compact support.
\end{proof}

\begin{lem}\label{L:Riesz}
Let the notation be as in Definition~\ref{D:G}.
Then the triple $(\Gm, \Gm_{+}, 1)$
is a Riesz group with order unit.
For $h \in H$ define
$\bt_h \colon \Gm \to \Gm$
by $\bt_h (\et) (x) \et (h^{-1} x)$
for $\et \in \Gm$ and $x \in P$.
Then $\bt$ is an action of $H$ on $(\Gm, \Gm_{+}, 1)$
and the function $\om \colon \Gm \to \R$
of Lemma~\ref{L:BasicOnG}(\ref{L:BasicOnG:1})
is an $H$-invariant state on $(\Gm, \Gm_{+}, 1)$.
\end{lem}

\begin{proof}
Clearly $\Gm$ is a torsion free abelian group
and $\Gm_{+}$ is a subsemigroup containing~$0$.

We show that $\Gm_{+}$ generates~$\Gm$.
Let $\et \in \Gm$.
Then $\et$ is bounded by Lemma \ref{L:BasicOnG}(\ref{L:BasicOnG:0}),
so there is $t \in \Dt$ such that $t > \et (x)$ for all $x \in P$.
Then $t, \, t - \et \in \Gm_{+}$ and $t - (t - \et) = \et$.
Similarly, we can easily check that any element of $\Gm$
is dominated by an integer multiple of~$1$.
Thus $1$ is an order unit.

It remains to prove Riesz interpolation.
Let $\rh_1, \rh_2, \sm_1, \sm_2 \in \Gm$ satisfy
$\rh_j \leq \sm_k$ for $j, k \in \{ 1, 2 \}$.
Choose $r_1, r_2, s_1, s_2 \in \Dt$ and a compact set $F_0 \S P$
such that
\[
\supp (\rh_1 - r_1), \, \supp (\rh_2 - r_2),
 \, \supp (\sm_1 - s_1), \, \supp (\sm_2 - s_2) \S F_0.
\]
Choose $n$ such that $F_0 \S L_n$.
Define $\et_0 \colon L_n \to \Dt$
by $\et_0 (x) = \max (\rh_1 (x), \, \rh_2 (x))$ for $x \in L_n$.
Set
\[
t = \frac{1}{\mu (L_n)} \int_{L_n} \et_0 \, d \mu.
\]
The expression $\int_{L_n} \et_0 \, d \mu$
is a finite linear combination of
elements of $\Dt$ with coefficients of the form $\mu (L)$
for compact open subsets $L \S P$.
The hypotheses on~$\Dt$ therefore
imply that $\int_{L_n} \et_0 \, d \mu \in \Dt$,
and then that $t \in \Dt$.
Now define $\et \colon P \to \Dt$ by
\[
\et (x)
 = \begin{cases}
     \et_0 (x)   &  x \in L_n \\
     t           &  x \in P \SM L_n.
\end{cases}
\]
Clearly $\supp (\et - t)$ is compact and contained in~$L_n$,
and
\[
\frac{1}{\mu (L_n)} \int_{L_n} (\et - t) \, d \mu = 0,
\]
so $\et \in \Gm$ by Lemma~\ref{L:BasicOnG}(\ref{L:BasicOnG:3}).

Clearly
\[
\rh_1 (x), \, \rh_2 (x) \leq \et_0 (x) \leq \sm_1 (x), \, \sm_2 (x)
\]
for $x \in L_n$.
Moreover, using Lemma~\ref{L:BasicOnG}(\ref{L:BasicOnG:3}),
for $j = 1, 2$ we have
\[
r_j = \frac{1}{\mu (L_n)} \int_{L_n} \rh_j \, d \mu
    \leq \frac{1}{\mu (L_n)} \int_{L_n} \et_0 \, d \mu
    = t,
\]
and similarly $t \leq s_1, s_2$.
Therefore $\rh_1, \rh_2 \leq \et \leq \sm_1, \sm_2$.
This completes the verification that
$(\Gm, \Gm_{+}, 1)$ is a Riesz group.

It is immediate that $\bt_h$ is an automorphism
of $(\Gm, \Gm_{+}, 1)$ for $h \in H$,
and it is immediate that $h \mapsto \bt_h$
is a \hm{} and that $\om$ is an $H$-invariant state.
\end{proof}

\begin{cor}\label{C:AFG}
Let the notation be as in Definition~\ref{D:G}.
Then there exists a unital AF~algebra $A$ such that
$(K_0 (A), \, K_0 (A)_{+}, \, [1]) \cong (\Gm, \Gm_{+}, 1)$,
and such that for every $h \in H$
there exists an automorphism $\af_h \in \Aut (A)$
which induces on $\Gm$ the map
$(\af_h)_* (\et) (x) = \et ( h^{-1} x)$
for $\et \in \Gm$ and $x \in P$.
Moreover,
there is a unique tracial state $\ta$ on~$A$
such that
the induced map $\ta_* \colon \Gm \to \R$
is the map $\om$ of Lemma~\ref{L:BasicOnG}(\ref{L:BasicOnG:1}),
and $\ta \circ \af_h = \af_h$ for all $h \in H$.
\end{cor}

We do not claim that $h \mapsto \af_h$
is a group \hm.

\begin{proof}[Proof of Corollary~\ref{C:AFG}]
The existence of $A$
follows from Theorem~2.2 of~\cite{EHS} and Theorem~5.5 of~\cite{Ell1}.
Using the $H$-invariance of~$\mu$,
it is clear that the claimed formula for $(\af_h)_*$
defines an automorphism of $(\Gm, \Gm_{+}, 1)$,
and the existence of $\af_h$
now follows from Theorem~4.3 of~\cite{Ell1},
together with the observation that the isomorphism constructed there
induces the given isomorphism of dimension ranges.

The existence and uniqueness of~$\ta$
follow from Theorem IV.5.3 of~\cite{Dv}.
That $\ta$ is $\af_h$-invariant for all $h \in H$
follows from uniqueness of~$\ta$ and $H$-invariance of~$\om$.
\end{proof}

\begin{rmk}\label{R:ZNoRiesz}
In Definition~\ref{D:G},
take $P = \Z$,
take $H = \Z$,
let the action of $H$ on~$P$ be generated by
$x \mapsto x + 1$,
and take $\mu$ to be counting measure.
By Lemma~\ref{L:Riesz},
the choices
$\Dt = \Q$, $\Dt = \Z \big[ \tfrac{1}{2} \big]$,
and many others, make $(\Gm, \Gm_{+})$ a Riesz group.
A slightly more involved argument,
which we omit,
shows that one gets a Riesz group using any subgroup of~$\Q$
which contains~$1$ and is dense in~$\R$.
(We don't know a convenient analogous statement
for general $P$.)
However, with $P$ and $h$ as given,
the choice $\Dt = \Z$ does not give a Riesz group.
Letting $\dt_x$ be the function which is $1$ at~$x$
and zero elsewhere,
Riesz interpolation can be shown to fail
for
\[
\rh_1  = 1,
\,\,\,\,\,\,
\rh_2 = 1 - \dt_0 + \dt_1,
\,\,\,\,\,\,
\sm_1 = 2,
\andeqn
\sm_2 = 2 - \dt_0 + \dt_1.
\]
\end{rmk}

Recall (see the beginning of Section~IV.5 of~\cite{Dv})
that an order ideal
in a partially ordered abelian group $(\Gm, \Gm_{+})$
is a subgroup $N$ which is generated by $N \cap \Gm_{+}$
and such that $\et \in N$ and $0 \leq \mu \leq \et$
imply $\mu \in N$.

\begin{lem}\label{L:IdealsG}
Let $(\Gm, \Gm_{+}, 1)$ be as in Definition~\ref{D:G}.
For a compact subset $F \S P$,
let $N_F \S \Gm$ be the set of all
$\et \in \Gm$ such that $\et (x) = 0$ for all $x \in F$.
Then $F \mapsto N_F$ is a bijection from the set of compact
subsets of $P$ to the set of nonzero order ideals of~$\Gm$.
\end{lem}

\begin{proof}
One immediately checks that $N_F$ as defined
is in fact an order ideal.

We first prove that $F \mapsto N_F$ is injective.
So suppose $F_1 \neq F_2$.
\Wolog{} there is $x \in F_1$ such that $x \not\in F_2$.

We claim that there is $h \in H$
such that $h x \not\in \{ x \} \cup F_2$.
Otherwise, ${\overline{H x}} \S \{ x \} \cup F_2$,
and is therefore a nonempty compact $H$-invariant subset of~$P$.
This contradicts the hypothesis on the action of~$H$,
and proves the claim.

Since $P \SM F_2$ is open and $h$ is \ct,
there is a compact open set $K \S P$
such that
\[
x \in K,
\,\,\,\,\,\,
K \cap F_2 = \E,
\,\,\,\,\,\,
h K \cap F_2 = \E,
\andeqn
K \cap h K = \E.
\]
Now define $\et \colon P \to \Dt$ by
\[
\et (x)
 = \begin{cases}
   1  & x \in K
        \\
   -1 & x \in h K
       \\
   0  & x \in P \SM ( K \cup h K ).
\end{cases}
\]
Then $\et \in \Gm_0$ because $\mu$ is $h$-invariant.
We have $\et \in N_{F_2}$ but $\et \not\in N_{F_1}$,
so $N_{F_2} \neq N_{F_1}$.

We now prove that $F \mapsto N_F$ is surjective.
Let $N \S \Gm$ be a nonzero order ideal of~$\Gm$.
Set
\[
F =
 \big\{ x \in P \colon
  {\mbox{$\et (x) = 0$ for all $\et \in N$}} \big\}.
\]
Since $N$ is nonzero and an order ideal,
there is a nonzero positive element $\rh \in N$.
By Lemma \ref{L:BasicOnG}(\ref{L:BasicOnG:4}), the set
\[
F_0 = \big\{ x \in P \colon \rh (x) = 0 \big\}
\]
is a compact open set.
Clearly $F \S F_0$, so $F$ is compact.

Clearly $N \S N_F$,
so it suffices to prove that $N_F \S N$.
Indeed, it suffices to prove that $N_F \cap \Gm_{+} \S N$.
So let $\et \in N_F \cap \Gm_{+}$.
Set
\[
E = \big\{ x \in P \colon \et (x) = 0 \big\},
\]
which is a compact open set
by Lemma \ref{L:BasicOnG}(\ref{L:BasicOnG:4}).
Therefore
$K = F_0 \cap (P \SM E)$ is a compact open set
such that $K \cap F = \E$.
For every $y \in K$,
there is therefore $\nu \in N$ such that $\nu (y) \neq 0$,
so there is $\rh_y \in N \cap \Gm_{+}$ such that
$\rh_y (y) \neq 0$.
Then in fact $\rh_y (y) > 0$.
Set $U_y = \{ x \in P \colon \rh_y (x) > 0 \}$.
The sets $U_y$ are open and cover~$K$,
so there is a finite set $S \S K$
such that $K \S \bigcup_{y \in S} U_y$.
Set $\sm = \rh + \sum_{y \in S} \rh_y$.
Then $\sm \in N \cap \Gm_{+}$ and $\sm (x) > 0$
for all $x \in K \cup ( P \SM F_0)$,
hence for all $x \in P \SM E$.
By Lemma \ref{L:BasicOnG}(\ref{L:BasicOnG:4}),
there is $t \in (0, \I)$
such that $\sm (x) \geq t$ for all $x \in P \SM E$.

By Lemma \ref{L:BasicOnG}(\ref{L:BasicOnG:0}),
there is $s > 0$ such that $\et (x) \leq s$ for all $x \in P$.
Choose $m \in \N$ such that $m t \geq s$.
We claim that $m \rh \geq \et$.
Indeed,
if $x \in P \SM E$,
then $m \rh (x) \geq m t \geq s \geq \et (x)$,
while if $x \in E$ then $m \rh (x) \geq 0 = \et (x)$.
The claim follows.
Since $N \cap \Gm_{+}$ is hereditary,
it follows that $\et \in N$.
This completes the proof.
\end{proof}

\begin{cor}\label{C:AFPrimZ}
Let $A$, $H$, and $\af_h$ for $h \in H$ as in Corollary~\ref{C:AFG}.
Interpret $h$ as a homeomorphism of~$P$.
Then $\Prim (A) \cong \Xi (P)$,
in such a way that for every $h \in H$,
and interpreting $h$ as a \hme{} of~$P$,
the automorphism $\af_h$
induces the \hme{} $\Xi (h) \colon \Xi (P) \to \Xi (P)$
of Proposition~\ref{P:Funct}.
\end{cor}

\begin{proof}
The order ideals of $K_0 (A)$
are in bijective order preserving correspondence
with the ideals of~$A$.
(See Proposition IV.5.1 of~\cite{Dv}.)
Lemma~\ref{L:IdealsG}
therefore implies that the ideals of~$A$
are in bijective order preserving correspondence
with the open subsets of $\Xi (\Z)$.
The identification $\Prim (A) \cong \Xi (P)$ follows.
For $h \in H$,
the automorphism $(\af_h)_*$ acts on the order ideals of $K_0 (A)$,
as described in Lemma~\ref{L:IdealsG},
via $(\af_h)_* (N_F) = N_{h (F)}$ for compact sets $F \S P$.
The corresponding open set in $\Xi (P)$
is $\Xi (P) \SM F$.
So the action on $\Xi (P)$ is $\Xi (h)$.
\end{proof}

\begin{cor}\label{C-FnAct508}
Let $n \in \N$,
let $P$ be a nonempty totally disconnected
noncompact second countable locally compact Hausdorff space,
let $\mu$ be a Borel measure on~$P$
such that $\mu (U) > 0$ for every nonempty open set $U \S P$,
and let $h_1, h_2, \ldots, h_n \colon P \to P$
be measure preserving homeomorphisms.
Assume that there is no nonempty compact subset $L \S P$
such that $h_j (L) = L$ for $j = 1, 2, \ldots, n$.
Then there is a unital nonsimple prime AF algebra~$A$
and a minimal action $\af \colon F_n \to \Aut (A)$
of the free group $F_n$ on $n$~generators
such that $\Prim (A)$ is equivariantly homeomorphic to $\Xi (P)$
with the action of $F_n$ which sends the generators to
$\Xi (h_1), \, \Xi (h_2), \, \ldots, \, \Xi (h_n)$.
\end{cor}

Some suitable spaces $P$ and actions of $F_n$ on~$P$
can be obtained by taking $G = F_n$
in the first three parts of Example~\ref{E-111109dd}.

\begin{proof}[Proof of Corollary~\ref{C-FnAct508}]
In Definition~\ref{D:G},
take $H = F_n$.
Let $\mu$ be as in the statement.
The existence of $L_1 \S L_2 \S \cdots \S P$
and of a suitable countable subgroup~$\Dt$ as in Definition~\ref{D:G}
follow from Corollary~\ref{C-Exh505} and Lemma~\ref{L-Ctbl505}.

Apply Corollary~\ref{C:AFG} to $h_1, h_2, \ldots, h_n$,
regarded as generators of~$F_n$,
getting automorphisms $\af_1, \af_2, \ldots, \af_n$.
These generate an action $\af \colon F_n \to \Aut (A)$.
The $F_n$-equivariant homeomorphism
$\Prim (A) \cong \Xi (P)$ follows from Corollary~\ref{C:AFPrimZ}.
The algebra $A$ is prime by Corollary~\ref{L-Prime426},
but not simple because $P \neq \E$.
Minimality of~$\af$ follows from Lemma~\ref{L:MinXiP}.
\end{proof}

\begin{cor}\label{C-ZAF428}
There exists a unital nonsimple prime AF algebra~$A$
such that there is a minimal effective action
$\af \colon \Z \to \Aut (A)$
and such that $\Prim (A)$ is equivariantly homeomorphic to $\Xi (\Z)$.
\end{cor}

\begin{proof}
Apply Corollary~\ref{C-FnAct508}
with $n = 1$,
with $P = \Z$,
with $h_1 (x) = x + 1$ for $x \in \Z$,
and with $\mu$ being counting measure.
\end{proof}

\begin{qst}\label{Q-RokhCP-430}
Let $\af \colon \Z \to \Aut (A)$ be as in Corollary~\ref{C-ZAF428}.
Is $C^* (\Z, A, \af)$ necessarily an AT~algebra?
What if, in addition, $\af$~has the Rokhlin property?
(This can always be arranged,
by tensoring with an action of $\Z$
on a UHF algebra with the Rokhlin property.)
\end{qst}

\begin{exa}\label{E-FnMin508}
Let $n \geq 2$.
Then there is a unital nonsimple prime AF algebra~$A$
such that there is a minimal action
$\af \colon F_n \to \Aut (A)$
which is effective on $\Prim (A)$,
and such that,
for any abelian subgroup $H \S F_n$,
the restricted action $\af |_{H}$ is not minimal.

To see this, let $S$ be the set of abelian subgroups of~$F_n$.
Subgroups of free groups are free,
so every abelian subgroup is singly generated.
Therefore $S$ is countable.
In Corollary~\ref{C-FnAct508},
take $P = \coprod_{H \in S} F_n / H$,
take $\mu$ to be counting measure,
and let $h_1, h_2, \ldots, h_n \colon P \to P$
be the generators of the obvious action of $F_n$ on $P$
given by left translation on $F_n / H$ for each $H \in S$.
There are no compact $F_n$-invariant subsets
since the $F_n$-orbits are the infinite sets $F_n / H$ for $H \in S$.
So Corollary~\ref{C-FnAct508} gives a minimal action
$\af \colon F_n \to \Aut (A)$.
It is clearly effective on $\Prim (A)$.

Now let $H \S F_n$ be an abelian subgroup.
Then the action of $H$ on $P$ has a fixed point,
so Lemma~\ref{L:MinXiP} implies that the action of $H$ on $\Prim (A)$
is not minimal.
Therefore the action of $H$ on $A$ is not minimal.
\end{exa}

\section{Minimal actions on nonsimple prime C*-algebras:
  nonexistence}\label{Sec:2a}

\indent
As discussed in the introduction, we are interested in minimal
actions on \ca{s} which are simple but not prime.
One example appears in Corollary~\ref{C-ZAF428},
and many more will appear in Theorem~\ref{T:ExistMinAct}.
Here, we describe some constraints.
In particular, there can be no such action if the group is compact,
or if the group is almost connected
(the quotient by its identity component is compact)
and the algebra has the ideal property.

If $G$ is compact (and the action is \ct),
if $A$ is separable and prime,
and if $\af \colon G \to \Aut (A)$ is minimal,
then $A$ is simple.
This, and some more general nonexistence results,
are consequences of the following lemma.

\begin{lem}\label{L:CptGMinX}
Let $G$ be a compact group
and let $X$ be a ${\mathrm{T}}_0$-space.
Suppose that $G$ acts \ct ly and minimally on ~$X$
and that there is $x_0 \in X$ such that ${\overline{ \{ x_0 \} }} = X$.
Then $X = \{ x_0 \}$.
\end{lem}

\begin{proof}
Since $X$ is a ${\mathrm{T}}_0$-space,
there can be at most one point in $X$ whose closure is~$X$.
It follows that $g x_0 = x_0$ for all $g \in G$.
The conclusion now follows from the last statement
in Lemma~2.1 of~\cite{OPd3}.
Since the rest of the argument is somewhat shorter than the
proof of Lemma~2.1 of~\cite{OPd3},
we give it anyway.

Suppose $X$ contains a point $x \neq x_0$.
Since $X$ is a ${\mathrm{T}}_0$-space
and there is no open set containing $x$ but not~$x_0$,
there must be an open set $U \S X$
such that $x_0 \in U$ but $x \not\in U$.
Let $\mu \colon G \times X \to X$ be $\mu (g, x) = g x$.
Then $\mu^{-1} (U)$ is an open set in $G \times X$ containing
$G \times \{ x_0 \}$.
Since $G \times \{ x_0 \}$ is compact
and $\mu^{-1} (U)$ is a union of product open sets,
there are open sets $W_1, W_2, \ldots, W_n \S G$
and $V_1, V_2, \ldots, V_n \S X$
such that $W_j \times V_j \S \mu^{-1} (U)$ for all~$j$
and $G \times \{ x_0 \} \S \bigcup_{j = 1}^n W_j \times V_j$.
\Wolog{} $x_0 \in V_j$ for all~$j$.
Set $V = \bigcap_{j = 1}^n V_j$.
Then also $G V$
is a $G$-invariant open subset of $X$ containing~$x_0$
and contained in~$U$.
Since $x \not\in U$, this contradicts minimality.
\end{proof}

\begin{cor}\label{C:CptGMinA}
Let $\af \colon G \to \Aut (A)$ be a minimal action of
a Hausdorff topological group $G$ on a separable prime \ca~$A$.
Following (loosely) Notation~\ref{N:PrimA},
set $H = \big\{ g \in G \colon \Prim (g) = \id_{\Prim (A)} \big\}$.
If $G / H$ is compact, then $A$ is simple.
\end{cor}

\begin{proof}
Apply Lemma~\ref{L:CptGMinX}
to the action of $G / H$ on $\Prim (A)$.
\end{proof}

Recall (see the introduction to~\cite{Ps})
that a \ca~$A$ has the {\emph{ideal property}}
if every ideal in $A$ is generated as an ideal by the \pj{s} it contains.
In particular, every \ca{} with real rank zero
has the ideal property.
(The ideal property appears to have been first
defined by Stevens.
See the introduction in~\cite{Stv}.)

\begin{prp}\label{P:NotRR0}
Let $A$ be a \ca{} with the ideal property.
Let $\af \colon G \to \Aut (A)$ be an action of
a connected Hausdorff topological group on~$A$.
Then the action of~$G$ on $\Prim (A)$ is trivial.
\end{prp}

\begin{proof}
We claim that if $J \S A$ is an ideal,
then $\af_g (J) \S J$ for all $g \in G$.
To prove this,
let $p \in J$ be a \pj.
Continuity of $g \to \af_g (p)$ and connectedness of~$G$
imply that for all $g \in G$,
the \pj{} $\af_g (p)$ is \mvnt{} to~$p$.
Thus, there is $s \in A$ such that $s^* s = p$ and $s s^* = \af_g (p)$.
Then $\af_g (p) = s p s^* \in J$.
Since the elements $\af_g (p)$, for \pj{s} $p \in J$,
generate $\af_g (J)$ as an ideal,
it follows that $\af_g (J) \S J$.

Applying the claim to $g^{-1}$,
we get $\af_g (J) = J$ for all $g \in G$.
Thus, all ideals in~$A$ are $G$-invariant.
The result follows.
\end{proof}

\begin{cor}\label{C:ConnNotRR0}
Let $G$ be a connected Hausdorff topological group.
Let $A$ be a \ca{} with the ideal property.
Suppose there exists a minimal action $\af \colon G \to \Aut (A)$.
Then $A$ is simple.
\end{cor}

\begin{proof}
This is immediate from Proposition~\ref{P:NotRR0}.
\end{proof}

In particular,
if $G$ acts minimally on a \ca~$A$ with real rank zero,
then $A$ is simple.

\begin{cor}\label{C:PrimeNotRR0}
Let $G$ be a Hausdorff topological group.
Let $G_0$ be the connected component of~$G$ containing~$1$,
and suppose $G / G_0$ is compact.
Let $A$ be a separable prime \ca{} with the ideal property.
Suppose there exists a minimal action $\af \colon G \to \Aut (A)$.
Then $A$ is simple.
\end{cor}

\begin{proof}
This is immediate from Proposition~\ref{P:NotRR0}
and Corollary~\ref{C:CptGMinA}.
\end{proof}

\section{Minimal actions on nonsimple prime C*-algebras:
  existence}\label{Sec:2b}

\indent
In this section,
we construct a variety of examples of minimal actions of groups~$G$
on nonsimple prime \ca{s}~$A$.
The algebras will be produced by applying Corollary~\ref{C-111109bb}
to $\Xi (P)$ for suitable $G$-spaces~$P$,
but some work
(based on the uniqueness part of Theorem~\ref{T:ExistA})
will be required to lift the action of $G$ on $\Xi (P)$
to an action on~$A$.

\begin{ntn}\label{N:LatN}
For a \ca~$A$,
we let ${\mathcal{I}} (A)$ denote the lattice
of closed ideals of~$A$.

For a topological space~$X$,
we let ${\mathbb{O}} (X)$ denote the lattice
of open subsets of~$X$.
\end{ntn}

\begin{rmk}\label{R:SupInf}
Recall that $I \mapsto \Prim (I)$
defines a lattice isomorphism
${\mathcal{I}} (A) \to {\mathbb{O}} ( \Prim (A) )$.
For $U \S \Prim (A)$ open,
we denote by $J (U)$ the corresponding ideal in~$A$.

Arbitrary supremums and infimums
are defined in ${\mathcal{I}} (A)$:
the infimum of a family $( J_{\ld} )_{\ld \in \Ld}$
of ideals is their intersection $\bigcap_{\ld \in \Ld} J_{\ld}$,
and the supremum is the closed ideal generated by
$\bigcup_{\ld \in \Ld} J_{\ld}$.

In ${\mathbb{O}} (X)$, arbitrary supremums (the union)
and infimums (the interior of the intersection) also exist.

The bijection
${\mathcal{I}} (A) \to {\mathbb{O}} ( \Prim (A) )$
preserves arbitrary supremums and infimums.
\end{rmk}

\begin{dfn}[Definition~2.9 of~\cite{HK}]\label{D:MAI}
Let $A$ be a \ca, and let $J \S A$ be a closed ideal.
We define the subalgebra $M (A, J)$ of the multiplier algebra $M (A)$
by
\[
M (A, J) = \{ x \in M (A) \colon x A \S J \}.
\]
\end{dfn}

Following the usual convention,
if $B$ and $C$ are subalgebras of a \ca~$A$,
we write
\[
B C
 = \spn \big( \big\{ b c \colon
    {\mbox{$b \in B$ and $c \in C$}} \big\} \big).
\]

\begin{lem}\label{R:MAIRmk}
The subalgebra $M (A, J)$ of Definition~\ref{D:MAI}
has the following properties:
\begin{enumerate}
\item\label{R:MAIRmk:1}
$M (A, J) = \{ x \in M (A) \colon A x \S J \}$.
\item\label{R:MAIRmk:2}
$M (A, J)$ is the kernel of the canonical map $M (A) \to M (A / J)$.
In particular, it is a closed ideal in $M (A)$.
\item\label{R:MAIRmk:3}
${\overline{M (A, J) A}} = M (A, J) \cap A = J$.
\item\label{R:MAIRmk:4}
$M (A, J)$ is the strict closure of $J$ in $M (A)$.
\end{enumerate}
\end{lem}

\begin{proof}
The identification of $M (A, J)$
as the kernel of the map in~(\ref{R:MAIRmk:2}) is clear.
Part~(\ref{R:MAIRmk:1}) follows,
since, for the same reason,
the kernel of the map in~(\ref{R:MAIRmk:2})
is also equal to $\{ x \in M (A) \colon A x \S J \}$.
Part~(\ref{R:MAIRmk:3}) follows
from the easily checked inclusions
\[
J \S M (A, J) \cap A
  = {\overline{(M (A, J) \cap A)^2}}
  \S {\overline{M (A, J) A}}
  \S J.
\]

We prove~(\ref{R:MAIRmk:4}).
The ideal $M (A, J)$ is strictly closed because
the map in~(\ref{R:MAIRmk:2}) is strictly \ct.
It remains to prove that $M (A, J)$ is contained
in the strict closure of~$J$.
So let $x \in M (A, J)$.
By Theorem~1 of~\cite{Ar},
we may choose an approximate identity $(e_{\ld})_{\ld \in \Ld}$
for $A$ which is quasicentral for $M (A)$.
Then $x e_{\ld} \in J$ by the definition of $M (A, J)$.
We verify that $x e_{\ld} \to x$ strictly.
Let $a \in A$.
Then $x e_{\ld} a \to x a$ because $e_{\ld} a \to a$.
Also, $\| x e_{\ld} - e_{\ld} x \| \to 0$ by quasicentrality,
and $a e_{\ld} x \to a x$ because $a e_{\ld} \to a$,
so $a x e_{\ld} \to a x$.
This is the required strict convergence.
\end{proof}

The following definition is in Remark~2.4 of~\cite{HK}.

\begin{dfn}\label{D:dtInfty}
Let $A$ be any nonzero stable \ca.
We define a unital \hm{}
$\dt_{A, \infty} \colon M (A) \to M (A)$,
well defined up to unitary equivalence
(see Lemma~\ref{L-dtIRmk}(\ref{L-dtIRmk:2})),
as follows.
Choose isometries $s_1, s_2, \ldots \in M (A)$
with orthogonal ranges and such that
$\sum_{n = 1}^{\infty} s_n s_n^*$ converges to $1$ in the strict
topology.
Then define, again with convergence in the strict topology,
\[
\dt_{A, \infty} (a) = \sum_{n = 1}^{\infty} s_n a s_n^*
\]
for $a \in M (A)$.
We write $\dt_{\infty}$ when $A$ is clear.
\end{dfn}

We give several properties of $\dt_{A, \infty}$.
The first two are well known.
The first three are in Remark~2.4 of~\cite{HK}.

\begin{lem}\label{L-dtIRmk}
Let $A$ be a nonzero stable \ca.
Then:
\begin{enumerate}
\item\label{L-dtIRmk:1}
There exist isometries $s_1, s_2, \ldots \in M (A)$
as required in Definition~\ref{D:dtInfty}.
\item\label{L-dtIRmk:2}
Any two choices of isometries $s_1, s_2, \ldots \in M (A)$
as in Definition~\ref{D:dtInfty}
give unitarily equivalent \hm{s}.
\item\label{L-dtIRmk:3}
For any \hm{} $\dt_{\infty} \colon M (A) \to M (A)$
as in Definition~\ref{D:dtInfty},
the \hm{} $\dt_{\infty} \circ \dt_{\infty}$
is unitarily equivalent to~$\dt_{\infty}$.
\item\label{L-dtIRmk:4}
Suppose $J$ is an ideal in~$A$ and $a \in M (A)$.
Then $a \in M (A, J)$ \ifo{} $\dt_{\infty} (a) \in M (A, J)$.
\end{enumerate}
\end{lem}

\begin{proof}
Part~(\ref{L-dtIRmk:1}) is easy,
using $A \cong K \otimes A$
and the existence of such isometries in $L (l^2)$.

For~(\ref{L-dtIRmk:2}),
let $s_1, s_2, \ldots$ and $t_1, t_2, \ldots$
be two sets of isometries in $M (A)$,
in each of which the ranges are orthogonal,
and such that $\sum_{n = 1}^{\infty} s_n s_n^*$
and $\sum_{n = 1}^{\infty} t_n t_n^*$ converge strictly to~$1$.
Set $u = \sum_{n = 1}^{\infty} t_n s_n^*$,
again with convergence in the strict topology.
Then
\[
u \left( \sum_{n = 1}^{\infty} s_n a s_n^* \right) u^*
  = \sum_{n = 1}^{\infty} t_n a t_n^*
\]
for all $a \in M (A)$.

Part~(\ref{L-dtIRmk:3}) is a special case of part~(\ref{L-dtIRmk:2}),
as is seen by taking
$t_1, t_2, \ldots$ to consist of all products $s_k s_l$
with $k, l \in \N$.

We prove~(\ref{L-dtIRmk:4}).
If $a \in M (A, J)$,
then
$\dt_{\infty} (a) = \sum_{n = 1}^{\infty} s_n a s_n^* \in M (A, J)$
since $M (A, J)$ is a strictly closed ideal
(Lemma~\ref{R:MAIRmk}(\ref{R:MAIRmk:2}) and~(\ref{R:MAIRmk:4})).
On the other hand,
if $\dt_{\infty} (a) \in M (A, J)$,
then $a = s_1^* \dt_{\infty} (a) s_1 \in M (A, J)$
because $M (A, J)$ is an ideal.
\end{proof}

Recall that if $A$ and $B$ are \ca{s},
with $B$ unital,
then two \hm{s} $\ph, \ps \colon A \to B$ are
{\emph{unitarily homotopic}}
(or {\emph{asymptotically unitarily equivalent}})
if there is a \ct{} path $t \mapsto u_t$ of unitaries in~$B$,
defined for $t \in [0, \I)$,
such that $\limi{t} \| u_t \ph (a) u_t^* - \ps (a) \| = 0$
for all $a \in A$.

The \hm{} called $\et_0$ in the next theorem
is called $h_0$ in~\cite{HK}.
We don't use the letter $h$ for \hm{s} here,
to avoid confusion with group elements.

\begin{thm}\label{T:HKReconstr}
Let $A$ be a separable
stable nuclear \ca.
Let $\Om \S \Open (\Prim (A))$
be a sublattice of $\Open (\Prim (A))$
which contains $\E$ and $\Prim (A)$,
and which is closed under arbitrary supremums and infimums.
Then there is an injective nondegenerate \hm{}
$\et_0 \colon A \to \M (A)$
with following properties:
\begin{enumerate}
\item\label{T:HKReconstr:1}
The infinite repeat $\delta_{\infty} \circ \et_0$
is unitarily equivalent to $\et_0$.
\item\label{T:HKReconstr:2}
For $U \in \Open (\Prim(A))$ let
$V \in \Omega$ be given by
$V = \bigcup \{ W \in \Omega \colon W \subset U \}$.
Then $\et_0 (J (V)) = \et_0 (A) \cap \M (A, \, J (U))$.
\end{enumerate}
Moreover, $\et_0$
satisfying these conditions
is uniquely determined up to unitary homotopy.

For any such \hm~$\et_0$,
the Cuntz-Pimsner algebra ${\mathcal{O}}_{\mathcal{H}}$
(Definition~1.1 of~\cite{Pm})
of the
Hilbert $A$--$A$--bimodule ${\mathcal{H}} = (A, \et_0)$
is stable and nuclear, and
it is the same as the
Toeplitz algebra ${\mathcal{T}} ({\mathcal{H}})$
(Definition~1.1 of~\cite{Pm})
of ${\mathcal{H}}$.
Moreover, the natural embedding of
$A$ into ${\mathcal{O}}_{\mathcal{H}}$ defines a
lattice isomorphism from $\Omega$
onto $\Open (\Prim ( {\mathcal{O}}_{\mathcal{H}}))$.
\end{thm}

\begin{proof}
To produce~$\et_0$,
we want to apply Proposition~2.15 of~\cite{HK} with $A = B$.
The requirement that ${\mathcal{O}}_2 \otimes A$ have a
regular commutative subalgebra is satisfied
by Theorem~6.11 of~\cite{KR}.
We use Remark~A.4 of~\cite{HK} to get an order preserving map
$\Ps \colon {\mathbb{O}} (\Prim (A)) \to {\mathbb{O}} (\Prim (A))$
whose range is exactly $\Om$ and which satisfies
conditions (I) and~(II) in Definition~1.1 of~\cite{HK}.
Moreover,
the construction in Remark~A.4 of~\cite{HK} gives a formula for~$\Ps$:
for $U \S \Prim (A)$ open,
$\Ps (U)$ is the union of all $W \in \Om$ such that $W \S U$.
(See Definition~A.2 of~\cite{HK}.)
(There a misprint in Remark~A.4 of~\cite{HK}
one must take $\Th = \Ps \circ \Ph$, not $\Th = \Ph \circ \Ps$.)

Proposition~2.15 of~\cite{HK} now provides an injective nondegenerate
\hm{} $\et_0 \colon A \to M (A)$
such that $\dt_{\infty} \circ \et_0$ is unitarily equivalent to~$\et_0$
(this is (\ref{T:HKReconstr:1}) of the conclusion)
and such that
$J (\Ps (U)) = \et_0^{-1} \big( \et_0 (A) \cap \M (A, \, J (U)) \big)$
for every open set $U \S \Prim (A)$.
Applying $\et_0$ to both sides
and using the formula for $\Ps (U)$,
we get (\ref{T:HKReconstr:2}) of the conclusion.

The uniqueness statement is Remark~2.16 of~\cite{HK}.

Let $\io \colon A \to {\mathcal{O}}_{\mathcal{H}}$
be the natural embedding.

We next prove that $\io$ induces a lattice isomorphism
from $\Om$ to $\Open (\Prim ( {\mathcal{O}}_{\mathcal{H}}))$,
using Corollary~4.31 of~\cite{HK}.

We first claim that if $I \S A$ is a closed ideal,
then $\et_0 (I) A \S I$ \ifo{} $\Prim (I) \in \Om$.
To prove the claim,
set
\[
U = \Prim (I)
\andeqn
V = \bigcup \{ W \in \Om \colon W \S U \}.
\]
Thus $J (U) = I$,
so by part~(\ref{T:HKReconstr:2}) (already proved),
we have
\[
\et_0 (J (V)) = \et_0 (A) \cap M (A, I).
\]
If $U \in \Om$ then $V = U$,
so $\et_0 (I) \S M (A, I)$.
Therefore, by definition,
$\et_0 (I) A \S I$.
On the other hand, suppose that $U \not\in \Om$.
Since $\Om$ is closed under supremums,
$V$ is a proper subset of~$U$.
Since $\et_0$ is injective,
it follows that $\et_0 (A) \cap M (A, I) = \et_0 (J (V))$
is a proper subset of $\et_0 (I)$.
Thus $\et_0 (I) \not\S M (A, I)$, whence $\et_0 (I) A \not\S I$.
This proves the claim.

It now follows from Remark~2.20 of~\cite{HK}
that if $I \S A$ is a closed ideal such that $\et_0 (I) A \S I$,
then $\et_0 (I) = \et_0 (A) \cap [A + M (A, I) ]$.
So Corollary~4.31 of~\cite{HK} applies,
and the fact that $\io$ induces a lattice isomorphism follows.

Condition~(\ref{T:HKReconstr:1})
implies that $\et_0 (A) \cap A = \{ 0 \}$.
Using this fact,
the isomorphism
${\mathcal{T}} ({\mathcal{H}}) \cong {\mathcal{O}} ({\mathcal{H}})$
follows from Corollary~3.14 of~\cite{Pm}.
Stability follows from Corollary~4.28 of~\cite{HK}.
Nuclearity follows from the description in Corollary~4.28 of~\cite{HK}
of the algebra~$C$ there;
see Remark~4.29 of~\cite{HK}.
\end{proof}

The Cuntz-Pimsner algebra ${\mathcal{O}}_{\mathcal{H}}$
of Theorem~\ref{T:HKReconstr}
is also strongly purely infinite
and the natural embedding of
$A$ into ${\mathcal{O}}_{\mathcal{H}}$
is a $KK (\Omega; \cdot, \cdot)$-equivalence.
See footnote~3 on page~7 of~\cite{HK}.
We will not need these facts.
(We also don't need stability.)

For a \ca~$A$ and $\ph \in \Aut (A)$,
we denote by $M (\ph)$
the induced automorphism of the multiplier algebra $M (A)$.

Recall that if $G$ is a topological group
and $\af \colon G \to \Aut (A)$ is a \ct{} action
of $G$ on a \ca~$A$,
then an {\emph{$\af$-cocycle}} is a strictly \ct{} function
$g \mapsto u_g$ from $G$ to the unitary group of $M (A)$
such that $u_{g h} = u_g M (\af_g) (u_h)$ for all $g, h \in G$.
If $g \mapsto u_g$ is an $\af$-cocycle,
then the formula ${\widetilde{\af}}_g (a) = u_g \af_g (a) u_g^*$,
for $g \in G$ and $a \in A$,
also defines a \ct{} action of $G$ on~$A$.
This action is called a {\emph{cocycle perturbation}} of~$\af$,
and the actions $\af$ and ${\widetilde{\af}}$ are
said to be {\emph{exterior equivalent}}.
Further recall that actions $\af \colon G \to \Aut (A)$
and $\bt \colon G \to \Aut (B)$
are {\emph{cocycle conjugate}} if there is an isomorphism
$\ph \colon A \to B$
and a cocycle perturbation ${\widetilde{\af}}$ of~$\af$
such that
$\bt_g = \ph \circ {\widetilde{\af}}_g \circ \ph^{-1}$
for all $g \in G$.
Let $K = K (l^2)$.
Then actions $\af$ and $\bt$ as above are
called {\emph{stably cocycle conjugate}} if the actions
$g \mapsto \id_K \otimes \af_g$
and $g \mapsto \id_K \otimes \bt_g$,
of $G$ on $K \otimes A$ and $K \otimes B$,
are cocycle conjugate.

In the definition of stable cocycle conjugacy,
one can replace the trivial action of $G$ on~$K$
with any inner action,
that is, one of the form $g \mapsto \Ad (v_g)$
for some unitary representation $g \mapsto v_g$
of $G$ on~$l^2$.
The reason is that $g \mapsto \Ad (v_g) \otimes \af_g$
is a cocycle perturbation of $\id_K \otimes \af$,
using the cocycle $g \mapsto v_g \otimes 1$.

It seems unlikely that, even for $A$ and $B$ stable,
stable cocycle conjugacy implies cocycle conjugacy.

\begin{thm}\label{T:EquivReconstr}
Let the notation and assumptions
be as in Theorem~\ref{T:HKReconstr}.
In addition, let $G$ be a second countable locally compact group
and let $\af \colon G \to \Aut (A)$ be an action of $G$ on~$A$.
Assume that $\Om$ is $\af$-invariant
the sense that for every $g \in G$ and $U \in \Om$,
we have $\Prim (\af_g) (U) \in \Om$.
Then there exist a stably cocycle conjugate action
$\gm \colon G \to \Aut (A)$,
for which the conjugation map $\ph \colon K \otimes A \to K \otimes A$
from the discussion above induces the identity on $\Prim (A)$,
and a \hm{} $\et_0 \colon A \to \M (A)$
as in Theorem~\ref{T:HKReconstr},
such that in addition
$\et_0$ is $\gm$-equivariant in the sense that
$\et_0 \circ \gm_g = \M (\gm_g) \circ \et_0$
for all $g \in G$.
\end{thm}

\begin{proof}
Take $K = K (l^2 \otimes L^2 (G))$.
Let $g \mapsto z_g$ be the regular representation of $G$ on $L^2 (G)$,
for $g \in G$
set $v_g = 1 \otimes z_g \in L (l^2 \otimes L^2 (G))$,
define $B = K \otimes A$,
and define an action $\bt \colon G \to \Aut (B)$
by $\bt_g = \Ad (v_g) \otimes \af_g$
for $g \in G$.
Since $A$ is stable,
the action $g \mapsto \id_K \otimes \af_g$
is stably conjugate to~$\af$
using a conjugation map
which induces the identity on $\Prim (A)$,
so $\bt$
is stably cocycle conjugate to~$\af$
using a conjugation map
which induces the identity on $\Prim (A)$.
In particular,
under the obvious identification of $\Prim (B)$
with $\Prim (A)$,
the sublattice $\Om$ of the hypotheses is $\bt$-invariant.

For a locally compact Hausdorff space~$X$ and a \ca~$C$,
we let $C_{ {\mathrm{b}}, {\mathrm{st}} } (X, M (C))$
denote the \ca{} of bounded strictly \ct{} functions
from $X$ to $M (C)$.
We will construct maps
\begin{align*}
& K \otimes K \otimes A
  \stackrel{\nu}{\longrightarrow} K \otimes A
  \stackrel{\ph}{\longrightarrow}
     C_{ {\mathrm{b}}, {\mathrm{st}} }
             \big( \Z \times G, \, M (K \otimes A) \big)
           \\
& \hspace*{10em} {\mbox{}}
  \stackrel{\mu}{\longrightarrow} M (K \otimes K \otimes A)
  \stackrel{\dt_{\I}}{\longrightarrow} M (K \otimes K \otimes A),
\end{align*}
which are equivariant with respect to suitable actions of~$G$
(defined below) on the algebras shown,
and we will then take
$\et_0 = \dt_{\I} \circ \mu \circ \ph \circ \nu$.

The action on $K \otimes A$
is $g \mapsto \Ad (v_g) \otimes \af_g = \bt_g$.

The action on $K \otimes K \otimes A$
is
$g \mapsto \Ad (v_g) \otimes \Ad (v_g) \otimes \af_g
   = \Ad (v_g) \otimes \bt_g$.

For the action on
$C_{ {\mathrm{b}}, {\mathrm{st}} }
      \big( \Z \times G, \, M (K \otimes A) \big)$,
begin by defining an action
$\ta^{(0)} \colon G \to \Aut (C_0 (\Z \times G))$
by
\[
\ta^{(0)}_g (f) (n, h) = f (n, \, g^{-1} h)
\]
for $g, h \in G$, $n \in \Z$, and $f \in C_0 (\Z \times G)$.
Identifying $C_0 (\Z \times G, \, K \otimes A)$
with $C_0 (\Z \times G) \otimes K \otimes A$,
define
\[
\ta \colon G \to \Aut \big( C_0 (\Z \times G, \, K \otimes A) \big)
\]
by $\ta_g = \ta^{(0)}_g \otimes \bt_g$ for $g \in G$.
Use Corollary~3.4 of~\cite{APT} to identify
\begin{equation}\label{Eq:MCX428}
C_{ {\mathrm{b}}, {\mathrm{st}} }
          \big( \Z \times G, \, M (K \otimes A) \big)
  = M \big( C_0 (\Z \times G, \, K \otimes A) \big).
\end{equation}
Take the action of $G$ to be $g \mapsto M (\ta_g)$.
Directly in terms of
$C_{ {\mathrm{b}}, {\mathrm{st}} }
        \big( \Z \times G, \, M (K \otimes A) \big)$,
it is easy to see that we have
\[
M (\ta_g) (b) (n, h) = M (\bt_g) (b (n, \, g^{-1} h))
\]
for $g, h \in G$, $n \in \Z$,
and
$b \in
 C_{ {\mathrm{b}}, {\mathrm{st}} }
       \big( \Z \times G, \, M (K \otimes A) \big)$.
We warn that for
$b \in M \big( C_0 (\Z \times G, \, K \otimes A) \big)$,
the function $g \mapsto M (\ta_g) (b)$ is strictly \ct{}
but not in general norm \ct.

Finally, the action of $G$ on $M (K \otimes K \otimes A)$
is
\[
g \mapsto M \big( \Ad (v_g) \otimes \Ad (v_g) \otimes \af_g \big)
\]
for $g \in G$.
Again, this action is strictly \ct,
but generally not norm \ct.

As a representation space of~$G$,
the space $l^2 \otimes L^2 (G) \otimes l^2 \otimes L^2 (G)$
is the direct sum of countably many copies
of the regular representation of~$G$.
Thus, there is an equivariant unitary
\[
w \colon l^2 \otimes L^2 (G) \otimes l^2 \otimes L^2 (G)
  \to l^2 \otimes L^2 (G),
\]
and we take $\nu (k \otimes a) = w k w^* \otimes a$
for
\[
k \in K \big( l^2 \otimes L^2 (G) \otimes l^2 \otimes L^2 (G) \big)
\andeqn
a \in A.
\]
This map is obviously an equivariant isomorphism.

Next, we construct~$\ph$.
Apply Theorem~\ref{T:HKReconstr} with $B = K \otimes A$ in place of~$A$,
and let $\et_1 \colon K \otimes A \to M (K \otimes A)$
be the resulting \hm.
(We do not expect $\et_1$ to be equivariant.)
Then define
\[
\ph \colon K \otimes A \longrightarrow
  C_{ {\mathrm{b}}, {\mathrm{st}} }
       \big( \Z \times G, \, M (K \otimes A) \big)
\]
by
\[
\ph (a) (n, h) = \big( M (\bt_h) \circ \et_1 \circ \bt_h^{-1} \big) (a)
\]
for $a \in K \otimes A$, $n \in \Z$, and $h \in G$.
(The right hand side does not depend on~$n$.)

We check that the proposed definition of $\ph (a)$
is in fact a strictly \cfn.
Let $b \in K \otimes A$.
Then for $h \in G$ and $n \in \Z$,
we have,
using at the second step the fact that the product is in $K \otimes A$,
\[
\ph (a) (n, h) b
 = M (\bt_h)
  \big( \big( \et_1 \circ \bt_h^{-1} \big) (a) \bt_h^{-1} (b) \big)
 = \bt_h
  \big( \big( \et_1 \circ \bt_h^{-1} \big) (a) \bt_h^{-1} (b) \big).
\]
The functions
\[
(n, h) \mapsto \big( \et_1 \circ \bt_h^{-1} \big) (a)
\andeqn
(n, h) \mapsto \bt_h^{-1} (b)
\]
are norm \ct,
from which it follows that
$(n, h) \mapsto \ph (a) (n, h) b$ is norm \ct.
Similarly
$(n, h) \mapsto b \ph (a) (n, h)$ is norm \ct.
Strict continuity of $\ph (a)$ follows.

We claim that $\ph$ is equivariant.
For $a \in K \otimes A$,
$g, h \in G$, and $n \in \Z$,
we have
\begin{align*}
\big( M (\ta_g) \circ \ph) (a) (n, h)
& = M (\bt_g) \big( \ph (a) (n, \, g^{-1} h) \big)
\\
& = \big( M (\bt_g) \circ M (\bt_{g^{-1} h}
     \circ \et_1 \circ \bt_{g^{-1} h}^{-1} \big) (a)
\\
& = \big( M (\bt_h) \circ \et_1 \circ \bt_{h}^{-1} \big) (\bt_g (a))
  = (\ph \circ \bt_g) (a) (n, h).
\end{align*}
The claim is proved.

Now we construct~$\mu$.
Let $\mu_0 \colon C_0 (\Z \times G) \to M (K) = L (l^2 (\Z \times G))$
send $f \in C_0 (\Z \times G)$ to the operator $\mu_0 (f)$
given by multiplication by~$f$.
Then one checks that there is a unique \hm{}
\[
\mu_1 \colon C_0 (\Z \times G, \, K \otimes A)
       \to M (K \otimes K \otimes A)
\]
such that
$\mu_1 (f \otimes b) \cdot (k \otimes c) = \mu_0 (f) k \otimes b c$
for $f \in C_0 (\Z \times G)$, $b, c \in K \otimes A$, and $k \in K$.
Recall that if $C$ and $D$ are \ca{s},
then a \hm{} $\ph \colon C \to M (D)$
is called nondegenerate if ${\overline{\rh (C) D}} = D$.
One easily checks that $\mu_1$ is nondegenerate in this sense.
Therefore Corollary 1.1.15 of~\cite{JK}
provides a unique strict to strict \ct{} \hm{}
\[
\mu \colon M \big( C_0 (\Z \times G, \, K \otimes A) \big)
  \to M (K \otimes K \otimes A)
\]
such that
\[
\mu (x) \cdot \mu_1 (a) b = \mu_1 (x a) b
\]
for
\[
x \in M \big( C_0 (\Z \times G, \, K \otimes A) \big),
\,\,\,\,\,\,
a \in C_0 (\Z \times G, \, K \otimes A),
\andeqn
b \in K \otimes K \otimes A.
\]
Making the identification in~(\ref{Eq:MCX428}),
we get
\[
\mu \colon C_{ {\mathrm{b}}, {\mathrm{st}} }
       \big( \Z \times G, \, M (K \otimes A) \big)
    \to M (K \otimes K \otimes A),
\]
as required.

We claim that $\mu$ is equivariant.
It suffices to prove that $\mu_1$ is equivariant.
A calculation shows that
$\mu_0 \big( \ta_g^{(0)} (f) \big) = v_g \mu_0 (f) v_g^*$
for $f \in C_0 (\Z \times G)$ and $g \in G$.
For $f \in C_0 (\Z \times G)$, $g \in G$, $k \in K$, and
$b, c \in K \otimes A$,
we now have
\begin{align*}
& \big[ M (\Ad (v_g) \otimes \bt_g) \circ \mu_1 \big] (f \otimes b)
      \cdot (k \otimes c)
       \\
& \hspace*{3em} {\mbox{}}
 = M (\Ad (v_g) \otimes \bt_g) \big( \mu_1 (f \otimes b) \cdot
            \big[ \Ad (v_g)^{-1} (k) \otimes \bt_g^{-1} (c) \big] \big)
       \\
& \hspace*{3em} {\mbox{}}
 = M (\Ad (v_g) \otimes \bt_g)
    \big( \mu_0 (f) v_g^* k v_g \otimes b \bt_g^{-1} (c) \big)
 = v_g \mu_0 (f) v_g^* k \otimes \bt_g (b) c
       \\
& \hspace*{3em} {\mbox{}}
 = \mu_0 \big( \ta_g^{(0)} (f) \big) k \otimes \bt_g (b) c
 = \mu_1 \big( \ta_g (f \otimes b) \big) \cdot ( k \otimes c).
\end{align*}
Since $f$, $k$, $b$, and $c$ are arbitrary,
we therefore get
\[
\big[ M (\Ad (v_g) \otimes \bt_g) \circ \mu_1 \big]
  = \mu_1 \circ \ta_g.
\]
This is equivariance of~$\mu_1$.
The claim is proved.

We now define $\dt_{\I}$,
making a specific choice in Definition~\ref{D:dtInfty}.
We interpret $M (K \otimes K \otimes A)$ as the bounded operators
on the Hilbert $A$-module
$l^2 \otimes L^2 (G) \otimes l^2 \otimes L^2 (G) \otimes A$.
Choose isometries $t_1, t_2, \ldots \in L (l^2)$
with orthogonal ranges and such that
$\sum_{n = 1}^{\infty} t_n t_n^*$ converges to $1$ in the
strong operator topology.
Then set $s_n = t_n \otimes 1 \otimes 1 \otimes 1 \otimes 1$
for $n \in \N$.
The resulting map
\[
\dt_{\I} \colon
   M (K \otimes K \otimes A) \to M (K \otimes K \otimes A),
\]
following Definition~\ref{D:dtInfty},
is clearly equivariant.

It remain to verify that the resulting equivariant \hm{}
\[
\et_0 = \dt_{\I} \circ \mu \circ \ph \circ \nu \colon
    K \otimes K \otimes A \to M (K \otimes K \otimes A)
\]
is injective, nondegenerate,
and
satisfies conditions (\ref{T:HKReconstr:1}) and~(\ref{T:HKReconstr:2})
of Theorem~\ref{T:HKReconstr}.
Injectivity is immediate,
nondegeneracy is easy,
and condition~(\ref{T:HKReconstr:1})
is immediate from Lemma~\ref{L-dtIRmk}(\ref{L-dtIRmk:3}).

As a first step towards proving
condition~(\ref{T:HKReconstr:2}) of Theorem~\ref{T:HKReconstr},
we claim that if $J$ is an ideal in $A$ and
$b \in
 C_{ {\mathrm{b}}, {\mathrm{st}} }
       \big( \Z \times G, \, M (K \otimes A) \big)$,
then $\mu (b) \in M (K \otimes K \otimes A, \, K \otimes K \otimes J)$
\ifo{} for all $h \in G$ and $n \in \Z$,
we have $b (n, h) \in M (K \otimes A, \, K \otimes J)$.
The statement to be proved is that
$\mu (b) \cdot (K \otimes K \otimes A) \S K \otimes K \otimes J$
\ifo{} $b (n, h) \cdot (K \otimes A) \S K \otimes J$
for all $h \in G$ and $n \in \Z$,
which is clear.

Next, for any open set $U \S \Prim (A)$,
define the open set $S_U \S \Prim (A)$ by
$S_U = \bigcup \{ W \in \Omega \colon W \subset U \}$.
To simplify notation, for $V \S \Prim (A)$ abbreviate
$\Prim (\af_g) (V)$ to $g V$.
Since $\Om$ is $\af$-invariant,
we have $S_{g U} = g \cdot S_U$
for every open $U \S \Prim (A)$.

Now let $a \in K \otimes A$ and let $U \S \Prim (A)$ be open.
We claim that
\begin{enumerate}
\item\label{SUIdeal-1}
$(\mu \circ \ph) (a)
 \in M \big( K \otimes K \otimes A, \, K \otimes K \otimes J (U) \big)$
\setcounter{TmpEnumi}{\value{enumi}}
\end{enumerate}
\ifo{}
\begin{enumerate}
\setcounter{enumi}{\value{TmpEnumi}}
\item\label{SUIdeal-2}
$a \in K \otimes J (S_U)$.
\setcounter{TmpEnumi}{\value{enumi}}
\end{enumerate}
By the previous claim and the definition of~$\ph$,
condition~(\ref{SUIdeal-1})
is equivalent to:
\begin{enumerate}
\setcounter{enumi}{\value{TmpEnumi}}
\item\label{SUIdeal-3}
$\big( M (\bt_h) \circ \et_1 \circ \bt_h^{-1} \big) (a)
   \in M \big( K \otimes A, \, K \otimes J (U) \big)$
for all $h \in G$ and $n \in \Z$.
\setcounter{TmpEnumi}{\value{enumi}}
\end{enumerate}
Using $\bt_h^{-1} (K \otimes J (U) ) = K \otimes J (h^{-1} U)$,
we can rewrite~(\ref{SUIdeal-3}) as:
\begin{enumerate}
\setcounter{enumi}{\value{TmpEnumi}}
\item\label{SUIdeal-4}
$\et_1 ( \bt_h^{-1} (a) )
   \in M \big( K \otimes A, \, K \otimes J (h^{-1} U) \big)$
for all $h \in G$ and $n \in \Z$.
\setcounter{TmpEnumi}{\value{enumi}}
\end{enumerate}
By the condition in Theorem~\ref{T:HKReconstr}(\ref{T:HKReconstr:2})
for~$\et_1$,
condition~(\ref{SUIdeal-4})
is equivalent to:
\begin{enumerate}
\setcounter{enumi}{\value{TmpEnumi}}
\item\label{SUIdeal-5}
$\bt_h^{-1} (a) \in K \otimes J (h^{-1} U)$
for all $h \in G$ and $n \in \Z$.
\end{enumerate}
Since~(\ref{SUIdeal-5}) is the same as~(\ref{SUIdeal-2}),
the claim is proved.

Since $\nu$ is an isomorphism which preserves the ideals,
it follows that for $a \in K \otimes K \otimes A$,
we have
\[
(\mu \circ \ph \circ \nu) (a)
  \in M \big( K \otimes K \otimes A, \, K \otimes K \otimes J (U) \big)
\]
\ifo{} $a \in K \otimes K \otimes J (S_U)$.
Condition~(\ref{T:HKReconstr:2}) of Theorem~\ref{T:HKReconstr}
now follows by applying Lemma~\ref{L-dtIRmk}(\ref{L-dtIRmk:4}).
\end{proof}

\begin{thm}\label{T:ExistAct}
Let $G$ be a second countable locally compact group,
let $P$ be a second countable noncompact
locally compact Hausdorff space,
and let $G$ act \ct ly on~$P$.
Then there exists a separable nuclear \ca~$C$
and an action $\gm$ of $G$ on $C$
such that $K \otimes \OT \otimes C$,
with the action $g \mapsto \id_{K \otimes \OT} \otimes \gm_g$,
is equivariantly isomorphic to~$C$,
and such that
$\Prim (C)$ is equivariantly isomorphic to $\Xi (P)$
with its induced action of~$G$.
\end{thm}

\begin{proof}
By Proposition~\ref{P:ExistDR}
and Proposition~\ref{P:AppOfDR},
there exists a second countable compact Hausdorff space~$Z_0$
and a \ct{} open surjection
$\pi_0 \colon Z_0 \to \Xi (P)$.
It follows from Lemma~\ref{L:VerPseudo} that $\pi$ is
pseudo-open and pseudo-epimorphic.
We may therefore apply Theorem~\ref{T:ExistA}
to produce a separable nuclear \ca~$A$
such that $K \otimes \OT \otimes A \cong A$
and such that there is an isomorphism $l \colon \Prim (A) \to \Xi (P)$.

Let $Z = G \times Z_0$,
and let $\pi = \id_G \times \pi_0 \colon Z \to G \times \Xi (P)$.
Then $\pi$ is also a \ct{} open surjection,
so, by the same reasoning as in the previous paragraph,
the space $G \times \Xi (P)$ satisfies
the hypotheses of Theorem~\ref{T:ExistA}.
The algebra $B = C_0 (G) \otimes A$
is separable, nuclear, and satisfies
$K \otimes \OT \otimes B \cong B$.
Let $m \colon \Prim (C_0 (G)) \to G$ be the standard isomorphism,
and write $m \times l$ for the corresponding isomorphism
$\Prim (B) \to G \times \Xi (P)$.
Proposition~\ref{P:IndActXi} implies that
$(g, x) \mapsto (g, g x)$ is a \hme{}
of $\Prim (B) \cong G \times \Xi (P)$.
So, using the uniqueness statement in Theorem~\ref{T:ExistA},
there exists $\ph \in \Aut (B)$
which induces this \hme.
That is,
\[
\big( (m \times l) \circ \Prim (\ph) \circ
             (m \times l)^{-1} \big) (g, x)
  = (g, g x)
\]
for $g \in G$ and $x \in \Xi (P)$.

Define $\rh_g \in \Aut (C_0 (G))$ by
$\rh_g (f) (h) = f (h g)$ for $g, h \in G$ and $f \in C_0 (G)$.
Then
$\big( m \circ \Prim (\rh_g) \circ m^{-1} \big) (h) = h g^{-1}$
for $g, h \in G$.
Therefore the action
\[
g \mapsto \bt_g
  = \ph^{-1} \circ (\rh_g \otimes \id_A) \circ \ph
 \in \Aut (B)
\]
induces the \hme{} $(h, x) \mapsto (h g^{-1}, \, g x)$
on $\Prim (B) \cong G \times \Xi (P)$.

Let $\Om \S {\mathcal{I}} (B) \cong {\mathbb{O}} (G \times \Xi (P))$
be the sublattice of ${\mathcal{I}} (B)$ given in terms of open sets
by
\[
\Om = \big\{ G \times U \colon
      {\mbox{$U$ is an open subset of $\Xi (P)$}} \big\}.
\]
With $B$ in place of $A$ and $\bt$ in place of $\af$,
this choice of $\Om$
clearly satisfies the hypotheses of Theorem~\ref{T:EquivReconstr}.
Let $C_0$ be the \ca~${\mathcal{O}}_{\mathcal{H}}$ of the conclusion
of the theorem.
Equivariance of $\et_0$ means that
${\mathcal{H}}$ is an equivariant Hilbert $A$--$A$-bimodule
in the sense of Remark 4.10(2) of~\cite{Pm}.
Therefore Remark 4.10(2) of~\cite{Pm}
provides an induced action $\gm^{(0)} \colon G \to \Aut (C_0)$.
The conclusion also gives an equivariant isomorphism
of $\Om$ with ${\mathcal{I}} (C_0)$.
There is an obvious equivariant lattice isomorphism
$\Om \cong {\mathbb{O}} (\Xi (P))$,
and we thus obtain an equivariant homeomorphism
$\Prim (C_0) \cong \Xi (P)$.

All that is missing is stability under tensoring with $K \otimes \OT$.
To this end, set $C = K \otimes \OT \otimes C_0$
and $\gm_g = \id_{K \otimes \OT} \otimes \gm_g^{(0)}$.
\end{proof}

\begin{prp}\label{P:RR0}
Suppose the space $P$ in Theorem~\ref{T:ExistAct} is
infinite and discrete.
Then the \ca~$C$ constructed in Theorem~\ref{T:ExistAct}
has real rank zero.
\end{prp}

\begin{proof}
We may as well take $P = \Z$.
By Corollary~\ref{C:AFPrimZ},
there is an AF~algebra $A$
such that $\Prim (A) \cong \Xi (\Z)$.
(This can also be deduced from Section~5 of~\cite{BE}.)
Therefore $D = K \otimes \OT \otimes A$
is a separable nuclear \ca{} such that
$K \otimes \OT \otimes D \cong D$ and $\Prim (A) \cong \Xi (\Z)$.
Since $D$ is the tensor product of $\OT$
with the AF~algebra $K \otimes A$,
clearly $D$ has real rank zero.
The uniqueness statement in Theorem~\ref{T:ExistA}
implies that $D \cong C$.
\end{proof}

\begin{thm}\label{T:ExistMinAct}
Let $G$ be a noncompact second countable locally compact group,
and let $P$ be a locally compact second countable $G$-space
with no nonempty $G$-invariant compact subsets.
Then there exists a separable nuclear nonsimple prime \ca~$C$
such that $K \otimes \OT \otimes C \cong C$,
and such that there is a minimal action
$\gm \colon G \to \Aut (C)$
for which $\Prim (C)$ is equivariantly homeomorphic to $\Xi (P)$.
If the action on~$P$ is effective (Definition~\ref{D-Effective}),
then the action on~$C$ and
the induced action on $\Prim (C)$ are also effective.
If $P$ is discrete then $C$ may be taken to have real rank zero.
\end{thm}

See Example~\ref{E-111109dd}
for the existence of suitable choices for~$P$.

\begin{proof}[Proof of Theorem~\ref{T:ExistMinAct}]
We apply Theorem~\ref{T:ExistAct}.
Let $C$ be the resulting \ca,
with $\gm \colon G \to \Aut (C)$.
Then there is an equivariant homeomorphism
from $\Prim (C)$ to $\Xi (G)$.
Since $\Xi (G)$ has more than one point, $C$ is not simple.
However, $C$ is prime by Corollary~\ref{L-Prime426}.

If the action of $G$ on $P$ is effective,
then so is the action of $G$ on $\Xi (G)$.
The action of $G$ on~$C$ is then also effective.

The last sentence follows from Proposition~\ref{P:RR0}.
\end{proof}

When $G$ is $\Z$ or~$\R$,
we can require that the algebra be unital.
For $G = \Z$,
this follows directly from the results in Section~\ref{Sec:AF}.

\begin{prp}\label{P:ZUnital}
There exists a separable unital nuclear nonsimple prime \ca~$C$
with real rank zero
such that $\OT \otimes C \cong C$,
and such that there is a minimal effective action
$\gm \colon \Z \to \Aut (C)$
for which induced action on $\Prim (C)$ is also effective.
\end{prp}

\begin{proof}
Let $A$ and the action $\af \colon \Z \to \Aut (A)$
be the algebra and action of Corollary~\ref{C-ZAF428}.
Then take $C = \OT \otimes A$
and $\gm_n = \id_{\OT} \otimes \af_n$ for $n \in \Z$.
\end{proof}

For $G = \R$, more work is needed.
The following lemma is essentially due to Connes.

\begin{lem}\label{L-RMkInv502}
Let $A$ be a \ca,
let $\af \colon \R \to \Aut (A)$ be an action,
and let $p \in A$ be a \pj.
Then there exist a \pj{} $q \in A$
which is \mvnt{} to~$p$
and a cocycle perturbation $\bt \colon \R \to \Aut (A)$ of~$\af$
such that $q$ is $\bt$-invariant.
\end{lem}

\begin{proof}
It is enough to prove this for the unitization $A^{+}$ of~$A$,
since if $p \in A$ is \mvnt{} to a \pj{} $q \in A^{+}$,
then $q \in A$.
So assume that $A$ is unital.

Let $A^{\infty}$ be the dense subalgebra of $A$
consisting of the $\af$-smooth elements in~$A$.
Lemma~4 in Section~VI of~\cite{Cn}
provides a \pj{} $q \in A^{\I}$
which is \mvnt{} to~$p$.
The existence of~$\bt$
now follows from Proposition~4 in Section~II of~\cite{Cn}.
\end{proof}

\begin{lem}\label{L-ExistP508}
Let $A$ be a \ca{} such that $\Prim (A)$ is compact.
Then ${\mathcal{O}}_{\infty} \otimes A$ contains a full \pj.
\end{lem}

\begin{proof}
Since $M_2 \otimes {\mathcal{O}}_{\infty} \otimes A$
is isomorphic to a full subalgebra
of ${\mathcal{O}}_{\infty} \otimes A$,
it suffices to prove that $M_2 \otimes {\mathcal{O}}_{\infty} \otimes A$ contains a full \pj.

We first claim that there is $x \in A$ such that $\| x + I \| \geq 1$
for all $I \in \Prim (A)$.
Recall (Proposition 3.3.2 of~\cite{Dx})
that for any $x \in A$,
the map $I \mapsto \| x + I \|$ is lower semi\ct{}
on $\Prim (A)$.
For $I \in \Prim (A)$,
choose $x_I \in A_{+}$ such that $\| x_I + I \| \geq 2$.
The sets
\[
U_I = \big\{ J \in \Prim (A) \colon \| x_I + J \| > 1 \big\}
\]
form an open cover of $\Prim (A)$.
So there is a finite set $F \S \Prim (A)$
such that the sets $U_I$, for $I \in F$, cover $\Prim (A)$.
The element $x = \sum_{I \in F} x_I$ verifies the claim.

Choose \cfn{s} $f, g, h \colon [0, \I) \to [0, 1]$
such that $f (0) = g (0) = h (0) = 0$,
$f g = g$, $g h = h$, and $h (\ld) = 1$ for all $\ld \in [1, \I)$.
Let $x$ be as in the claim.
Then $f (x), g (x), h (x) \in A_{+}$
and satisfy $f (x) g (x) = g (x)$ and $g (x) h (x) = h (x)$.
Moreover, for all $I \in \Prim (A)$,
we have $\| h (x) + I \| = 1$.
Define elements
\[
a, b_1, b_2, c_1, c_2
 \in M_2 \otimes {\mathcal{O}}_{\infty} \otimes A = M_2 ({\mathcal{O}}_{\infty} \otimes A)
\]
by
\[
a = [1 \otimes f (x)] \oplus [1 \otimes f (x)],
\,\,\,\,\,\,
b_1 = [1 \otimes g (x)] \oplus 0,
\,\,\,\,\,\,
b_2 = 0 \oplus [1 \otimes g (x)],
\]
\[
c_1 = [1 \otimes h (x)] \oplus 0,
\andeqn
c_2 = 0 \oplus [1 \otimes h (x)].
\]
Then
\begin{equation}\label{Eq:Unit508}
a b_1 = b_1,
\,\,\,\,\,\,
b_1 c_1 = c_1,
\,\,\,\,\,\,
a b_2 = b_2,
\andeqn
b_2 c_2 = c_2.
\end{equation}

It follows from Proposition~4.5 of~\cite{KR0}
that $M_2 \otimes {\mathcal{O}}_{\infty} \otimes A$ is purely infinite.
Since $a$ and $c_1$ are both full,
there is a sequence $(y_n)_{n \in \N}$
in $M_2 \otimes {\mathcal{O}}_{\infty} \otimes A$
such that $\limi{n} y_n c_1 y_n^* = a$.
\Wolog{} $y_n c_1 y_n^* \neq 0$ for all $n \in \N$.
For $n \in \N$ define $z_n = \| y_n c_1 y_n^* \|^{-1/2} y_n$.
Since $\| a \| = 1$,
we also get $\limi{n} z_n c_1 z_n^* = a$,
and now $\| z_n c_1 z_n^* \| = 1$ for all $n \in \N$.

For $n \in \N$,
define $x_n = c_1^{1/2} z_n^*$.
Then $\| x_n \| = 1$.
We have $\limi{n} x_n^* x_n = a$,
so~(\ref{Eq:Unit508}) implies
\[
\limi{n} x_n^* x_n b_1 = b_1
\andeqn
\limi{n} x_n^* x_n b_2 = b_2.
\]
Therefore $\limi{n} (1 - x_n^* x_n) b_1 = 0$.
By approximating $(1 - x_n^* x_n)^{1/2}$ with polynomials
in $1 - x_n^* x_n$ with no constant term,
we deduce that $\limi{n} (1 - x_n^* x_n)^{1/2} b_1 = 0$.
Using~(\ref{Eq:Unit508}),
we get $b_1 c_1^{1/2} = c_1^{1/2}$,
so $b_1 x_n = x_n$.
Therefore
\[
\big\| (1 - x_n^* x_n)^{1/2} x_n \big\|
  \leq \big\| (1 - x_n^* x_n)^{1/2} b_1 \big\|
                \cdot \| x_n \|
  = \big\| (1 - x_n^* x_n)^{1/2} b_1 \big\|.
\]
Thus
\begin{equation}\label{Eq:OrtLim508}
\lim_{n \to \I} \big\| (1 - x_n^* x_n)^{1/2} x_n \big\| = 0.
\end{equation}

We also get
\begin{equation}\label{Eq:Ort508}
x_n^* b_2 = z_n c_1^{1/2} b_2 = 0
\end{equation}
for all $n \in \N$.

We now follow the proof of Theorem~3.1 of~\cite{BC},
except that we don't have an exact scaling element.
Define
\[
v_n = x_n + (1 - x_n^* x_n)^{1/2}
     \in (M_2 \otimes {\mathcal{O}}_{\infty} \otimes A)^{+}
\]
for $n \in \N$.
Then
\[
v_n^* v_n
  = x_n^* x_n + (1 - x_n^* x_n)^{1/2} x_n
     + \big[ (1 - x_n^* x_n)^{1/2} x_n \big]^* + (1 - x_n^* x_n),
\]
so $\limi{n} v_n^* v_n = 1$ by~(\ref{Eq:OrtLim508}).
Redefining initial terms of the sequence $(v_n)_{n \in \N}$
to be~$1$,
we may moreover
assume that $v_n^* v_n$ is invertible for all $n \in \N$,
and define $s_n = v_n (v_n^* v_n)^{-1/2}$.
Then $\limi{n} \| s_n - v_n \| = 0$.
Also,
\[
v_n v_n^*
  = x_n x_n^* + x_n (1 - x_n^* x_n)^{1/2}
      +  (1 - x_n^* x_n)^{1/2} x_n^* + (1 - x_n^* x_n).
\]
{}From (\ref{Eq:Ort508}) and~(\ref{Eq:Unit508}),
we conclude that $\limi{n} v_n v_n^* b_2 = 0$,
so that $\limi{n} v_n v_n^* c_2 = 0$.
Therefore $\limi{n} s_n s_n^* c_2 = 0$.
Also, for all $n \in \N$, we have
$1 - v_n v_n^* \in M_2 \otimes {\mathcal{O}}_{\infty} \otimes A$,
so
\[
1 - s_n s_n^*
 = 1 - v_n v_n^* + v_n \big[ 1 - (v_n^* v_n)^{-1} \big] v_n^*
 \in M_2 \otimes {\mathcal{O}}_{\infty} \otimes A.
\]

Choose some $n \in \N$ such that $\| s_n s_n^* c_2 \| < \tfrac{1}{2}$.
Set $e = 1 - s_n s_n^*$, which is a projection.
For every $I \in \Prim (A)$,
we have
\begin{align*}
\| e + M_2 \otimes {\mathcal{O}}_{\infty} \otimes I \|
& \geq \| e c_2 + M_2 \otimes {\mathcal{O}}_{\infty} \otimes I \|
  \geq \| c_2 + M_2 \otimes {\mathcal{O}}_{\infty} \otimes I \|
          - \| s_n s_n^* c_2 \|
     \\
& = \| h (x) + I \|
          - \| s_n s_n^* c_2 \|
  > \tfrac{1}{2}.
\end{align*}
Therefore $e + M_2 \otimes {\mathcal{O}}_{\infty} \otimes I \neq 0$.
It follows that $e$ is full.
\end{proof}

\begin{prp}\label{P-RUnital}
There exists a separable unital nuclear nonsimple prime \ca~$C$
such that $\OT \otimes C \cong C$,
and such that there is a minimal effective action
$\af \colon \R \to \Aut (C)$
for which induced action on $\Prim (C)$ is also effective.
\end{prp}

\begin{proof}
Apply Theorem~\ref{T:ExistMinAct} with $G = \R$
and with $P = \R$ with the translation action,
getting $C_0$ and an action $\gm^{(0)} \colon \R \to \Aut (C_0)$.
Since $\Xi (\R)$ is compact,
Lemma~\ref{L-ExistP508} provides a full \pj{} $p_0 \in C_0$.
Using Lemma~\ref{L-RMkInv502},
we can find an exterior equivalent action
$\gm^{(1)} \colon \R \to \Aut (C_0)$
and a \pj{} $p \in C_0$ which is
\mvnt{} to~$p_0$
such that $\gm^{(1)}_t (p) = p$ for all $t \in \R$.
Then $p$ is also full.
Now take $C = \OT \otimes p C_0 p$
and $\gm_t = \id_{\OT} \otimes \gm^{(1)}_t |_{p C_0 p}$.
\end{proof}

\section{Exotic actions on $\OT$}\label{Sec:3}

\indent
In this section,
for every second countable locally compact abelian group
which is not discrete,
we prove the existence of an action
$\af \colon G \to \Aut (K \otimes \OT)$
such that the crossed product is prime but not simple.
This behavior is unlike what one usually expects
for crossed products by actions on simple \ca{s}.
For $G = S^1$ and $G = \R$,
the action can be taken to be on $\OT$ instead of $K \otimes \OT$,
and for $G = S^1$
the action can be taken to be on
a simple separable stably finite unital nuclear \ca.

To obtain such examples,
we start with actions of ${\widehat{G}}$
as in Theorem~\ref{T:ExistMinAct},
Proposition~\ref{P:ZUnital},
or Corollary~\ref{C-FnAct508} with $n = 1$,
and take the dual action of~$G$.
We do not know whether crossed products
by actions such as in Theorem~\ref{T:ExistMinAct}
are necessarily simple
(although this is true in some cases),
so the first part of this section is devoted to a method
for modifying such an action slightly
so as to ensure simplicity of the crossed product.

Related examples can be produced using~\cite{Brt}.
The algebras are unital for more groups,
in particular for connected abelian groups,
but there is much less control over the primitive ideal spaces
of the crossed products.
See Theorem~\ref{T_2X31-Brt}.

We will use the Fock-Toeplitz construction of~${\mathcal{O}}_{\infty}$.
For reference, and to establish notation,
we summarize it here.
It is a special case of, for example,~\cite{Pm}.

\begin{dfn}\label{D:TCAlg}
Let $H$ be a Hilbert space.
Define the Fock-Toeplitz \ca{} ${\mathcal{E}} (H)$
to be the universal \uca{} on generators
$s_{\xi}$ for $\xi \in H$ and relations:
\begin{enumerate}
\item\label{D:TCAlg:Add}
$s_{\xi + \et} = s_{\xi} + s_{\et}$
for $\xi, \et \in H$.
\item\label{D:TCAlg:Mult}
$s_{ \ld \xi} = \ld s_{\xi}$
for $\ld \in \C$ and $\xi \in H$.
\item\label{D:TCAlg:Cz}
$s_{\xi}^* s_{\et} = \langle \et, \xi \rangle \cdot 1$
for $\xi, \et \in H$.
\end{enumerate}
\end{dfn}

In the notation of Definition~1.1 of~\cite{Pm},
this algebra is ${\mathcal{T}}_H$ when $H$ is regarded
in the obvious way as a $\C$--$\C$--bimodule in the sense of~\cite{Pm}.
It is the same as the algebra ${\mathcal{O}}_H$
of Definition~1.1 of~\cite{Pm},
by Corollary~3.14 of~\cite{Pm}.

\begin{lem}\label{L:TCOfL2}
Let $H$ be a separable in\fd{} Hilbert space.
Then ${\mathcal{E}} (H) \cong {\mathcal{O}}_{\infty}$.
\end{lem}

\begin{proof}
This is Example~(1) after Proposition~1.3 in~\cite{Pm}.
\end{proof}

\begin{lem}\label{L:GOnFc}
Let $G$ be a locally compact group,
and let $g \mapsto u_g$ be a unitary representation of $G$
on a Hilbert space~$H$.
Then there is a
unique \ct{} action $\bt \colon G \to \Aut ({\mathcal{E}} (H))$
such that $\bt (s_{\xi}) = s_{u_g \xi}$ for all $\xi \in H$.
\end{lem}

\begin{proof}
See Remark 4.10(2) of~\cite{Pm}.
\end{proof}

\begin{lem}\label{L:KerComp}
Let $\af \colon G \to \Aut (A)$ be an action of a
second countable infinite locally compact group~$G$ on a \ca~$A$.
Set $B = {\mathcal{E}} (L^2 (G))$,
and let $\bt \colon G \to \Aut (B)$ be the action of
Lemma~\ref{L:GOnFc} coming from the left regular representation
$z$ of $G$ on $L^2 (G)$.
Let $(u, \pi)$ be a covariant representation of
$(G, \, B \otimes A, \, \bt \otimes \af)$ on a Hilbert space~$H$,
such that $\pi$ is injective and nondegenerate.
Let $\io \colon A \to B \otimes A$ be $\io (a) = 1 \otimes a$,
and let $\sm \colon C^* (G, A, \af) \to L (H)$
be the integrated form of the covariant representation
$(u, \, \pi \circ \io)$ of $(G, A, \af)$.
Then $\ker (\sm)$ is contained in the kernel
of the canonical \hm{}
$C^* (G, A, \af) \to C^*_{\mathrm{r}} (G, A, \af)$.
\end{lem}

\begin{proof}
Since $\pi$ is injective and nondegenerate,
it extends to an injective unital \hm,
also denoted~$\pi$,
from $M (B \otimes A)$ to $L (H)$.
For $b \in B$, write its image in $M (B \otimes A)$
as $b \otimes 1$.

We claim that there is an isometric linear map
$v \colon L^2 (G) \otimes H \to H$
such that $v (\xi \otimes \et) = \pi (s_{\xi} \otimes 1) \et$
for $\xi \in L^2 (G)$ and $\et \in H$.
To prove the claim, the following computation suffices:
for $\xi_1, \xi_2 \in L^2 (G)$ and $\et_1, \et_2 \in H$,
\[
\langle \pi (s_{\xi_1} \otimes 1) \et_1,
      \, \pi (s_{\xi_2} \otimes 1) \et_2 \rangle
  = \langle \pi (s_{\xi_2}^* s_{\xi_1} \otimes 1) \et_1,
      \, \et_2 \rangle
  = \langle \xi_1, \xi_2 \rangle \langle \et_1, \et_2 \rangle.
\]

We now consider the regular covariant representation
of $(G, A, \af)$ coming from the representation
$\pi \circ \io$ of $A$ on~$H$.
(See 7.7.1 of~\cite{Pd}.)
On $L^2 (G, H) = L^2 (G) \otimes H$,
we define the representation $z^H$ of $G$
by $z_g^H = z_g \otimes 1$,
and we define the representation $\rh_0$ of $A$
by
\[
(\rh_0 (a) \xi) (h)
   = (\pi \circ \io \circ \af_h^{-1} ) (a) ( \xi (h))
\]
for $a \in A$, $\xi \in L^2 (G, H)$, and $h \in G$.
Then the regular covariant representation is $(z^H, \rh_0)$.

Let $y \in L ( L^2 (G, H) )$ be the unitary given by
$(y \xi) (h) = u_h (\xi (h))$
for $\xi \in L^2 (G, H)$ and $h \in G$.

We claim that $y z_g^H y^* = z_g \otimes u_g$
for $g \in G$.
To see this, we compute, for $g, h \in G$ and $\xi \in H$:
\begin{align*}
\big( (y z_g^H y^* ) \xi \big) (h)
 & = u_h \big( \big( ( z_g^H y^*) \xi \big) (h) \big)
   = u_h \big( ( y^* \xi) (g^{-1} h) \big)
        \\
 & = u_h \big( u_{h^{-1} g} (\xi (g^{-1} h)) \big)
   = u_g (\xi (g^{-1} h))
   = \big( ( z_g \otimes u_g) \xi \big) (h).
\end{align*}
This proves the claim.

Define $\rh \colon A \to L ( L^2 (G) \otimes H)$
by $\rh (a) = 1 \otimes \pi (1 \otimes a)$ for $a \in A$.
We claim that $y \rh_0 ( \cdot ) y^* = \rh$.
To see this, we compute, for $a \in A$, $h \in G$, and $\xi \in H$:
\begin{align*}
\big( y \rh_0 ( a ) y^* \big) \xi (h)
 & = u_h \big( ((\rh_0 ( a ) y^* ) \xi ) (h) \big)
   = u_h \big( \pi \circ \io \circ \af_h^{-1} \big) (a)
            \big( (y^* \xi) (h) \big)
        \\
 & = u_h \big( \pi \circ \io \circ \af_h^{-1} \big) (a)
            \big( u_h^*( \xi (h)) \big)
        \\
 & = \big( u_h \big[ \pi
                \circ \big( \gm_h^{-1} \otimes \af_h^{-1} \big)
             (1 \otimes a) \big] u_h^* \big) ( \xi (h))
        \\
 & = \pi (1 \otimes a) (\xi (h))
   = \big( \rh (a) \xi \big) (h).
\end{align*}
This proves the claim.

It follows that $(z \otimes u, \, \rh)$
is a covariant representation of
$(G, A, \af)$ which is unitarily equivalent to $(z^H, \rh_0)$.

We now claim that $v$ intertwines $(z \otimes u, \, \rh)$
and $(u, \, \pi \circ \io)$.
For $g \in G$, $\xi \in L^2 (G)$, and $\et \in H$, we have:
\begin{align*}
v (z_g \otimes u_g) (\xi \otimes \et)
  & = v (z_g \xi \otimes u_g \et)
    = \pi (s_{z_g \xi} \otimes 1) u_g \et
    = \pi \big( \bt_g (s_{\xi}) \otimes 1 \big) u_g \et
        \\
  & = \pi \big( ( \bt_g \otimes \af_g ) (s_{\xi} \otimes 1) \big)
              u_g \et
    = \big[ u_g \pi (s_{\xi} \otimes 1) u_g^* \big] u_g \et
        \\
  & = u_g \pi (s_{\xi} \otimes 1) \et
    = u_g v (\xi \otimes \et).
\end{align*}
Also, for $a \in A$, $\xi \in L^2 (G)$, and $\et \in H$, we have:
\begin{align*}
v \rh (a) (\xi \otimes \et)
 & = v \big( 1 \otimes \pi (1 \otimes a) \big) (\xi \otimes \et)
   = v \big( \xi \otimes \pi (1 \otimes a) \et \big)
        \\
 & = \pi (s_{\xi} \otimes 1) \big( \pi (1 \otimes a) \et \big)
   = \pi (1 \otimes a) \pi (s_{\xi} \otimes 1) \et
   = \pi (1 \otimes a) v (\xi \otimes \et).
\end{align*}
These two calculations prove the claim.

It follows that the integrated form $\sm$ of $(u, \, \pi \circ \io)$
has a summand (the restriction to the range of $v$)
which is unitarily equivalent to the regular representation
of $C^* (G, A, \af)$ obtained from the representation
$\pi \circ \io$ of~$A$.
The kernel $I$ of this representation thus
satisfies $\ker (\sm) \S I$,
and $I$ is equal to the kernel
of the canonical \hm{}
$C^* (G, A, \af) \to C^*_{\mathrm{r}} (G, A, \af)$.
\end{proof}

The following result is easy and well known,
but we have been unable to find a statement in the literature.

\begin{lem}\label{L_2X25_CPDlim}
Let $G$ be a locally compact group,
and let $(A_{\ld})_{\ld \in \Ld}$
be a direct system of \ca{s}
with actions $\af^{\ld} \colon G \to \Aut (A_{\ld})$
and in which the maps are injective and equivariant.
Set $A = \dirlim_{\ld} A_{\ld}$,
with the direct limit action $\af \colon G \to \Aut (A)$.
Then there is a canonical isomorphism
$\dirlim_{\ld} C^*_{\mathrm{r}} (G, A_{\ld}, \af^{\ld})
   \cong C^*_{\mathrm{r}} (G, A, \af)$.
\end{lem}

\begin{proof}
We may assume that $A_{\ld} \subset A$ for all $\ld \in \Ld$,
and that the maps of the system are the inclusions.
Choose a Hilbert space~$H$ and a faithful representation
$\pi \colon A \to L (H)$.
Let $\sm \colon C^*_{\mathrm{r}} (G, A, \af) \to L (L^2 (G, H))$
be the representation of the reduced crossed product obtained
from the regular representation of $(G, A, \af)$
associated with~$\pi$.
Then $\sm$ is faithful.
Moreover, for $\ld \in \Ld$,
the representation $\pi |_{A_{\ld}}$ is faithful,
and the corresponding regular representation
of $C^*_{\mathrm{r}} (G, A_{\ld}, \af^{\ld})$ therefore
identifies this algebra
with a subalgebra of $C^*_{\mathrm{r}} (G, A, \af)$.
It is clear that, with these identifications,
we have
${\overline{\bigcup_{\ld \in \Ld}
       C^*_{\mathrm{r}} (G, A_{\ld}, \af^{\ld}) }}
  = C^*_{\mathrm{r}} (G, A, \af)$.
\end{proof}

Recall from the introduction to~\cite{KW}
that an action $\af \colon G \to \Aut (A)$
is {\emph{exact}} if for every $G$-invariant ideal $J \S A$,
with induced actions ${\overline{\af}}$ on $A / J$
and $\af_{( \cdot )} |_J$ on $J$,
the kernel of the map
$C^*_{\mathrm{r}} (G, A, \af)
   \to C^*_{\mathrm{r}} (G, A / J, {\overline{\af}})$
is exactly $C^*_{\mathrm{r}} (G, J, \af_{( \cdot )} |_J)$.

\begin{lem}\label{C:IContr}
Let $\af \colon G \to \Aut (A)$ be an action of a
second countable infinite locally compact group~$G$ on a \ca~$A$.
Set $B = {\mathcal{E}} (L^2 (G))$,
and let $\bt \colon G \to \Aut (B)$ be the action of
Lemma~\ref{L:GOnFc} coming from the left regular representation
$z$ of $G$ on $L^2 (G)$.
Set $D = \bigotimes_{k = 1}^{\I} B$,
with the action $\ld \colon G \to \Aut (D)$ given by
$\ld_g = \bigotimes_{k = 1}^{\I} \bt_g$.
Assume that $\ld \otimes \af \colon G \to \Aut (D \otimes A)$
is an exact action.
Set $F = C^*_{\mathrm{r}} (G, \, D \otimes A, \, \ld \otimes \af)$.
Let $J \S F$ be an ideal such that,
as subalgebras of $M (F)$
and using the notation of Definition~\ref{D:MAI},
we have
$(D \otimes A) \cap M (F, J) = 0$.
Then $J = 0$.
\end{lem}

\begin{proof}
For $n \in \N$, define the $n$th stage algebras $D_n$ and $F_n$,
and action $\ld^{(n)} \colon G \to \Aut (D_n)$,
by
\[
D_n = \bigotimes_{k = 1}^{n} B,
\,\,\,\,\,\,
\ld_g^{(n)} = \bigotimes_{k = 1}^{n} \bt_g,
\andeqn
F_n = C^*_{\mathrm{r}}
         \big( G, \, D_n \otimes A, \, \ld^{(n)} \otimes \af \big).
\]
Then $D = \dirlim_n D_n$.
Let $\io_n \colon D_n \otimes A \to D_{n + 1} \otimes A$
be the map coming from the inclusion of $D_n$ in $D_{n + 1}$.
Since the algebras $D_n$ are nuclear
(by Lemma~\ref{L:TCOfL2}),
the maps $\io_n$ are injective and
$D \otimes A = \dirlim_n D_n \otimes A$,
with the direct limit being taken with respect to these maps.
Moreover, $F = \dirlim_n F_n$ by Lemma~\ref{L_2X25_CPDlim}.
Let
\[
\kp_n \colon
 C^*  \big( G, \, D_n \otimes A, \, \ld^{(n)} \otimes \af \big) \to F_n
\andeqn
\kp \colon C^* (G, \, D \otimes A, \, \ld \otimes \af) \to F
\]
be the comparison maps from the full to the reduced \cp s.

Let $\sm \colon F \to L (H)$
be a nondegenerate representation of $F$ whose kernel
is exactly~$J$.
Then $\sm$ is the integrated form of some covariant representation
$(u, \pi)$ of $(G, \, D \otimes A, \, \ld \otimes \af)$.
Use the same letter $\sm$ for the corresponding
map $M (F) \to L (H)$.
The kernel of this extension is exactly the kernel
of the map $M (F) \to M (F / J)$, which is $M (F, J)$.
Observe that $\pi$ is injective,
since it is the restriction of $\sm$ to $D \otimes A \S M (F)$,
and $(D \otimes A) \cap M (F, J) = 0$.
Furthermore, $\pi$ is nondegenerate.

Let
\[
\rh_n \colon
 C^* \big( G, \, D_n \otimes A, \, \ld^{(n)} \otimes \af \big) \to L (H)
\]
be the integrated form of the covariant representation
$(u, \, \pi |_{D_n \otimes A} )$.
Then $\rh_n$ and $( \sm |_{F_n} ) \circ \kp_n$
come from the same covariant representation
$(u, \, \pi |_{D_n \otimes A} )$
of $\big( G, \, D_n \otimes A, \, \ld^{(n)} \otimes \af \big)$,
so $\rh_n = ( \sm |_{F_n} ) \circ \kp_n$.
Thus $\ker ( \kp_n) \S \ker (\rh_n)$.
Since
$\pi |_{D_n \otimes A}
   = \big( \pi |_{D_{n + 1} \otimes A} \big) \circ \io_n$,
Lemma~\ref{L:KerComp} implies that
in fact $\ker ( \rh_n) = \ker (\kp_n)$,
whence $\ker (\sm |_{F_n} ) = 0$.
Now
\[
J = \ker (\sm)
  = {\overline{\bigcup_{n = 1}^{\I} (\ker (\sm) \cap F_n) }}
  = {\overline{\bigcup_{n = 1}^{\I} \ker (\sm |_{F_n} ) }}
  = 0.
\]
This completes the proof.
\end{proof}

\begin{thm}\label{C:ILatt}
Let $\af \colon G \to \Aut (A)$ be an action of a
second countable infinite exact locally compact group~$G$ on a \ca~$A$.
Set $B = {\mathcal{E}} (L^2 (G))$,
and let $\bt \colon G \to \Aut (B)$ be the action of
Lemma~\ref{L:GOnFc} coming from the left regular representation
$z$ of $G$ on $L^2 (G)$.
Set $D = \bigotimes_{k = 1}^{\I} B$,
with the action $\ld \colon G \to \Aut (D)$ given by
$\ld_g = \bigotimes_{k = 1}^{\I} \bt_g$.
Then every ideal in
$C^*_{\mathrm{r}} (G, \, D \otimes A, \, \ld \otimes \af)$
is the reduced crossed product by an invariant ideal of $D \otimes A$.
\end{thm}

\begin{proof}
Set $F = C^*_{\mathrm{r}} (G, \, D \otimes A, \, \ld \otimes \af)$.
Let $J \S F$ be an ideal.
Set $L = ( D \otimes A ) \cap M (F, J)$,
which is an ideal in $D \otimes A$.

We claim that $L$ is $G$-invariant.
For $g \in G$, let $u_g$ be the standard unitary in the \cp,
which is in $M (F)$.
We start by showing that $u_g M (F, J) u_g^* \S M (F, J)$.
So let $a \in M (F, J)$ and let $b \in F$.
Then $\big( u_g a u_g^* \big) b = u_g a \big( u_g^* b \big)$.
We have $u_g^* b \in F$ since $F$ is an ideal in $M (F)$,
so that $a \in M (F, J)$ implies $a \big( u_g^* b \big) \in J$.
Since $J$ is an ideal in $M (F)$,
it follows that $u_g a \big( u_g^* b \big) \in J$.
This shows that $u_g a u_g^* \in M (F, J)$.
So $(\ld \otimes \af)_g (L) = u_g L u_g^* \S L$.
This proves the claim.

The algebra $D$ is simple and nuclear by Lemma~\ref{L:TCOfL2},
so there is an ideal $L_0 \S A$ such that $L = D \otimes L_0$.
Clearly $L_0$ is $G$-invariant.
Moreover,
$C^*_{\mathrm{r}}
 \big( G, \, L, \, (\ld \otimes \af)_{(\cdot)} |_{L})$
is an ideal in $F$ which is contained in~$J$.

Let ${\overline{\af}}$ be the action of $G$ on $A / L_0$.
Set
$E = C^*_{\mathrm{r}} \big( G, \, D \otimes (A / L_0),
      \, \ld \otimes {\overline{\af}} \big)$.
Since $D$ is nuclear and $G$ is exact, there is a short exact sequence
\[
0 \longrightarrow C^*_{\mathrm{r}}
 \big( G, \, L, \, (\ld \otimes \af)_{(\cdot)} |_{L})
  \longrightarrow F
  \longrightarrow E
  \longrightarrow 0.
\]
Since $F \to E$ is surjective,
we also have a surjective extension $\pi \colon M (F) \to M (E)$.

Let $I$ be the image of $J$ in~$E$.
We claim that $\pi^{-1} (M (E, I)) \S M (F, J)$.
To see this,
let $x \in M (F)$
with $\pi (x) \in M (E, I)$.
Then $\pi (x F) = \pi (x) E \S I$
by the definition of $M (E, I)$.
So $x F \S \pi^{-1} (I)$.
Also $x F \S F$, so $x F \S \pi^{-1} (I) \cap F = J$.
This proves the claim.

It now follows that
$M (E, I) \cap \big( D \otimes (A / L_0) \big) = 0$.
Lemma~\ref{C:IContr} now implies that $I = 0$,
whence
$J = C^*_{\mathrm{r}}
 \big( G, \, L, \, (\ld \otimes \af)_{(\cdot)} |_{L})$.
\end{proof}

We can now arrange for crossed products by actions like those
in Theorem~\ref{T:ExistMinAct},
Proposition~\ref{P:ZUnital},
and Proposition~\ref{P-RUnital}
to be simple.
We thus obtain the following results.

\begin{thm}\label{T-CPO2-430}
Let $G$ be a noncompact second countable amenable
locally compact group,
and let $P$ be a nonempty second countable
locally compact Hausdorff $G$-space with no nonempty compact
$G$-invariant subsets.
Then there exists a separable nuclear nonsimple prime \ca~$C$
such that $K \otimes \OT \otimes C \cong C$
and a minimal action $\gm \colon G \to \Aut (C)$
such that $\Prim (C)$ is equivariantly homeomorphic to $\Xi (P)$
and such that $C^* (G, C, \gm) \cong K \otimes \OT$.
\end{thm}

\begin{thm}\label{T-CPAutO2-430}
Let $P$ be a nonempty totally disconnected
noncompact second countable locally compact Hausdorff space,
let $\mu$ be a Borel measure on~$P$
such that $\mu (U) > 0$ for every nonempty open set $U \S P$,
and let $h \colon P \to P$
be a measure preserving homeomorphism.
Assume that there is no nonempty compact subset $L \S P$
such that $h (L) = L$.
Then there exists a separable nuclear unital nonsimple prime \ca~$C$
such that $\OT \otimes C \cong C$
and a minimal action $\gm \colon \Z \to \Aut (C)$
such that $\Prim (C)$ is equivariantly homeomorphic to $\Xi (P)$
with the action coming from $\Xi (h)$
as in Proposition~\ref{P:IndActXi}
and such that $C^* (\Z, C, \gm) \cong \OT$.
\end{thm}

\begin{thm}\label{T-CPbyR-430}
Let $P$ be a nonempty second countable
locally compact Hausdorff space with a continuous action of~$\R$
which has no nonempty compact $\R$-invariant subsets.
Then there exists a separable nuclear unital nonsimple prime \ca~$C$
such that $\OT \otimes C \cong C$
and a minimal action $\gm \colon \R \to \Aut (C)$
such that $\Prim (C)$ is equivariantly homeomorphic to $\Xi (P)$
and such that $C^* (\R, C, \gm) \cong K \otimes \OT$.
\end{thm}

\begin{proof}[Proof of Theorem~\ref{T-CPO2-430}]
Apply Theorem~\ref{T:ExistAct} to $G$ and~$P$,
obtaining a separable nuclear \ca~$C_0$
and an action $\gm^{(0)} \colon G \to \Aut (C_0)$.
Let $D$ and $\ld \colon G \to \Aut (D)$ be as in Theorem~\ref{C:ILatt}.
Let $\io \colon G \to \Aut (K \otimes \OT)$ be the trivial action.
Take $C = K \otimes \OT \otimes D \otimes C_0$,
and take $\gm \colon G \to \Aut (C)$
to be given by $\gm_g = \io_g \otimes \ld_g \otimes \gm^{(0)}_g$
for $g \in G$.
Clearly $\Prim (C)$ is equivariantly homeomorphic to $\Xi (P)$
and $C$ is separable, nuclear, nonsimple, prime,
and satisfies $K \otimes \OT \otimes C \cong C$.
Also $\gm$ is minimal.

Set $E = C^* (G, \, D \otimes C_0, \, \ld \otimes \gm^{(0)} )$.
Then Theorem~\ref{C:ILatt} implies that $E$ is simple.
Clearly $E$ is separable,
and $E$ is nuclear because $G$ is amenable.
Moreover,
$C^* (G, C, \gm) \cong K \otimes \OT \otimes E$.
It follows from the uniqueness statement
in Theorem~\ref{T:ExistA}
that $K \otimes \OT \otimes E \cong K \otimes \OT$.
\end{proof}

\begin{proof}[Proof of Theorem~\ref{T-CPAutO2-430}]
Let $\af \colon \Z \to \Aut (A)$
be the action of Corollary~\ref{C-FnAct508} (with $n = 1$).
Let $D$ and $\ld \colon G \to \Aut (D)$ be as in Theorem~\ref{C:ILatt}.
Let $\io \colon G \to \Aut (\OT)$ be the trivial action.
Take $C = \OT \otimes D \otimes A$,
and take $\gm = \io \otimes \ld \otimes \af \colon G \to \Aut (C)$.
The rest of the proof is the same as for Theorem~\ref{T-CPO2-430},
except that $C^* (G, C, \gm)$ is now unital
and isomorphic to the tensor product of $\OT$ and a
simple separable unital nuclear \ca.
Therefore $C^* (G, C, \gm) \cong \OT$
by Theorem~3.8 of~\cite{KP1}.
\end{proof}

\begin{proof}[Proof of Theorem~\ref{T-CPbyR-430}]
The proof is essentially the same as that of Theorem~\ref{T-CPO2-430}.
We apply Theorem~\ref{T:ExistAct} to $G$ and~$P$,
obtaining a separable nuclear \ca~$C_0$
and an action $\gm^{(0)} \colon G \to \Aut (C_0)$.
Arguing as in the proof of Proposition~\ref{P-RUnital},
instead of $K \otimes \OT \otimes C_0 \cong C_0$,
we may require that $C_0$ be unital and satisfy
$\OT \otimes C_0 \cong C_0$.
Following the proof of Theorem~\ref{T-CPO2-430},
tensor with $\OT \otimes D$ instead of with
$K \otimes \OT \otimes D$.
The resulting algebra~$C$ is unital.
Let $\gm \colon \R \to \Aut (C)$ be the resulting action.
The proof of Theorem~\ref{T-CPO2-430}
implies that $K \otimes C^* (\R, C, \gm) \cong K \otimes \OT$.
Since $C^* (\R, C, \gm)$ is nonunital,
it follows that $C^* (\R, C, \gm) \cong K \otimes \OT$.
\end{proof}

\begin{cor}\label{C:ExoticCptAb}
Let $G$ be a second countable abelian locally compact group
which is not discrete.
Let $P$ be a nonempty second countable
locally compact Hausdorff space,
and let ${\widehat{G}}$ act \ct ly on~$P$.
Assume that there are no nonempty compact
${\widehat{G}}$-invariant subsets of~$P$.
Then there is an action $\af \colon G \to \Aut (K \otimes \OT)$
such that $C^* (G, \, K \otimes \OT, \, \af)$
is a nonsimple prime \ca{}
with $\Prim ( C^* (G, \, K \otimes \OT, \, \af) ) \cong \Xi (P)$.
\end{cor}

There are many choices for~$P$.
See Example~\ref{E-111109dd}.

\begin{proof}[Proof of Corollary~\ref{C:ExoticCptAb}]
Apply Theorem~\ref{T-CPO2-430} with ${\widehat{G}}$
in place of~$G$,
obtaining an action $\gm \colon {\widehat{G}} \to \Aut (C)$.
Identify $G$ with its second dual,
and let $\af = {\widehat{\gm}}$.
The conclusion follows from Takai duality.
\end{proof}

Related examples can be produced using~\cite{Brt}.
The algebras are unital for more groups,
but there is much less control over the primitive ideal spaces
of the crossed products.

\begin{thm}[Bratteli]\label{T_2X31-Brt}
Let $G$ be a second countable locally compact abelian group
which is the product of a nontrivial connected group
and some other group.
Let $A = \bigotimes_{n = 1}^{\infty} M_2$
be the $2^{\infty}$~UHF algebra.
Then there exists an action $\af \colon G \to \Aut (A)$
such that $C^* (G, A, \af)$ is prime but not simple.
\end{thm}

\begin{proof}
Choose $\af$ as in Corollary~2.4 of~\cite{Brt}.
Then $C^* (G, A, \af)$ is prime by Corollary~2.3 of~\cite{Brt},
and not simple by Theorem~3.1 of~\cite{Brt}.
\end{proof}

It is not necessary
to require that $G$ be a product with one of the factors
connected.
An example is given in Proposition~3.4 of~\cite{Brt}
in which $G$ is compact and totally disconnected.

In some cases, $\Prim (C^* (G, A, \af))$ can be topologically
more complicated than the spaces $\Xi (P)$ considered here.
See, for example, Theorem~3.2 of~\cite{Brt}.
However, it seems difficult to specify any particular choices
of $\Prim (C^* (G, A, \af))$ in advance.

\begin{rmk}\label{R_2X27No}
There are no examples
as in Corollary~\ref{C:ExoticCptAb} or Theorem~\ref{T_2X31-Brt}
when $G$ is discrete.
One could apply Corollary~\ref{C:CptGMinA} to the dual action.
Another way to see this
is via the Connes spectrum $\Gm (\af)$
and the strong Connes spectrum ${\widetilde{\Gm}} (\af)$.
Since $A$ is simple,
$C^* (G, A, \af)$ is simple
\ifo{} ${\widetilde{\Gm}} (\af) = {\widehat{G}}$
by Theorem~3.5 of~\cite{Ks1},
and $C^* (G, A, \af)$ is prime
\ifo{} $\Gm (\af) = {\widehat{G}}$
by Theorem~5.8 of~\cite{OPd1}.
But ${\widetilde{\Gm}} (\af) = \Gm (\af)$
by Proposition~3.8 of~\cite{Ks1}.
\end{rmk}

When $G = \R$ or $G = S^1$,
we can find actions on~$\OT$
instead of on $K \otimes \OT$.
In both cases, if we don't care what the
primitive ideal space of the crossed product is,
we can get an action on the $2^{\infty}$~UHF algebra
from Theorem~\ref{T_2X31-Brt}.

We consider~$\R$ first.

\begin{prp}\label{P-RExO2502}
Let $P$ be a nonempty second countable locally compact Hausdorff space,
and let $\R$ act \ct ly on~$P$.
Assume that there are no nonempty compact
$\R$-invariant subsets of~$P$.
Then there is an action $\bt \colon G \to \Aut (\OT)$
such that
such that $C^* (\R, \OT, \bt)$
is a nonsimple prime \ca{}
with $\Prim ( C^* (\R, \OT, \bt) ) \cong \Xi (P)$.
\end{prp}

\begin{proof}
Corollary~\ref{C:ExoticCptAb}
provides an action $\af \colon \R \to \Aut (K \otimes \OT)$
such that $C^* (\R, \, K \otimes \OT, \, \af)$
is a nonsimple prime \ca{}
with
$\Prim ( C^* \big( \R, \, K \otimes \OT, \, \af) \big) \cong \Xi (P)$.
Using Lemma~\ref{L-RMkInv502},
we can find an exterior equivalent action
$\bt \colon \R \to \Aut (K \otimes \OT)$
and a \nzp{} $p \in K \otimes \OT$
such that $\bt_t (p) = p$ for all $t \in \R$.
Set $B = p (K \otimes \OT) p$.
Then $B \cong \OT$
and is a $\bt$-invariant hereditary subalgebra of $K \otimes \OT$.
We use the same letter $\bt$ for the restricted action
$\R \to \Aut (B)$.

We claim that $C^* (\R, B, \bt)$ is a full hereditary subalgebra
of $C^* (\R, \, K \otimes \OT, \, \bt)$.
Identify $K \otimes \OT$ with its image in
$M \big( C^* (\R, \, K \otimes \OT, \, \bt) \big)$.
Then one can check that
\[
C^* (\R, B, \bt) = p C^* (\R, \, K \otimes \OT, \, \bt) p
\]
by considering the dense convolution subalgebra
$C_{\mathrm{c}} (\R, \, K \otimes \OT, \, \bt)$
of $C^* (\R, \, K \otimes \OT, \, \bt)$.
Therefore $C^* (\R, B, \bt)$ is a hereditary subalgebra
of $C^* (\R, \, K \otimes \OT, \, \bt)$.
Let $I$ be the ideal it generates.
Then $I$ is ${\widehat{\bt}}$-invariant,
and therefore has the form $C^* (\R, J, \bt)$
for some $\bt$-invariant ideal $J \S K \otimes \OT$.
Since $K \otimes \OT$ is simple and $I \neq 0$,
we must have $I = C^* (\R, \, K \otimes \OT, \, \bt)$.
This completes the proof of the claim.

It follows that $C^* (\R, B, \bt)$ is stably isomorphic
to $C^* (\R, \, K \otimes \OT, \, \bt)$.
We have
\[
C^* (\R, \, K \otimes \OT, \, \bt)
  \cong C^* (\R, \, K \otimes \OT, \, \af)
\]
by the proof of part~(5) of Theorem 2.8.3 in~\cite{Ph1}.
Therefore
\[
\Prim (C^* (\R, B, \bt))
 \cong \Prim ( C^* \big( \R, \, K \otimes \OT, \, \bt) \big)
 \cong \Xi (P).
\]
This completes the proof.
\end{proof}

Now we consider~$S^1$.
We describe the kinds of spaces we use.

\begin{rmk}\label{R_2Y01-WhichP}
We will consider nonempty totally disconnected
noncompact second countable locally compact Hausdorff spaces~$P$
such that there exist a Borel measure $\mu$ on~$P$
and a measure preserving homeomorphism $h \colon P \to P$
such that $\mu (U) > 0$ for every nonempty open set $U \S P$
and such that there is no nonempty compact subset $L \S P$
with $h (L) = L$.
Examples of such spaces include $\Z$
(using counting measure and translation),
the product of $\Z$ with the Cantor set,
and the product $\Z \times Y$ for any countable closed
subset $Y \S \R$ which admits a Borel measure $\mu_0$ with full support.
For example, one could take
$Y = \big\{ 1, \tfrac{1}{2}, \tfrac{1}{4}, \ldots \big\} \cup \{ 0 \}$
with $\mu_0 ( \{ x \} ) = x$ for all $x \in Y$.
\end{rmk}

\begin{cor}\label{C-O2byS1-430}
Let $P$ be as in Remark~\ref{R_2Y01-WhichP}.
There is an action $\af \colon S^1 \to \Aut (\OT)$
such that $C^* (S^1, \OT, \af)$ is a nonsimple prime \ca{} with
$\Prim ( C^* (S^1, \OT, \af) ) \cong \Xi (P)$.
\end{cor}

\begin{proof}
The proof is the same as for Corollary~\ref{C:ExoticCptAb},
using Theorem~\ref{T-CPAutO2-430} in place of Theorem~\ref{T-CPO2-430}.
\end{proof}

We also have a stably finite version,
although we have not identified the algebra on which the action
takes place.
We need a different argument for simplicity.

\begin{lem}\label{L:TopFree}
Let $G$ be a group,
and let $P$ be a nonempty noncompact locally compact $G$-space
which is topologically free
in the sense that for every $g \in G \SM \{ 1 \}$,
the set $\{ x \in P \colon g x = x \}$ has empty interior.
Let $\af \colon G \to \Aut (A)$ be an action of $G$ on a \ca~$A$
such that $\Prim (A)$ is equivariantly homeomorphic to $\Xi (P)$.
Then the action of $G$ on~$A$ is topologically free in the sense
of Definition~1 of~\cite{AS}.
\end{lem}

\begin{proof}
The definition of~\cite{AS} is in terms of the action of $G$
on the set of unitary equivalence classes of irreducible
representations of~$A$,
but it clearly suffices to prove the condition
of Definition~1 of~\cite{AS} for
the action of $G$ on $\Prim (A)$ instead.
So,
following~\cite{AS},
let $g_1, g_2, \ldots, g_n \in  G \SM \{ 1 \}$.
Then
\[
\bigcap_{j = 1}^n \{ x \in \Xi (P) \colon g_j x \neq x \}
= \bigcap_{j = 1}^n \{ x \in P \colon g_j x \neq x \}
\]
is a finite intersection of dense open sets in~$P$,
hence dense in~$P$.
Since $P$ is not compact,
it follows that this set is dense in $\Xi (P)$ as well.
\end{proof}

\begin{prp}\label{P-StFbyS1-430}
Let $P$ be as in Remark~\ref{R_2Y01-WhichP}.
Then there is
a simple separable stably finite nuclear unital \ca~$B$
and an action $\bt \colon S^1 \to \Aut (B)$
such that $C^* (S^1, B, \bt)$ is a nonsimple prime \ca{} with
$\Prim ( C^* (S^1, B, \bt) ) \cong \Xi (P)$.
\end{prp}

\begin{proof}
Let $\af \colon \Z \to \Aut (A)$
be the action of Corollary~\ref{C-ZAF428}.
Take $B = C^* (\Z, A, \af)$,
and take $\bt = {\widehat{\af}}$.
The action of $\Z$ on $\Xi (\Z)$ is minimal by Lemma~\ref{L:MinXiP},
so the corollary to Theorem~1 of~\cite{AS}
and Lemma~\ref{L:TopFree}
imply that $B$ is simple.
Corollary~\ref{C:AFG}
provides an $\af$-invariant \tst{} on~$A$.
Therefore $B$ has a \tst,
so is stably finite.
The conclusion now follows from Takai duality.
\end{proof}

We end with two examples suggested by our methods
and two related questions.

\begin{exa}\label{E:FreeNotMin}
A topologically free
(Definition~1 of~\cite{AS})
action on a nonsimple prime \ca{} need not be minimal.

Apply Theorem~\ref{T:ExistAct}
with $G = \Z$,
taking $P$ to be the disjoint union of $\Z$ and the Cantor set~$K$,
and with $\Z$ acting on~$\Z$ by translation and on $K$
via a \mh.
Let $C$ be the resulting \ca.
Since $K$ is an invariant compact set,
Lemma~\ref{L:MinXiP} implies that the action of $\Z$ on
$\Xi (P)$ is not minimal,
so the action on $C$ is not minimal.
However, the action of $\Z$ on $P$ is free, so Lemma~\ref{L:TopFree}
above implies that the action on~$C$ is topologically free.\end{exa}

\begin{exa}\label{E:MinNotFree}
An effective (Definition~\ref{D-Effective}) minimal action
on a unital nonsimple prime \ca{}
need not be topologically free
(Definition~1 of~\cite{AS}).

Let $\gm \colon \Z \to \Aut (C)$
be as in Proposition~\ref{P:ZUnital}.
Let $\bt \colon \Z_2 \to \Aut (M_2)$ send the nontrivial group element
$1 \in \Z_2$ to
$\Ad \left( \left( \begin{smallmatrix} 0 & 1 \\
                       1 & 0 \end{smallmatrix} \right) \right)$.
Define $\af \colon \Z_2 \times \Z \to \Aut (M_2 \otimes C)$
by $\af_{g, n} = \bt_g \otimes \gm_n$
for $g \in \Z_2$ and $n \in \Z$.
It is immediate that $\af_{(1, 0)}$ fixes
all the unitary equivalence classes of irreducible representations.
\end{exa}

The action in Example~\ref{E:MinNotFree} is not effective
on the primitive ideal space.

\begin{qst}\label{Q-EffPrim430}
Does there exist a second countable locally compact group~$G$
and a minimal action of~$G$
on a nonprime simple \ca~$A$ which is effective on $\Prim (A)$
but not topologically free?
\end{qst}

\begin{qst}\label{Q:MinNotFreeForZ}
Is a minimal effective action of~$\Z$
on a unital nonsimple prime \ca{}
necessarily topologically free?
\end{qst}

For comparison, recall that a \mh{} of an infinite
Hausdorff space necessarily generates a free action.


\begin{thebibliography}{33}

\bibitem{APT}
C.~A.\  Akemann, G.~K.\  Pedersen, and J.~Tomiyama,
{\emph{Multipliers of C*-algebras}},
J.~Funct.\  Anal.\  {\textbf{13}}(1973), 277--301.

\bibitem{AS}
R.~J.\  Archbold and J.~S.\  Spielberg,
{\emph{Topologically free actions and ideals in discrete C*-dynamical
systems}},
Proc.\  Edinburgh Math.\  Soc.~(2) {\textbf{37}}(1994), 119--124.

\bibitem{Ar}
W.~Arveson,
{\emph{Notes on extensions of C*-algebras}},
Duke Math.~J.\  {\textbf{44}}(1977), 329--355.

\bibitem{BC} B.~Blackadar and J.~Cuntz,
{\emph{The structure of stable algebraically simple C*-algebras}},
Amer.\  J.\  Math.\  {\textbf{104}}(1982), 813--822.

\bibitem{Brt} O.~Bratteli,
{\emph{Crossed products of UHF~algebras by product type actions}},
Duke Math.~J.\  {\textbf{46}}(1979), 1--23.

\bibitem{BE}
O.~Bratteli and G.~A.\  Elliott,
{\emph{Structure spaces of approximately finite-dimensional
C*-algebras,~II}},
J.~Funct.\  Anal.\  {\textbf{30}}(1978), 74--82.

\bibitem{Cn} A.~Connes,
{\emph{An analogue of the Thom isomorphism for crossed products of
a C*-algebra by an action of ${\mathbb{R}}$}},
Advances in Math.\  {\textbf{39}}(1981), 31--55.

\bibitem{Cu2} J.~Cuntz, {\emph{K-theory for certain
C*-algebras}}, Ann.\  Math.\  {\textbf{113}}(1981), 181--197.

\bibitem{Dn} A.~I.\  Danilenko,
{\emph{Strong orbit equivalence of locally compact Cantor minimal
   systems}},
International J.\  Math.\  {\textbf{12}}(2001), 113--123.

\bibitem{Dv} K.~R.\  Davidson,
{\emph{C*-Algebras by Example}},
Fields Institute Monographs no.~6,
Amer.\  Math.\  Soc., Providence RI, 1996.

\bibitem{Dx} J.~Dixmier, {\emph{C*-Algebras}}, North-Holland,
Amsterdam, New York, Oxford, 1977.

\bibitem{EHS} E.~G.\  Effros, D.~E.\  Handelman, and C.~L.\  Shen,
{\emph{Dimension groups and their affine representations}},
Amer.\  J.\  Math.\  {\textbf{102}}(1980),  385--407.

\bibitem{Ell1} G.~A.\  Elliott, {\emph{On the classification
of inductive limits of sequences of semisimple finite-dimensional
algebras}}, J.~Algebra {\textbf{38}}(1976), 29--44.

\bibitem{Gd}  K.~R.\  Goodearl,
{\emph{$K_0$ of multiplier algebras of C*-algebras
   with real rank zero}},
K-Theory {\textbf{10}}(1996), 419--489.

\bibitem{GL} E.~C.\  Gootman and A.~J.\  Lazar,
{\emph{Applications of noncommutative duality to crossed product
C*-algebras determined by an action or coaction}},
Proc.\  London Math.\  Soc.\  (3) {\textbf{59}}(1989), 593--624.

\bibitem{HK}
H.~Harnisch and E.~Kirchberg,
{\emph{The inverse problem for primitive ideal spaces}},
preprint 2005:
Preprintreihe
des SFB 478 -- Geometrische Strukturen in der Mathematik
des Mathematischen Instituts
der Westf\"{a}lischen Wilhelms-Universit\"{a}t M\"{u}nster,
Heft 399.

\bibitem{HfKm} K.~H.\  Hofmann and K.~Keimel,
{\emph{A general character theory for partially ordered
  sets and lattices}},
Mem.\  Amer.\  Math.\  Soc.\  no.~122(1972).

\bibitem{HL} K.~H.\  Hofmann and J.~D.\  Lawson,
{\emph{The spectral theory of distributive continuous lattices}},
Trans.\  Amer.\  Math.\  Soc.\  {\textbf{246}}(1978), 285--310.

\bibitem{HW} W.~Hurewicz and H.~Wallman,
{\emph{Dimension Theory}}, Princeton U.\  Press, Princeton, 1948.

\bibitem{It}
B.~A.\  Itz\'{a}-Ortiz,
{\emph{Continuous and discrete flows on operator algebras}},
J.~Australian Math.\  Soc.\  {\textbf{86}}(2009), 169--176.

\bibitem{Iz1} M.~Izumi,
{\emph{The Rohlin property for automorphisms of C*-algebras}},
pages 191--206 in:
{\emph{Mathematical Physics in Mathematics and Physics (Siena, 2000)}},
Fields Inst.\  Commun.\  vol.~30,
Amer.\  Math.\  Soc., Providence RI, 2001.

\bibitem{JK} K.~K.\  Jensen and K.~Thomsen,
{\emph{Elements of KK-Theory}},
Birkh\"{a}user, Boston, Basel, Berlin, 1991.

\bibitem{Kb} E.~Kirchberg,
{\emph{The classification of purely infinite C*-algebras
  using Kasparov's theory}},
in preparation.

\bibitem{KrC} E.~Kirchberg,
in preparation.

\bibitem{KP1} E.~Kirchberg and N.~C.\  Phillips,
{\emph{Embedding of exact C*-algebras in the Cuntz algebra
   ${\mathcal{O}}_2$}},
J.~reine angew.\  Math.\  {\textbf{525}}(2000), 17--53.

\bibitem{KR0} E.~Kirchberg and M.~R{\o}rdam,
{\emph{Non-simple purely infinite C*-algebras}},
Amer.\  J.\  Math.\  {\textbf{122}}(2000), 637--666.

\bibitem{KR} E.~Kirchberg and M.~R{\o}rdam,
{\emph{Purely infinite C*-algebras: Ideal-preserving zero homotopies}},
Geom.\  Funct.\  Anal.\    {\textbf{15}}(2005), 377--415.

\bibitem{KW} E.~Kirchberg and S.~Wassermann,
{\emph{Exact groups and continuous bundles of C*-algebras}},
Math.\  Annalen {\textbf{315}}(1999), 169--203.

\bibitem{Ks1} A.~Kishimoto,
{\emph{Simple crossed products of C*-algebras by locally compact
  abelian groups}},
Yokohama Math.~J.\  {\textbf{28}}(1980), 69--85.

\bibitem{OPd1} D.~Olesen and G.~K.\  Pedersen,
{\emph{Applications of the Connes spectrum to C*-dynamical systems}},
J.~Funct.\  Anal.\   {\textbf{30}}(1978), 179--197.

\bibitem{OPd3} D.~Olesen and G.~K.\  Pedersen,
{\emph{Applications of the Connes spectrum
 to C*-dynamical systems,~III}},
J.~Funct.\  Anal.\   {\textbf{45}}(1982), 357--390.

\bibitem{Ps} C.~Pasnicu,
{\emph{AH~algebras with the ideal property}},
pages 277--288 in:
{\emph{Operator Algebras and Operator Theory (Shanghai, 1997)}},
L.~Ge, H.~Lin, Z.-J.\  Ruan, D.~Zhang, and S.~Zhang (eds.),
Contemporary Mathematics vol.~228, Amer.\  Math.\  Soc.,
Providence RI, 1998.

\bibitem{Pd} G.~K.\  Pedersen,
{\emph{C*-Algebras and their Automorphism Groups}},
Academic Press, London, New York, San Francisco, 1979.

\bibitem{Ph1}
N.~C.\  Phillips,
{\emph{Equivariant K-Theory and Freeness
of Group Actions on C*-Algebras}}, Springer-Verlag Lecture Notes in
Math.\  no.~1274, Springer-Verlag, Berlin, Heidelberg, New York,
London, Paris, Tokyo, 1987.

\bibitem{Pm} M.~Pimsner,
{\emph{A class of C*-algebras generalizing both Cuntz-Krieger
algebras and crossed products by~${\mathbb{Z}}$}},
pages 189--212 in:
{\emph{Free Probability Theory (Waterloo, ON, 1995)}},
Fields Inst.\  Commun.\  vol.~12,
Amer.\  Math.\  Soc., Providence, RI, 1997.

\bibitem{PV} M.~Pimsner and D.~Voiculescu,
{\emph{Exact sequences
for K-groups and Ext-groups of certain cross-product C*-algebras}},
J.~Operator Theory {\textbf{4}}(1980), 93--118.

\bibitem{Sc2} C.~Schochet,
{\emph{Topological methods for C*-algebras II:
geometric resolutions and the K\"{u}nneth formula}},
Pacific J.\  Math.\  {\textbf{98}}(1982), 443--458.

\bibitem{Stv} K.~H.\   Stevens, {\emph{The classification of
certain non-simple approximate interval algebras}},
pages 105--148 in: {\emph{Operator Algebras and their Applications, II
(Waterloo, ON, 1994/1995)}},
Fields Inst.\  Commun.\  vol.~20,
Amer.\  Math.\  Soc., Providence, RI, 1998.

\end{thebibliography}
\end{document}